\documentclass{article}
\usepackage[all]{xy}
\usepackage{amsmath, amscd, amsthm}
\usepackage{longtable}
\usepackage{amscd}
\usepackage{mathrsfs}
\usepackage{graphicx}
\usepackage{latexsym}
\usepackage{empheq}
\usepackage{tikz-cd}
\usepackage{wasysym}
\usepackage[T2A]{fontenc}
\usepackage[utf8]{inputenc} 
\usepackage{titlesec}
\titleformat{\subsection}[runin]{\bfseries}{\thesubsection}{.5em}{}
\titleformat{\section}[block]{\bfseries\large}{\thesection}{.5em}{}
\usepackage{amsmath, amsthm, xcolor,amsfonts,amssymb,latexsym}
\theoremstyle{plain}\newtheorem{Thm}{Theorem}
\newtheorem{Prop}{Proposition}
\newtheorem{Proposition}[Prop]{Proposition}

\newtheorem{Cor}{Corollary}
\newtheorem{Lemma}{Lemma}
\theoremstyle{definition}

\newtheorem{example}{Example}
\newtheorem{Example}[example]{Example}
\newtheorem{Remark}{Remark}
\newtheorem{problem}{Problem}

\newcommand{\nez}{\ne \{0\}}
\newcommand{\ez}{= \{0\}}
\newcommand{\al}{\alpha}

\newcommand{\ov}{\overline}

\newcommand{\tA}{{\widetilde{A}}}

\newcommand{\tC}{{\widetilde{C}}}

\newcommand{\bpm}{\begin{pmatrix}}
\newcommand{\epm}{\end{pmatrix}}

\newcommand{\ld}{\ldots}

\newcommand{\ve}{\varepsilon}
\newcommand{\vp}{\varphi}

\newcommand{\beq}{\begin{equation}}
\newcommand{\eeq}{\end{equation}}
\newcommand{\beas}{\begin{eqnarray*}}
\newcommand{\eeas}{\end{eqnarray*}}

\newcommand{\F}{\mathbf{F}}

\newcommand{\ZZ}{\mathbb{Z}}

\newcommand{\la}{\mathbb{\langle }}
\newcommand{\ra}{\mathbb{ \rangle }}
\newcommand{\N}{{\mathbb N}}

\newcommand{\cA}{\mathcal{A}}
\newcommand{\cB}{{\mathcal B}}
\newcommand{\cC}{\mathcal{C}}

\newcommand{\cF}{\mathcal{F}}

\newcommand{\cT}{\mathcal{T}}

\newcommand{\cP}{\mathcal{P}}

\newcommand{\cS}{\mathcal{S}}

\newcommand{\X}{\mathcal{X}}
\newcommand{\cX}{\mathcal{X}}

\newcommand{\frM}{\mathfrak{M}}
\newcommand{\frN}{\mathfrak{N}}

\newcommand{\frV}{\mathfrak{V}}
\newcommand{\frU}{\mathfrak{U}}
\newcommand{\frW}{\mathfrak{W}}

\newcommand{\V}{\mathfrak{V}}

\newcommand{\id}{\mathrm{id}}

\newcommand{\qvar}{\mathrm{qvar}}
\newcommand{\cont}{\mathrm{cont}}

\newcommand{\var}{\mathrm{var\, }}

\DeclareMathOperator{\End}{\mathrm{End}}

\DeclareMathOperator{\soc}{\mathrm{soc}}

\DeclareMathOperator{\Ker}{\mathrm{Ker}\,}
\DeclareMathOperator{\rank}{\mathrm{rank}\,}
\DeclareMathOperator{\Aut}{\mathrm{Aut}}

\DeclareMathOperator{\Card}{\mathrm{Card}}

\newcommand{\GL}{\mathrm{GL}}

\newcommand{\wg}{\mathrm{wt\,}}

\begin{document}

\title{Locally finite varieties of nonassociative algebras}
\author{Yuri Bahturin\\
Department of Mathematics and Statistics\\Memorial University of Newfoundland\\St. John's, NL, A1C5S7
\and Alexander Olshanskii\\Department of Mathematics\\1326 Stevenson Center\\Vanderbilt University\\Nashville, TN 37240}
\date{}
\maketitle
\begin{abstract}
We study locally finite varieties (=primitive classes) of linear algebras over finite fields. We do not assume that our algebras are associative or Lie. We are interested in the basic properties of finite algebras in these varieties such as: nilpotence, solvability, simplicity, freeness, projectivity, and injectivity. We are also interested in the numerical estimates of the ratio of the number of algebras with various classical properties to the total number of all algebras of a fixed dimension $n$. Among these properties are having no proper nontrivial subalgebras or  no nontrivial automorphisms, etc.
\end{abstract}
\tableofcontents

\section{Introduction, definitions and elementary properties}\label{esAYO Intro}

\subsection{Topic of the article.}

Given a field $\F$, a (linear) \emph{algebra $A$ over a field $\F$} (or an \emph{$\F$-algebra}) is a vector space over $\F$ with one  bilinear operation $A\times A\to A$, written as $(a,b)\mapsto ab$, where $a,b,ab\in A$. Linear algebras over fields form an important, yet particular case of \emph{universal} algebras where the underlying set $A$ is not a vector space and the number of operations $\underbrace{A\times\cdots\times A}_n\to A$ of various \emph{arities} $n$ is unlimited. Some of the questions asked in the theory of linear algebras had been first asked in the theory of universal algebras which lies on the boundary line between Algebra and Logic.

In this paper, we study finite algebras, that is, finite-dimensional algebras over finite fields. The general interest in the identities of finite universal algebras, in particular, finite rings, groups, semigroups, etc. dates back to the papers devoted to the Finite Basis Problem (FBP), asking: Do all identities of a finite (universal) algebra $A$ follow from a finite set of them? The positive answers were obtained initially in the classes of finite groups, associative and Lie rings in the papers Oates - Powel \cite{OaPo}, Kruse \cite{Kruse}, Lvov \cite{Lvov}, \cite{Lvov2}, Bahturin - Olshanskii \cite{BO}. Later, the problem was solved in the affirmative for  alternative, Jordan, Malcev and certain even more general finite algebras (see Medvedev \cite{Medvedev2}).
  
On the other hand, even earlier, in 1955, Lyndon \cite{Lyn} had built an example of a universal algebra with one binary operation of order 6 without FBP. Then Mursky \cite{Mur} had suggested an example of an algebra of order 3 without FBP. Later, Polin \cite{Polin} produced  examples of finite-dimensional linear algebras over any field without FBP.

A natural way to get into this stream is to consider finite universal algebras, in particular, semigroups, latices, and other non-linear algebraic structures. The most general approach to the identities of finite universal algebras can be found in the book \cite{HRM}.

The subject of our research has intermediate generality. We do not assume that the algebras satisfy
any classical identities like the associativity or the Lie identities. However we do assume the
linearity of the multiplication. The reader can generalize some of the results we formulate in this
paper to the case of related structures such as rings and $\Omega$-algebras. To be more consistent, we limit ourselves to a single bilinear operation and focus on finitely generated and locally finite varieties of linear algebras.

In the next subsections we give some basic definitions and results about varieties of algebras. They can be found in many texts, starting with the founding papers of Birkhoff \cite{Birk}, B. H. Neumann \cite{BHN}, A. I. Malcev \cite{AIM} and such seminal textbooks as Malcev's \cite{ASM}, Kurosh's \cite{KUR}, Cohn's \cite{PMC}, Hanna Neumann's \cite{HN}. Also see the educational survey by Kharlampovich and Sapir  \cite{KHS}.

\subsection{Absolutely free algebras.}\label{ssAFA} For any field $\F$ and a nonempty set $X$, one defines an (absolutely) free algebra $\cF(X)$ as follows. We first define the \textit{nonassociative monomials} in $X$ by induction on their \textit{degree}. The monomials of degree 1 are the elements of $X$. We write $\deg x=1$ if $x\in X$. Once the monomials of all degrees less than $n>1$ have been defined, the monomials of degree $n$ are all possible expressions of the form $(u)(v)$ where $\deg u=k\ge 1$ and $\deg v=n-k\ge 1$. The set of all nonassociative monomials in $X$ is denoted by $\Gamma(X)$. One has $\Gamma(X)=\bigcup_{n=1}^\infty \Gamma_n$ where $\Gamma_n$ is the set of monomials of degree $n$. If we drop parentheses on a nonassociative monomial $v$ we obtain an associative word, called the \emph{support} of $v$. Conversely, one can set brackets on an associative word $w$ to obtain all possible nonassociative monomials with support $w$.

The set $\Gamma(X)$ is endowed with an operation ${\scriptstyle\circ}$, as follows. If $u\in \Gamma_k$ and $v\in \Gamma_\ell$ then $u{\scriptstyle\circ} v=(u)(v)\in\Gamma_{k+\ell}$. With respect to this operation, $\Gamma(X)$ is generated by $X$. In practice, while using this operation, one omits the symbol $\scriptstyle\circ$ and all unnecessary parentheses. So, one writes $(x_1x_2)x_2$ instead of  
$((x_1){\scriptstyle\circ}(x_2)){\scriptstyle\circ}(x_2)$.
 
Finally, one defines the free algebra $\cF(X)$ as the linear span of $\Gamma(X)$ with coefficients in $\F$. The elements of $\cF(X)$ are often called the \textit{nonassociative polynomials}. The product in $\cF(X)$ is the extension by linearity of the product in $\Gamma(X)$, and so $\cF(X)$ is a linear algebra called the \emph{free non-associative algebra with the set of free generators $X$}. The cardinality of $X$ is called the rank of $\cF(X)$, (In this paper nonassociative polynomials have no terms of zero degree!)

If $L$ and $M$ are subspaces of an algebra, then $LM$ denotes the linear span of all products
$lm$, where $l\in L$ and $m\in L$. The \emph{power} $A^k$ of any algebra $A$ is defined by induction as $A^1=A$ and $A^k = \sum_{s+t=k}A^sA^t$. We list the well-known properties:

\begin{Lemma}\label{power} Given an algebra $A$ over $\F$, for any $i > 1$, we have 
$A^i \subset A^{i-1}$,  $A^{i-1}A\subset A^i$ and $A^{i-1}A\subset A^i$.
Thus, $A^{i-1}$ is an ideal of $A$. 
\end{Lemma}

\proof
Clearly, $A^2 \subset A = A^1$. If $i > 2$ and $j < i$ then by induction, we have $A^j \subset A^{j-1}$. Hence, for $k, \ell < i$, $k + \ell \ge 3$, we obtain $A^kA^\ell \subset A^{k-1}A^\ell$ or $A^kA^\ell \subset A^kA^{l-1}$. Recalling the definition, we have  $A^i \subset A^{i-1}$. By the definion, and this inclusion, we have
$AA^{i-1} = A^1A^{i-1}\subset A^i\subset A^{i-1}$. Similarly, $A^{i-1}A\subset A^i$,
and so $A^{i-1}$ is an ideal of $A$.
\endproof

  If $\cF_n$ is the linear span of $\Gamma_n$ then $\cF_k\cF_{\ell}\subset \cF_{k+\ell}$ and so we have the canonical $\N$-grading of the free algebra, where $\N$ is the additive semigroup of natural numbers, as follows 
\begin{equation}\label{eGrading}
\cF(X)=\bigoplus_{n=1}^\infty\cF_n
\end{equation}

The elements of $\cF_n$ are called \textit{homogeneous} of degree $n$. The powers $\cF^n=\bigoplus_{k=n}^\infty \cF_k$  form a \emph{descending filtration} in $\cF(X)$: 
\begin{equation}\label{ePSF}
\cF(X)=\cF^1\supset \cF^2\supset\cdots\supset \cF^n\supset\cdots\mbox{ such that }\bigcap_{n=1}^\infty \cF^n=\{ 0\}.
\end{equation}

Given a nonassociative monomial $u(x_1,\ld,x_n)$ and the elements $a_1,\ld,a_n$ of  any algebra $P$ over any field $\F$, one can use the induction on $\deg u$  and the operation of $P$ to define the \textit{value} $u(a_1,\ld,a_n)\in P$.  The value of a monomial $x_i$ is just $a_i$, while $(uv)(a_1,\ld,a_n)=u(a_1,\ld,a_n)v(a_1,\ld,a_n)$. Now the value of a nonassociative polynomial $w(x_1,\ld,x_n)$ which is a linear combination of several monomials with coefficients in $\F$,  is just the linear combination of the values of these monomials with the same coefficients. 

If $\vp:X\to P$ is any mapping then the evaluation of nonassociative polynomials which replaces each $x\in X$ by $\vp(x)\in P$ is a unique homomorphism of algebras $\bar\vp:\cF(X)\to P$, extending $\vp$. This property is called the \textit{universal property} of $\cF(X)$. An immediate consequence of this property is the following. Given any algebra $P$,  there exists $X$ such that $P$ is a homomorphic image of $\cF(X)$ under a homomorphism $\ve$ mapping $X$ onto a generating set of $P$. Taking $K=\Ker\ve$, we get $P\cong \cF(X)/K$. 

\subsection{Varieties and  free algebras.}\label{ssPC}

If $w(x_1,\dots, w_n)$ is a nonassociative polynomial and $w(a_1,\ld,a_n)=0$, for all elements $a_1,\ld,a_n$ of an algebra $P$, we say that $w(x_1,\ld,x_n)=0$ is an \textit{identical relation} or shortly an \textit{identity} in $P$. 

Now let us consider the free algebra of countable rank $\cF_\infty=\cF(x_1,\ld,x_n,\ld)$ over $\F$. Given a nonempty subset $V\subset\cF_\infty$, the class $\V$ of all algebras (over fixed field $\F$) satisfying all identities $v(x_1,\ld,x_n)=0$, where $v(x_1,\ld,x_n)\in V$, is called the \textit{variety of algebras defined by $V$}. It follows from this definition that a variety is closed under subalgebras, factor-algebras and Cartesian products of algebras.

The \emph{trivial} variety consists of zero algebra only; it is defined by the identical relation $x=0$.

Given a variety $\V$ and algebra $P$, one considers the ideal $\V(P)$ (more often denoted by a matching notation $V(P)$) of all values $v(a_1,\ld,a_n)$, where $a_1,\ld,a_n\in P$ and $v(x_1,\ld,x_n)=0$ is an identity in $\V$. By analogy with the group-theoretic notation, one calls $V(P)$ \textit{the verbal ideal corresponding to the variety $\V$}. $V(P)$ is the least ideal of $P$ such that $P/V(P)\in\V$. So $V(P)=\{ 0\}$ iff $P\in \V$. The ideal $V(P)$ is closed under all endomorphisms of $P$. Such ideals are called \textit{fully invariant}.

If $\F$ is a fixed field, then there is one-one Galois-type correspondence between the varieties over $\F$ and the fully invariant ideals of $\cF_\infty$:
 
\begin{Proposition}\label{pGC}
Given a variety $\V$, the corresponding verbal ideal $V(\cF_\infty)$ is fully invariant. Conversely, every fully invariant ideal is a verbal ideal for an appropriate variety $\V$.$\hfill\Box$
\end{Proposition}  

If $U\subset\cF_\infty$ and $P$ is an algebra, one denotes by $U(P)$ the ideal of $P$ generated by all $u(a_1,\ld,a_n)$, where $a_1,\ld,a_n\in P$ and $u(x_1,\ld,x_n)\in U$.
 This is the \textit{verbal ideal of $P$ generated by the set of nonassociative polynomials $U$}. If $\frV(\cF_\infty )=U(\cF_\infty)$, then we say that $\V$ is defined by the identities $\{u=0\, |\,u\in U\}$. We also say that the set of identities $\{ u=0\,|\,u\in U\}$ is a \textit{basis of identities} of $\V$. 

Given a nonempty set $X$, the algebra $F(\V,X)=\cF(X)/V(\cF(X))$ is called a \textit{$\V$-free algebra with the free generating set $X$} or a \textit{ free algebra in $\V$ with the free generating set $X$}. If $\V$ is nonzero, $X$ is injectively embedded in $F(\V,X)$. Each of these algebras possesses the \textit{universal property within the variety $\V$}, namely, any map $\vp:X\to P\in\V$ extends to a unique homomorphism $\overline{\vp}:F(\V,X)\to P$. It follows that every algebra $P\in\V$ is isomorphic to a factor-algebra of an appropriate $F=F(\V,X)$. If we do not specify parameters $\V$ or $X$, then we simply say that $F$ is a \textit{relatively free algebra}.

For a $\V$-free algebra $F(\V,X)$, we also have a Galois-type correspondence between the subvarieties
of $\V$ and the verbal ideals of $F(\V,X)$. The join $\V_1\cup\V_2$ in the lattice of subvarieties of $\V$ is the smallest variety containing both subvarieties $\V_1$ and $\V_2$. The corresponding
verbal ideal of $F(\V,X)$ is the intersection $V_1(F(\V,X))\cap V_2(F(\V,X))$, and the intersection 
 $\V_1\cap\V_2$ corresponds the sum $V_1(F(\V,X))+ V_2(F(\V,X))$.

Examples of algebras, which are free in certain varieties, include free associative algebras, free Lie algebras, free commutative algebras, etc.

A useful property of relatively free algebras, following from the universal property, is given below.

\begin{Lemma}\label{lRelX} Let $0\ne v(x_1,\ld,x_n)\in\cF_\infty$ be a nonassociative polynomial and $F=F(\V,Y)$ a  free algebra in a variety $\V$ of algebras, with a free generating set $Y=\{ y_1,y_2,\ld\}$. Suppose $v(y_1,\ld,y_n)=0$ is a \emph{relation} among $y_1,\ld,y_n$ that holds in $F$. Then $v(x_1,\ld,x_n)=0$ is an \emph{identical relation} in $\V$. 
$\hfill\Box$
\end{Lemma}  

Next we mention the following.

Any map of a set $X$ to a set $Y$ extends to a unique homomorphism $F(\V,X)\to F(\V,Y)$. In particular, the embedding  $X\subset Y$ extends to the embedding  $F(\V,X)\subset F(\V,Y)$ and  $F(\V,X)$ is a retract of $F(\V,Y)$ under the retraction extending $x\mapsto x$ for $x\in X$ and $y\mapsto 0$ for $y\in Y\setminus X$. 

The cardinality of $X$ is called the \textit{rank} of $F(\V,X)$. Two relatively free algebras in a nontrivial variety $\V$ are isomorphic if and only if they have the same rank.
So, with $\V$ fixed, we can simply write $F_n=F_n(\V)$ for the relatively free algebra of rank $n$ in the variety $\V$. In this case, we can choose $X=\{ x_1,\ld,x_n\}$.

Given a set of algebras $\cS$ over a field $\F$, the \emph{variety generated by $\cS$}, denoted by $\var \cS$, is the set of all algebras over $\F$, satisfying all identities satisfied by all algebras from $\cS$. If $\cS=\{P\}$, where $P$ is an algebra, then we simply write $\var\cS=\var P$. Equivalently, $\var \cS$ is the intersection of all varieties $\frV$ such that $\cS\subset\V$.

Our aim is to provide new evidence of a close connection between the structure of a finite
algebra $A$, the identities of $A$, and the properties of algebras in the variety $\var A$.
A variety generated by a finite algebra is called \emph{finitely generated}. Studying finitely 
generated varieties is the main topic of this paper. It is quite apparent that if $A$ is
nilpotent or solvable, or simple then these properties must have a strong influence on the
structure of other algebras in $\var A$. In addition to finitely generated varieties, we will look at the varieties where all finite algebras are injective, or projective, or, more generally, being subalgebras of free algebras.

\subsection{Two one-dimensional algebras.}\label{ssFermat} Let $A$ be a one-dimensional algebra over a field $\bf F$. If $A$ has a non-zero multiplication,
then there is a basis vector $e\in A$ such that $e^2 =e$, and then the mapping $e\mapsto 1$ extends
to an isomorphism of $A$ onto the field $F$, which is an algebra over itself. 

As a result, up to isomorphism, for every field $\F$, we have just two one-dimensional $\F$-algebras.  In this subsection, we look at the identities excluding a one-dimensional algebra from a variety.

For the first statement, we consider the identities of the form $x = f(x)$, where the nonassociative polynomial $f(x)$ depends on one variable $x$ and contains only monomias of degree at least $2$. We will call such identities \emph{Fermat's identities} in honor of Fermat's Little Theorem.

\begin{Prop}\label{pFermat} A variety $\frV$ of   algebras over a field $\F$ satisfies some Fermat's identity 
if and only if, it does not contain a one-dimensional algebra with zero multiplication.
\end{Prop}
\begin{proof} No Fermat's identity holds in one-dimensional algebra $A$ with zero multiplication,
since the right-hand side of the identity takes only zero values in $A$. Hence $A$ does not
belong to a variety satisfying a  Fermat's identity.

Assume now that no Fermat's identity holds in a variety $\frV$. Then there
are no identities $f(x_1,\ld,x_n)=0$ where the left-hand side involves a monomial $x_i$ of degree $1$. Indeed, otherwise, one can replace other variables in $f$ with zero. By Lemma \ref{lRelX}, this mapping extends to the endomorphism on the free $\frV$-algebra and gives a Fermat's identity in $\frV$. It remains to mention that every polynomial without monomials of degree 1 vanishes in the
one-dimensional algebra $A$ with zero multiplication, whence  $A\in \V$.
\end{proof}

The degree $d_i$ of a monomial $v$ with respect to a variable $x_i$ can easily  be defined by induction as the number of occurrences of $x_i$ in the support of $v$. If $|\F|=q$, we call a polynomial $f$ \textit{quasihomogeneous} if for every $x_i$, all the monomial summands of $f$ have the same degree with respect to $x_i$, taken modulo $q-1$.

\begin{Prop}\label{pQuasi} A variety $\V$ of algebras over a finite field $\F$ does not contain the one-dimensional algebra $\F$ with nonzero multiplication if and only if, $\V$ satisfies some quasihomogeneous identity in one variable with nonzero sum of coefficients of its monomial summands.
\end{Prop}

\begin{proof}

Note that an identity with nonzero  coefficient sum of the monomials is not satisfied in algebra $\F$. Indeed, by substituting the identity element in place of all variables, we obtain a nonzero element of the field.

Let conversely, a variety $\frV$ have no quasihomogeneous identities in one variable with nonzero sum of coefficients, that is, all quasihomogeneous identities in one variable $x$ holding in $\frV$  have zero sum of coefficients. Note that $x^m = x^n$ is the identity of $\bf F$ if $m\equiv n \mod(q-1)$. Hence the quasihomogeneous  polynomials  with zero sum of coefficients are in the kernel of any homomorphism of the one-generated $\frV$-free  algebra $F_1$ onto  $\F$ since under any substitution of any element of $\F$ in place of $x$  all monomials will take equal values. It follows that these identities hold in $\F$. 

But then also all identities in one variable  are in the kernel of any such homomorphism, since  \cite[Section 8.4]{BOGP} the left side of any  identity is the sum of quasihomogeneous polynomials, which  also identically equal zero in $\frV$. Hence the mapping of the variable $x$ into a nonzero element of $\F$ has a well-defined extension to a homomorphism, that is, the algebra $\F$ is isomorphic to the factor-algebra of the algebra $F_1$. Now it follows that $\F\in\frV$.

\end{proof}

\subsection{Subalgebras of direct powers of simple algebras.}\label{ssSDPSA}

A variety $\V$ is called {\it locally finite} if every finitely generated algebra
in $\V$ is finite (equivalently, the $\V$-free algebras of finite ranks are finite).
Recall that an algebra $A$ with nonzero multiplication is called \emph{simple} if it has no proper nozero ideals. For instance, a composition section of any algebra is a simple algebra. If $\V$ has a finite simple algebra $A$, then it contains the direct powers $\underbrace{A\times\cdots\times A}_n$ of $A$ for all $n=1,2,\ld$. These play an important role in the description of \emph{all} algebras in $\V$. For further references, this subsection presents an auxiliary information on the structure of such direct products. 

As usual, the factors of a direct 
product $P$ can be identified with ideals of $P$. Recall that a non-zero algebra $A$ is called
\emph{simple} if $A$  has no proper non-zero ideals. Also, we canonically identify algebras $A_1,\ld,A_m$ with the subalgebras, denoted by the same letters in their direct product so that $A_1\times\cdots\times A_m=A_1\oplus\cdots\oplus A_m$.

\begin{Lemma}\label{lDirectSimple} Let an algebra $A$ be the direct product of simple algebras: \\  $A = A_1\times\dots \times A_t$,
where $A_1,\ld,A_s$ have non-zero multiplication and $A_{s+1},\dots, A_t$ are one-dimensional with
zero multiplication. Then 
\begin{enumerate}
  \item[$\mathrm{(i)}$] every ideal $I$ of $A$ is the direct sum of some algebras from the first cluster and a subspace of the sum of the subalgebras $A_{s+1},\dots, A_t$.;
  \item[$\mathrm{(ii)}$] every factor-algebra of $A$ is isomorphic to a product $A_{i_1}\times\dots\times A_{i_r}$,
where $1\le i_1<\ld<i_r\le t$.
\end{enumerate}
\end{Lemma} 

\proof If the projection of $I$ to $A_1$ is non-zero, then using multiplication by
the elements from $A_1$, one can obtain the containment $A_1\subset I$ since $A_1$ is simple
and has nonzero multiplication. Similar observation applied to $A_2,\dots, A_s$
proves the statement (i). The second statement follows from the first one and from
the homomorphism theorem.

\endproof

A subalgebra $B$ of a direct product $P = A_1\times\dots \times A_t=A_1\oplus\cdots\oplus A_m$ is called a {\it  subdirect product} of the algebras $A_1,\dots, A_t$ if $B$ projects {\it onto} each
of the factors $A_i$. 

\begin{Lemma}\label{lDirectSimple1} A subdirect product $B$ of simple algebras $A_1,\dots, A_t$ is isomorphic to the direct product $A_{i_1}\times\dots\times A_{i_r}$, where $1\le i_1<\ld<i_r\le t$.
\end{Lemma}

\proof Note that in the above notation, the intersection $B\cap A_t$ is an ideal in
both $B$ and $A_t$ since $B$ projects onto $A_t$. So it is either zero or the entire $A_t$.
In the first case, $B$ is isomorphic to the projection of it to $A_1\times\dots\times A_{t-1}$.
In the second case, $B= B'\times A_t$, where $B'$ is a subdirect product of the algebras
$A_1,\dots, A_{t-1}$. So in both cases one can use the obvious induction on $k$ to
complete the proof.

\endproof

  If $A$ is an algebra, $B$ is the direct power of $A$, that is, the set of all functions $f:\{1,\ld,n\}\to A$, then the \emph{diagonal} $\Delta$ of $B$ is the subalgebra of all constant functions $f$, that is, those for which there is $a\in A$ such that $f(i)=a$, for all $i=1,\ld,n$. Given any $n$-tuple of automorphisms $\tau=(\alpha_1,\ld,\alpha_n)$ of $A$, we consider the subalgebra $\Delta^\tau$, consisting of the functions $f^\tau$, $f\in\Delta$, given by $f^\tau(i)=\alpha_i(f(i))$; clearly,  $\Delta^\tau\cong\Delta$. If there is no confusion, the name of diagonal is used also when speaking about $\Delta^\tau$ in place of $\Delta$.  
  
  \begin{Lemma}\label{Diagonal}
  Let $S$ be a finite simple algebra, $S^2\nez$, $P$ and $Q$ are direct powers of $S$; We write $P=S(1)\oplus\cdots \oplus S(\ell)$ and $Q=S_1\oplus\cdots\oplus S_k$, where $S(i)$, $S_j$ are
  isomorphic to $S$. Assume that there is an injective homomorphism $\vp: P\to Q$. Then there are pairwise disjoint subsets $J_1,\ld,J_\ell$ of the set $\{ 1,\ld,k\}$ such that the image of each $S(j)$ is the diagonal $\Delta_j$ of the direct sum $\oplus_{i\in J_i}S_i$.  Also, $\vp(P)=\Delta_1\oplus\cdots\oplus \Delta_\ell$. 
  \end{Lemma}
  
  \begin{proof} Consider the projection $\pi_j$ of $\vp(S(1))$ to some $S_j$. Since these simple algebras
  have equal orders, either $\pi_i(\vp(S(1))\ez$ or $\pi_i(\vp(S(1))= S_j$. So we denote by $J_1$ the set
 of indexes $j$, such that the composition $\pi_j\vp$ is an isomorphism. Thus, we obtain a
  diagonal embedding of $S(1)$ into the sum $T(J_1)$ of the direct summands $S_j$ of $Q$ with $j\in J_1$.
  
  Similarly we have the diagonal embedding of every $S(i)$ into $T(J_i)$, $i=2,\dots, \ell$.
  If $j\in J_i\cap J_{i'}$ for some $j$, and $i\ne i'$, then $\pi_j\vp(S(i))= \pi_j\vp(S(i'))= S_j$.
  Since $\pi_j\vp$ is a homomorphism and $S(i)S(i')\ez$, we have to obtain $S_j^2\ez$. This contradiction
  completes the proof.
  
    \end{proof}
    
    \section{Nilpotent varieties}\label{esssNS}

\subsection{Equivalent definitions.}\label{ssEDN}

When one transfers the concept of nilpotency from classical algebras to
arbitrary non-associative algebras, there are different ways of 
generalisations. In this subsection we want to explain the correlation
of these definitions of nilpotency.

Let us call the descending series of subspaces
\begin{equation}\label{eCentFilt}
A = A_1 \supset A_2 \supset \ld 
\end{equation}
a \emph{central filtration} if $A_iA_j \subset A_{i+j}$ for any indices $i,j\ge 1$. The terms $A_i$ in (\ref{eCentFilt}) are ideals because $A_jA = A_jA_1 \subset A_{j+1}  \subset A_j$ and $AA_j \subset A_j$.
The lower central filtration is the power series of $A$. By Lemma \ref{power}, it is 
a central filtration.

\begin{Lemma}\label{lCFLCF} For any central series (\ref{eCentFilt}) and any $i$, we have $A^i\subset A_i$.
\end{Lemma}
\begin{proof}
Since $A^1 = A_1$, $A^2 = A_1^2\subset A_2$, using induction by $i\ge 2$, 
we have $A^k \subset A_k$, $A^\ell \subset A_{\ell}$ for $k + \ell = i$, it follows that
\[
 A^i =\sum_{k+\ell=i}A^kA^\ell\subset \sum_{k+\ell=i}A_kA_\ell\subset A_i.
\]
\end{proof}

For any algebra $A$, having central filtration (\ref{eCentFilt}) of length $c\ge 1$ (т.е. $A^{c+1} =A_{c+1}\ez$) we can define an upper central filtration $\{ Z_j\,|\, j\ge 1\}$ by induction,  starting with $Z_1 = A$. Suppose the subspaces $Z_j$ have been defined for $j < i$, where $i \ge 2$. Then, by definition, $Z_i$ consits of the elements $a\in A$ such that
\begin{equation}\label{eTwoStars}
b_1b_2\cdots b_sab_{s+1}\ld b_t = 0
\end{equation}
for any bracketing of the left hand side, if for some $i_1,\ld, i_t < i$ we have $b_j\in  Z_{i_j}$ and $\sum_{j=1}^{t} i_j \ge c - i + 1$.
\begin{Lemma}\label{lUpper}
$Z_i \subset Z_{i-1}$. The upper central filtration is a central filtration, hence consists of the ideals of $A$. For any central filtration (\ref{eCentFilt}) and any $i$, we have $A_i \subset Z_i$.
\end{Lemma}
\begin{proof}
The first claim follows simply because $c-i+1 < c-(i-1)+1$. 

For the second one, let us assume that $a\in Z_i, b\in Z_j$ and consider any product $b_1b_2\cdots b_s(ab)b_{s+1}\ld b_t$ where $b_k\in  Z_{i_k}$ and
\[
 \sum_{k=1}^{t}  i_k \ge c - i - j + 1.
 \] 
Since
$b \in Z_j$, it follows that
\[
 j +
\sum_{jk1}^{t}i_k \ge c - i + 1.
\] 
By definition of $Z_i$, it follows from $a \in Z_i$ that $b_1b_2\cdots b_s(ab)b_{s+1}\ld b_t=0$.
So by definition of $Z_{i+j}$, it follows that $ab\in Z_{i+j}$. As a result, $Z_iZ_j\subset Z_{i+j}$.

Now if $a\in A_i$, then induction and the definition of the central filtration (\ref{eCentFilt}) imply that all the products (\ref{eTwoStars}) are the elements of $A_{c+1}$ provided that  $\sum_{j=1}^{t}i_j \ge c - i + 1$. So  we have $a \in Z_i$, proving the last claim of the lemma.
\end{proof}

We conclude that if $A$ has a finite central filtration reaching $\{ 0\}$, it follows from Lemmas \ref{lCFLCF} and \ref{lUpper} that the lengths of the lower and upper central filtrations in $A$ coincide. These lengths are called the nilpotency class of $A$.

By the inductive definition of the power series, $A^i$ is a linear span of all products of $i$ elements from $A$ with any bracketing. Therefore all algebras satisfying $A^{c+1}\ez$ form the variety
$\frN_c$ given by all identities $x_1x_2 \ld x_{c+1} = 0$ with all possible bracketing on the left hand side. In particular, $\frN_1$  consists of all algebras with zero product.

\medskip
We come across another definition  by considering an ascending annihilator series of an algebra $A$:
\begin{equation}\label{eThreeStars}
  \{ 0\}=A_0\subset A_1\subset\ld
\end{equation}
where, for any  $i > 0$ the subspace $A_i/A_{i-1}$ belongs to the annihilator of $A/A_{i-1}$. Using induction, we easily prove that each $A_i$ is an ideal of $A$. In particular, this is true for the upper annihilator series $\{ I_i\}$ where  $I_1 =
\{a \in A\,|\, ab = ba = 0\,\;\; \forall b\in A\}$ is the annihilator $I(A)$ of $A$ and  $I_j$ is the full preimage in $A$ of the annihilator $I_j/I_{j-1}$ of $A/I_{j-1}$. By induction on $i$, with trivial basis $j=0$, we always have $A_j \subset I_j$, for any $j\ge 0$. Indeed, if already if $A_{j-1}\subset I_{j-1}$, since $A_j/A_{j-1}\subset I(A/A_{j-1})$ it follows that $(A_j + I_{j-1})/I_{j-1}$ annihilates $A/I_{j-1}$, and so  $A_j \subset I_j$.

To define the lower annihilator series of an arbitrary algebra $A$, we define the depth of an arbitrary nonassociative monomial. We assign depth 0 to the monomials $x_i$ and we say that the depth of the monomial $uv$ equals to the maximum of the depths of $u$, $v$ plus 1. For instance, the depth of $(x_1x_2)(x_3x_4)$ equals 2. We define the subspace $A(i)$ as a linear span in $A$ of all values of monomials of depth $i$. For instance, $A(0) = A$, $A(1) = A^2$.

To compare this series withan arbitrary annihilator series, we enumerate the terms of an arbitrary descending annihilator series in the reverse order: 
\begin{equation}\label{eFourStars}
  A=A'_0\supset A'_1\supset A'_2\supset\ldots
\end{equation}
where for each $i>0$, the ideal $A'_{i-1}/A'_i$ belongs to the annihilator of $A/A'_i$.
\begin{Lemma}\label{lLAS}
For any $i > 0$, we have $A(i) \subset A(i-1)$, the descending series $A(i)$ is an annihilator series. In particular, all its terms are ideals. For any annihilator series (\ref{eFourStars}), we have $A(i)\subset A_i'$.
\end{Lemma}


\begin{proof}
By induction, it follows from the definition of the depth, that any monomial of depth $i$ is the value of some other monomial of depth $i-1$. Thus the first claim follows. Since the products  $AA(i-1)$ and $A(i-1)A$ are the linear spans of some monomials of depth $i$, it follows that these products are in $A(i)$, proving that the lower annihilator series is indeed an annihilator series. Note further that $A(i)$ is the linear span of the values of monomials $uv$ in $A$ where one of $u,v$, say $u$ has the depth $\ge i - 1$. So we may assume $u\in A(i-1)$. Now, by induction, $A(i-1) \subset A'_{i-1}$ and then the value of $uv$ is an element of $A'_{i-1}A \subset A'_i$. Finally, we have $A(i) \subset A'_i$.
\end{proof}

As a result, we have obtained that if an annihilator series (\ref{eFourStars}) is finite and reaches $\{ 0\}$, that is, $A'_d=\{ 0\}$, for some $d$, then both upper and lower annihilator series have length at most $d$, that is, have the same length. We call this length the \emph{nilpotent depth of $A$}.

It follows from Lemma \ref{lLAS} that algebras of depth $\le d$ form a variety  $\frN(d)$, given by identities $w = 0$ for all multilinear monomials $w$ of degree $d+1$ and depth $d$. Indeed, any monomial of depth $d$ appears from a multilinear monomial of depth $d$ in $d+1$ variables under an appropriate substitution of nonassociative monomials in place of the variables. This substitution is trivial if $d=0$. Futher, if $w=uv$ where the depth of $w$ is $d$ and, say, the depth of $u$ is $d-1$ that is, for $u$, the substitution in  
monomial of degree $d$ and depth $d-1$ in $x_1,\ld,x_d$ is given, then we complement it by a substitution
$x_{d+1}\mapsto v$ in the monomial $(x_1\cdots x_d)x_{d+1}$.

\begin{Prop}\label{pNcNd}
\begin{enumerate}
  \item[$\mathrm{(i)}$] If an algebra $A$  has nilpotent  class $c$ then it has nilpotent depth $d\le c$.
  \item[$\mathrm{(ii)}$] If an algebra $A$ has nilpotent depth $c$, then it has
nilpotency class  $c\le 2^{d -1}+1$.
\end{enumerate}
\end{Prop}
\begin{proof}

(i) Every monomial of depth $d$ has degree at least $c+1$, an so it takes only zero
values in every algebra of nilpotency class $c$. So the first statement
follows from the above description on the identities defining the varieties
${\cal N}_c$ and ${\cal N}(d)$.

 (ii) Arguing by induction, we may assume that all products of degree at least $2^{d-2}+1$ are in  $I_1(A)$. But then in any product $uv$ of length $2^{d-1}+1$ one of $u,v$ has degree at least $2^{d-2}+1$ hence by induction belongs to  $I_1(A)$. Then $uv=0$, as needed. So, the nilpotent class of $A$ is at most $2^{d-1}+1$. 
 \end{proof}
 
 Now we are able to formulate, for references, the following
 
 \begin{Prop} \label{nilp} Given an algebra $A$ over a field $\bf F$, the following
 properties are equivalent.
\begin{enumerate}
  \item[$\mathrm{(i)}$] $A$ has a central series $A=A_1\ge A_2\ge \dots \ge A_{c+1}=\{ 0\}$ for some $c\ge 1$;
  \item[$\mathrm{(ii)}$] $A^{c+1}= \{0\}$ for some $c\ge 1$;
  \item[$\mathrm{(iii)}$] $A$ belongs to the variety ${\cal N}_c$ for some $c$;
  \item[$\mathrm{(iv)}$] $A$ has an annihilator series $\{0\}= A_0\subset  \dots\subset  A_d=A$ for some $d\ge 0$;
  \item [$\mathrm{(v)}$]$A$ belongs to the variety ${\cal N}(d)$ for some $d>0$.
\end{enumerate} 
 \end{Prop} 
 
This Proposition allows one to say that an algebra is {\it nilpotent}, by definition, if it has one of the properties (i) - (v) . A variety $\V$ of algebras is called {\it nilpotent} if every algebra $A\in \V$ is nilpotent. The nipotency class (depth) of
the algebras in a nilpotent variety is bounded from above, because if one gets
an infinite series $A_i$ of algebras in $\V$ with unbounded classes (depths), then the
Cartesian products of them is not nilpotent, being in the variety $\V$, a contradiction.
Thus the maximal $c$ (resp., $d$) such that $\V$ contains an algebra of nilpotent class
$c$ (of nilpotent depth $d$) is called the {\it nilpotency class} ({\it nilpotency depth})
of the variety $\V$.

\begin{Remark} \label{cequald} If we add the associative law to the identities defining the varieties
${\frN}_c$ and ${\frN}(c)$, respectively, then we obtain equivalent systems of identities,
that is, they define the same variety. Hence, the class and the depth of an associative
nilpotent algebra are equal. The same is true for nilpotent Lie algebras.

\end{Remark}

In the following lemma, we give some simple properties which are used throughout our treatment of algebras.

\begin{Lemma}\label{lNilpSimpleProp}
The following are true in any algebra $A$ over a field $\F$.
\begin{enumerate}
  \item[$\mathrm{(i)}$] If $A$ is nilpotent then the annihilator $I(A)$ has nontrivial intersection with any nontrivial ideal $J$ of $A$.
  \item[$\mathrm{(ii)}$] If an ideal $J$ belongs to the annihilator $I(A)$ of $A$ and $A/J$ is nilpotent then also $A$ is nilpotent
\end{enumerate}

\end{Lemma}

\proof (i) If $J$ has an element $a$ which does not belongs to $I(A)$, then for some
$b\in A$ either $ab\in J\backslash \{0\}$ or  $ba\in J\backslash \{0\}$. In the first case,
if $ab\in I(A)$, then we are done. Otherwise we can find an element $c\in A$ such that
either $(ab)c\in J\backslash \{0\}$ or  $c(ab)\in J\backslash \{0\}$. Since by Proposition \ref{nilp} (ii) number of factors in non-zero products are bounded, iterating this
argument, we obtain a non-zero element in the intersection $J\cap I(A)$ .

(ii) Since the factor-algebra $A/J$ has a finite annihilator series $J/J\subset  J_1/J\subset \dots$,
the algebra $A$ has a finite annihilator series $\{0\} \subset  J\subset  J_1\subset \dots$. Therefore
 $A$ is nilpotent by Proposition \ref{nilp} (iv).

\endproof 

\subsection{Subideals.}\label{ssSubideals}

A subalgebra $B$ of an algebra $A$ is called a \emph{subideal} if there is a series of subalgebras $B=B_0\subset B_1\subset\cdots\subset B_s=A$ such that each $B_{i-1}$ is an ideal in $B_i$, for all $i=1,\ld,s$. The least length $s$ of all such series  (called the \emph{subideal series}) is called the subideal depth of $B$. Subideals of depth 1 are obviously the usual ideals.

One of characterizations of nilpotent varieties is the following.

\begin{Thm} \label{subidnilp} A variety of algebras $\V$ over a field is nilpotent if and only if
every subalgebra $B$ of any algebra $A\in\V$ is a subideal of $A$.
\end{Thm} 

\proof Let a variety $\V$ be nilpotent, $A\in\V$, and $B$ a subalgebra of $A$.
If an algebra $A$ is nilpotent of depth $d$, and $\{0\}\subset  I_1\subset \dots \subset  I_d=A$ is an
annihilator series in $A$, then the following is a subideal series:
\[
B\subset B+I_1\subset B+I_2\subset\cdots\subset B+I_d=A.
\]

For the converse claim, consider a $\frV$-free algebra $F=F_\frV(y,x_1,x_2,x_3,\ld)$ and choose a subalgebra $G=F_\frV(x_1,x_2,x_3,\ld)$. Suppose, there is a subideal series
\[
G=G_0\subset G_1\subset\ld\subset G_k = F,
\]    
 Now let  $w$  be any monomial of depth $d$, in which every generator occurs at most once. If we manage to show that $w$ vanishes in $F$ for $d=k$, then  by Proposition \ref{nilp} (v), the nilpotence of $F$ and $\V$ will be established.
 
 We proceed by induction on the depth $d$ to prove that one of variables $x_i$ in a monomial $w$ of depth $d$ depending on $x_1,x_2,\dots $ can be replaced
 with $y$ so that the value  obtained belongs to
$G_{k-d}$. Since $G_{k-1}$ is an
ideal in $F$ if $d=1$ and $w=x_1x_2$ it follows that $yx_2\in G_{k-1}$. Then let
$w=uv$ and, say, $u$ has depth $d-1$, and so one can replace one variable in $u$
with $y$ and obtain an element of $G_{k-d+1}$. Since $v\in G$ and $G_{k-d}$ is an ideal of $G_{k-d+1}$,  after this replacement the value of $w=uv$ will belong to $G_{k-d}$. Thus, after a single replacement  $x_i\mapsto y$, any polynomial $w$ of depth $k$  in which every generator occurs at most once, becomes an element $v$ of $G$ .

On the one hand, for any substitution  $y\mapsto 0$, the respective homomorphism annihilates $v$, but on the other hand, if a homomorphism is identical on  $G$ then  $v$ should be mapped into itself since $v\in G$. Hence, $v=0$ in $F$. Since $w$ is an automorphic image of $v$, by Lemma \ref{lRelX} we have $w=0$ in $\V$, which is what we wanted.
\endproof

\begin{Remark}\label{2k1} A subalgebra of an algebra $A$ of nilpotent 
depth $d$ is a subideal of $A$ of depth at most $d$ as it is explained in the first
part of the proof of Theorem \ref{subidnilp}. The second part shows that if, for some constant $k$, it is true that in every finitely generated algebra of a variety $\frV$, any subalgebra is a subideal of depth at most $k$ then the variety is nilpotent. It is sufficient to consider in place of $F$, a $\frV$-free algebra freely generated by $x_1,\dots,x_{2^k},y$.

\end{Remark}

The \emph{local} version of Theorem \ref{subidnilp} is the following.

\begin{Thm}\label{cSubNilp}
Every subalgebra $B$ of any finite algebra $A$ in a locally finite variety $\frV$ is a subideal of $A$ if and only if  $\frV$ is locally nilpotent.
\end{Thm}

\proof The part 'if' follows from Remark \ref{2k1}. For 'only if', it suffices to prove
that given a finite non-nilpotent algebra $A\in\frV$, there exists a finite algebra
$C\in\frV$ containing a subalgebra $B$
which is not a subideal of $C$.  One may assume that all proper factor-algebras of $A$ are nilpotent
and $I(A)\ez$ by Lemma \ref{lNilpSimpleProp} (ii).

Let $M$ be a minimal ideal of $A$.  So $AM\nez$ (or $MA\ne \{ 0\}$), that is, there is $a\in A$ and $m\in M$ such that $am\ne 0$.

 Let $P = A_1\times A_2$ be the direct product of two copies of $A$ containing  the direct
 product $N$ of two copies $M_1$ and $M_2$ of $M$. 
  Denote by $B=\{(a,a)\,|\,a\in A\}$ the diagonal of $P$ and set $C = B+N$. Then $B$ is not an ideal in $C$ because $B M_1=(AM,\{ 0\})\not\subset B$.

Since $B+M_1 = C$, any subalgebra $D$ of $C$ strictly containing $B$  must contain a nonzero $(m_1,0)\in M_1$. Then $D$ has to contain the whole $M_1$ since $M_1$ is a
minimal ideal of $A_1$ and the products of the elements from $M_1$ by
an element $a\in A_1$ coincides with the products by the diagonal element
$(a,a)$. In this case, $D=C$, and so there is no subalgebras strictly  between $B$ and $C$. Hence $B$ is not a subideal of $C$. Since $C$ is finite and $C\in \V$, we obtain
the required counterexample.

\endproof

\subsection{Other criteria of nilpotency.}\label{ssNNandNS Var} We start with the annihilators in the algebras of nilpotent varieties.

\begin{Thm}\label{eclNonNilp} A variety of algebras $\frV$ over a field is nilpotent 
if and only if the annihilator of every nonzero algebra $A\in\frV$ is non-zero.
\end{Thm}
\begin{proof} The part `only if' follows from Proposition \ref{nilp} (iv).

To prove part `if', we assume that a variety $\frV$ is not nilpotent.
Then we will use the fact
that any nilpotent variety of algebras has finite basis of identities \cite{VLN}. As a result, the ideal of identities of a nilpotent subvariety of a variety $\frV$ cannot
be a proper union of an infinite increasing chain of verbal ideals in the
$\frV$-free algebra $F(\frV)$ of infinite rank.   Using Zorn's lemma and the Galois type correspondence between  the subvarieties of $\frV$ and the verbal ideals of $F(\V)$, we
conclude that  there is a non-nilpotent subvariety $\frU$ in $\frV$ whose proper
subvarieties are all nilpotent.
          
  Now let $F=F(\frU)$ be a free algebra of countable rank in $\frU$. The annihilator $I=I(F)$ is fully invariant, hence an ideal of identities of a subvariety $\frW$ of $\frU$ (see \cite[Proposition 1]{BOGP}),
  and the suvariety $\frW$ is proper in $\frU$ if $I\ne \{ 0\}$,  so it is nilpotent. Also, the factor-algebra $F/I$ is free in $\frW$, hence nilpotent. But then by Lemma \ref{lNilpSimpleProp} also $F$ is nilpotent, a contradiction. Thus  $I=\{ 0\}$, and so $F$ is the required
  counter-example.   
\end{proof}

\begin{Thm}\label{exNonnilpMalcev} Let $\frV$ be a locally finite variety of
  algebras. Then the following properties are equivalent.
\begin{enumerate}
  \item[$\mathrm{(i)}$] The variety $\frV$ is locally nilpotent;
  \item[$\mathrm{(ii)}$] For any $\frV$-free algebra $F$ of finite rank, any set of elements $Y$ linearly independent modulo $F^2$ is a freely generating set of $F$.
\end{enumerate}
\end{Thm}

\begin{proof} (i)$\to$ (ii) This is a well-known property of relatively free nilpotent
algebras over any field, proved by  A. I. Malcev (Theorems 1 and 2 in \cite{AIM}).

 (ii)$\to $(i)
Arguing by contradiction, we have
a non-nilpotent $F=F_n(\frV)$. Then we will construct a proper subalgebra $H$, such that $F = H + F^2$ and therefore a basis of $H$ modulo $F^2$ will be a non-generating set which generates $F$ modulo $F^2$ .

Since $\frV$ is not locally nilpotent, there is a non-nilpotent finite algebra in $\frV$ of minimal order. Let  $M$ be a minimal ideal of $A$. Then $B=A/M$ is nilpotent. Since $A$ is non-nilpotent, by Lemma \ref{lNilpSimpleProp} (ii), $M$ does not annihilate $A$.   Now we use the construction from the
proof of Theorem \ref{cSubNilp}. Namely, as in that proof, let
$P = A_1\times A_2$ be the direct product of two copies of $A$ containing  the direct
 product $N$ of two copies $M_1$ and $M_2$ of $M$, 
 $B=\{(a,a)\,|\,a\in A\}$ the diagonal of $P$,   $C = B+N$. We have $BM_1\ne \{ 0\}$, and $M_1$ is
 a minimal ideal in $C$. 

Since $BM_1\subset C^2\cap M_1$, the latter intersection is not zero. Therefore we
have $C^2\cap M_1=M_1$, because $M_1$ is a minimal ideal of $C$, whence $M_1\subset  C^2$.
 All the above can be said about $M_2$. It follows that $N\subset  C^2$. This latest containment and a strict containment $B\subset C$ can be used as follows.

Choose a $\frV$-free algebra $F$ such that there exists an epimorphism $\vp:F\to C$. Let us denote by $H$ the complete preimage of $B$ and by $K$ the complete preimage of $N$, both under $\vp$. Then $F=H+K$. Since $N \subset C^2 =\vp(F^2)$, it follows that $K\subset H+F^2$. As a result, $F = H + K \subset  H + K +F^2 = H +F^2$, as needed since $H\ne F$.
\end{proof}

\subsection{A residual property of finite relatively free nilpotent algebras.}\label{ssACFA}

In this subsection we show one more feature of finite nilpotent algebras.

Let $\cal S$ be a set of algebras. An algebra $B$ is {\it residually an algebra in
$\cal S$} if for every non-zero element $b\in B$ and some algebra $S\in \cal S$, there is an epimorphism $\vp: B\to S$ such that $\vp(b)\ne 0$. If $\cS=\{ A\}$ and $B$ is residually an algebra in $\cS$, then we say that \emph{$B$ is residually $A$}.

Suppose an $m$-dimensional algebra $A$ is nilpotent of class $c$ and has a central series, a flag $A=A_0\supset A_1\supset\ldots\supset A_m \ez$, which is a refinement of the power series $A\supset A^2\supset A^3\supset\cdots$. Let  $E = (e_1,\ld,e_m)$ be a basis in $A$, compatible with the flag, that is, $e_i\in A_{i-1}$, for all $i\ge 1$.

An analogue of the following lemma is known in Group Theory, but its proof there is more complex.

\begin{Lemma}\label{lPolFun}  Let  $p=p(x_1,\ld,x_k)$ be a nonassociative polynomial. Then the $j$th coordinate of the value $p(y_1,\ld,y_k)$ for $y_1,\ld,y_k \in A$ is a polynomial function $f_j$  of degree $\le c$ with zero constant term in the coordinates of  the
vectors $y_1,\ld,y_k$ in basis $E$. Here $j=1,\ld,m$.
\end{Lemma}

\begin{proof} Let us assign to each $e_i$ a weight $\wg e_j=s$ if $e_j\in A^s\backslash A^{s+1}$. Since for any $A^k$, a part of the basis $E$ is a basis of $A^k$, if the $\wg e_j=s$ and $\wg e_k=t$, then 
\[
e_je_k = \sum_\ell c_{jk}^\ell e_\ell\mbox{ where } c_{jk}^\ell=0\mbox{ if }\wg e_\ell< s+t.
\]

For the proof, we may assume that $p$ is a monomial in $x_1,\ld,x_k$. We will use induction over the degree of $p$ to prove a stronger claim, namely, that  $f_j$ has degree $\le s$, if the weight of $e_j$ equals $s$. If $\deg p = 1$, then the claim is obvious since in this case, the value of the function $f_j$ equals to the $j$th coordinate of  $y_1$.

Now let us assume that  $p = uv$, and that we have proved for the monomials  $u,v$ of lesser degrees that the values of $u$ can be written as $S_1=\sum_j \lambda_j e_j$, where $\lambda_j$ is a polynomial in the coordinates of  $y_1,\ld,y_k$, such that the degree of each polynomial $\lambda_j$ is at most the weight of $e_j$. Similarly, the values of $v$ are given by the sum  $S_2=\sum_\ell \mu_\ell e_\ell$.

When we multiply two summands of these sums,  $\lambda_j e_j$ and $\mu_\ell e_\ell$
with $\wg e_j =s$, $\wg e_\ell=t$, we obtain $\lambda_j \mu_l \sum_p c_{jl}^p e_p$, where the only vectors $e_p$ with nonzero coefficients are those, whose weights are at least $\wg e_j+\wg e_\ell=s+t$. At the same time, the degrees of their coefficients do not exceed   $\deg \lambda_j + \deg \mu_\ell\le s+t$. Similar restriction on the degrees of coefficients will hold when we consider the product  $S_1S_2$. The proof is complete.
\end{proof}

One more lemma we will need is the following.

\begin{Lemma}\label{lPolFun1} Let $A$ be an $m$-dimensional algebra of nilpotent class $c$ over a field $\F$, $|\F|=q$. Assume that a nonassociative polynomial $p=p(x_1,\ld,x_k)$ takes value  $b\in A$ after some evaluation in $A$. If $k\ge m(c+1)$, then there exists a generating set of elements $a_1,\ldots, a_k\in A$, such that $p(a_1,\ld,a_k) = b$.
\end{Lemma}

\begin{proof}

Assume $p(c_1,\ld,c_k) = b$, for some $c_1,\ld,c_k\in A$. We define the function $f(y_1,\ld,y_k) = p(c_1+y_1,\ld,c_k+y_k) - p(c_1,\ld,c_k)$ from the $k$-th direct power of $A$ to $A$, for which it is obvious that $f(0,\ld,0) = 0$. All solutions of the system of equations  
\[
f(y_1,\ld,y_k) = 0
\]
 in the coordinates in $E$ can be presented in the following way. Remember that $0$ on the right hand side has $m$ coordinates. According to Lemma \ref{lPolFun}, the solutions in $km$ coordinates of the vectors $y_1,\ld,y_k$ can be described as the solutions of a system of  $m$ algebraic equations $p_j = 0$, where each  $p_j$ is an ordinary (associative-commutative) polynomial of degree at most $c$. Since the set of $km$ zeros is the solution of the system, each $p_j$ has constant term 0.

The sum $d$ of degrees of all these polynomials do not exceed $cm$. By the Warning's Second Theorem \cite{W}, the number of solutions of this system is at least $q^{km - d}\ge q^{m(k-c)}$. It follows that $p$ takes value $b$ for at least $q^{m(k-c)}$ tuples of vectors $a_1,\ld, a_k$  of the form $c_1+y_1,\ld,c_k+y_k$, where  $f(y_1,\ld,y_k)=0$.

Thus, the proof will be complete if we check that $q^{m(k-c)}$ is greater than the number of ordered subsets of the elements $y_1,\ld,y_k$, which \emph{do not} generate $A$. Each such subset is contained in a subspace of dimension $(m-1)$ of the vector space $A$. The number of these subsets is equal to $q^{k(m-1)}$, while the number of subspaces is less than $q^m$. Hence the number of non-generating subsets is less than  $q^{km-k+m}$. Since $k\ge m(c+1)$, we finally have $m(k-c)\ge km-k+m$. The proof is complete.
\end{proof}

\begin{Thm}\label{cApprox}
If $A$ is a finite nilpotent algebra then all free algebras $F_n$ of ranks $n\ge m(c+1)$ in $\var A$ are residually-$A$ algebras. Here $m=\dim A$ and $c$ is the nilpotency class of $A$.
\end{Thm} 
\proof
Assume that  $p(x_1,\ld, x_n)\ne 0$ in the algebra $F_n = F(x_1,\dots, x_n)$. Then
$p(x_1,\dots, x_n)=0$ is not an identity of the algebra $A$ since $A$ generates
 $\var A$. Hence there are elements $a'_1,\ld,a'_n\in A$ such that
 $b=p(a'_1,\ld, a'_n)\ne 0$. By Lemma \ref{lPolFun1}, there are elements $a_1,\dots,
 a_n$ generating $A$ such that $b=p(a_1,\ld, a_n)$. Thus, the mapping $x_1\mapsto a_1,\dots, x_n\mapsto a_n$ extends to an epimorphism $\vp: F_n\to A$ such that
 $\vp(p(x_1,\ld, x_n)) = p(a_1,\ld, a_n) =b \ne 0$. So $F_n$ is a residually-$A$ 
 algebra.
 
\endproof

\section{Finite algebras in locally finite varieties}\label{esGF}

\subsection{Birkhoff's construction.}\label{ssBCN}

The following is a finite version of the fundamental G. Birkhoff's Theorem \cite{Birk}.

Let $\frV$ be a variety generated by a finite algebra $A$ over a finite field $\F$. Given a finite set $X=\{ x_1,\ld,x_n\}$ we define the subset $\Phi$ of the direct power $A^X$ as the set of \emph{nonzero} functions from $X$ to $A$. Obviously, $|\Phi|={|A|^n-1}$. We consider the Cartesian product $B$ of $|\Phi|$ copies of $A$. It is convenient to view  $B$ as the set of functions on $\Phi$ with values in $A$. The ``componentwise'' product $fg$ of $f,g\in B$  is given by $(fg)(\vp)=f(\vp)g(\vp)$
where $\vp\in\Phi$. We will call $B$ the $n$th \emph{Birkhoff's algebra of $A$}, denoted by $\cB_n(A)$. We call the dimension $\beta_A(n)$ of $B$ over the ground field $\F$ the \emph{Birkhoff's function of $A$}. We have $\beta_A(n)=(\dim A)(|A|^n-1)$.

For each  $x\in X$, the function $\bar{x}\in B$ is defined by $\bar{x}(\al)=\al(x)$, for all $\alpha\in\Phi$. Then the map $\iota:F(\frV,X)\to B$, given by $\iota(x)=\bar{x}$ is a homomorphism of $F=F(\frV,X)\to B$. Actually, this is a monomorphism. Indeed, if $x_1,\ld,x_n\in X$ and $w(x_1,\ld,x_n)\in \Ker\iota$  then for any nonzero function $\vp:\{x_1,\ld,x_n\}\to A$, we would have 
\[
\iota(w(x_1,\ld,x_n))(\vp)=w(\vp(x_1),\ld,\vp(x_n))=0,
\]
meaning that $w(x_1,\ld,x_n)=0$ is identity in $A$, hence $w(x_1,\ld,x_n)$ is a zero element of $F$. We can restate the conclusion to the above argument, as follows.

\begin{Thm}\label{etBirkhoof_Estimate}\emph{[Birkhoff]}
Let $\frV=\var A$ be a finitely generated variety, where $A$ is a finite   algebra. Then the free  $\frV$-algebra $F_n=F(x_1,\ld,x_n,\V)$ of rank $n$ is isomorphic to a subalgebra in the Birkhoff's algebra $\cB_n(A)$ of $A$. So $\dim F_n\le\beta_A(n)$, or $|F_n|\le |A|^{|A|^n-1}$. The same estimate is true for any $n$-generated algebra $C\in\var A$:
\begin{equation}\label{eBirkhoff}
   \dim C\le(\dim A)(|A|^n-1)\mbox{ or }|C|\le |A|^{|A|^n-1}. 
 \end{equation}
\end{Thm}

\begin{proof}
  The first inequality in (\ref{eBirkhoff}) follows because $C$ is a factor-algebra of $F_n$, hence 
  \[
  \dim C\le\dim F_n\le \beta_A(n)= (\dim A)(|A|^n-1).
  \]
  If $\F=q$ then 
  \[
  |C|=q^{\dim C}\le q^{(\dim A)(|A|^n-1)}=(q^{\dim A})^{|A|^n-1}\le |A|^{|A|^n-1},
  \]
  as claimed.
\end{proof}

\begin{Cor} \label{cc} Let $\V$ be the variety generated by a finite   algebra $A$.
Then the orders of chief factors (of  composition factors) in any algebra $B\in \V$
do not exceed the maximum of the orders of the chief factors (resp., of  composition factors)
of the algebra $A$.

\end{Cor}

\proof Using Jordan - H\" older Theorem for the algebras, we see first that the statement is true if $B$ is a direct power of $A$, second, it is true for subalgebras of  direct powers,
and so by Theorem \ref{etBirkhoof_Estimate}, for $\V$-free algebras of finite rank, and third, it is
inherited by their factor-algebras, which completes the proof.
 
\endproof

\begin{Remark} \label{subBirk} In many cases, the \emph{double exponential} estimate $|F_n|\le |A|^{|A|^n-1}$ can be improved. This was noticed as early as in 1937 in a paper by B. H. Neumann \cite{BHN}. 
 
To see this, 
let us denote the extension of $\lambda\in\Phi$  to $F_n$
by the same letter and consider for each $\lambda$ the subalgebra $A_\lambda=\lambda(F_n)\subset  A$. Then each function from the image $\iota(F_n)$ belong
to the space of dimension  
\begin{equation}\label{eBirkSub}
\sum_{\lambda\in\Phi}\dim A_\lambda\le(|A|^n-1)\dim A\le \beta_A(n).
\end{equation}

It follows that if an algebra $A$ has a proper nonzero subalgebra then the first inequality in (\ref{eBirkSub}) is strict and the dimension $d(n)$ of $F_n(\frV)$ is bounded from above by a function strictly smaller than Birkhoff's function. 

We treat this question in more detail in Theorem \ref{etFreeSpectrum1}. 

\end{Remark}

\subsection{Enveloping algebras and varieties.}\label{sssEAV}
 With each algebra $A$ over a field $\bf F$, one can associate three algebras of linear operators on $A$.

  The algebra $U_L(A)$ is a subalgebra in $\End_{\bf F} A$ generated by the operators $L(a)$ of left multiplication by the elements $a\in A$. Similarly, $U_R(A)$ is a subalgebra in $\End_{\bf F} A$ generated by the operators $R(a)$ of right multiplication by the elements $a\in A$. Finally, $U(A)$  is a subalgebra in $\End A$, generated by all the operators of both left and right multiplication $L(a), R(a)$. We will call $U_L(A)$, $U_R(A)$, and $U(A)$ the \emph{left, right, and (two-sided) associative envelopes} of $A$.

 Having in mind Remark \ref{cequald}, we state the following Proposition giving one more definition of nilpotency of \emph{nonassociative} algebras in terms of their \emph{associative} enveloping algebras.

\begin{Prop} An algebra $A$ is nilpotent of depth $d\ge 1$ if and
only if the associative algebra $U(A)$ is nilpotent of class $d-1$.
\end{Prop} 

\proof Let $A$ be nilpotent of depth $d$. By Proposition \ref{nilp} (iv), there is
an annihilator series $\{ 0\}=A_0\subset  A_1\subset \dots\subset  A_{d-1}\subset  A_d = A$. By the definition of this series, $L(a)(b)$ and $R(a)(b)$ belong to $A_{i-1}$ if $a\in A$ and $b\in A_i$, $i\ge 1$.
Therefore the same is true if one replaces the operators $L(a)$ and $R(a)$ with any elements
of $U(A)$. Thus, by induction, the product of any $d$ elements of $U(A)$ is zero  so that $U(A)$ is an associative algebra of nilpotent class at most $d-1$.

Assuming now that $U(A)$ is a nilpotent algebra of class $d-1$, we can see that the value of every monomial $w$ of degree $d$ and depth $d-1$ equals zero in $A$. Indeed, the computation of
a value of $w$ in $A$ can be interpreted as a consecutive application of $d-1$ operators
of left and right multiplication, whence $A\in {\frN}(d)$.
\endproof

Given a locally finite variety of algebras $\frV$, one can consider three varieties of associative algebras $\frM_L(\frV)$, $\frM_R(\frV)$, and $\frM(\frV)$,  generated by all left, respectively, right and two-sided associative envelopes of finite algebras $A\in\frV$.

\begin{Prop}\label{pCorr1} If $A$ is a finite algebra then the variety $\frM(\var A)$ is
generated by the finite associative algebra $U(A)$.
\end{Prop}

\begin{proof} 

Let $I$ be an ideal in the algebra $A$. Obviously, $I$ is invariant under the action of $U(A)$. Thus, $\vp\in U$ naturally acts on $A/I$ by $\vp(a+I)=\vp(a)+I$. 
Hence we have a natural map $U(A)\to U(A/I)$, which is an epimorphism.

Now if $B$ is a subalgebra in $A$, then the elements $R(b)$, $L(b)$, $b\in B$, generate a subalgebra $U'$ of $U(A)$ which leaves $B$ invariant. The homomorphism of restriction, maps $U'$ onto $U(B)$. Thus $U(B)\in \var U(A)$.

If $B=C\times D$ then there are two epimorphisms of restriction from $U(B)$ onto $U(C)$ and $U(D)$. The intersection of their kernels is trivial. As a result, $U(B)$ is isomorphic to the direct product of $U(C)$ and $U(D)$.

By Birkhoff's Theorem \ref{etBirkhoof_Estimate}, any finite algebra $B\in\var A$  can be obtained from $A$ using Birkhoff's operators of taking subgroups, homomorphic images and direct products. So it follows from the above that for any finite $B\in \var A$, the algebra $U(B)$ can be obtained from $U(A)$ by similar operations. Hence $U(B)\in\var U(A)$. 

\end{proof}

\begin{Remark}\label{rFRW} It is possible to define $\frW$  with the help of any, not necessarily finite algebras in $\frV$. This follows from the fact, that if a polynomial in the  linear operators of, say, left multiplications, defined by the elements $a_1,\ld,a_s$, acts trivially on any finitely generated subalgebra $\la a_1,\ld,a_s,\ld,a_t\ra$ of a locally finite algebra $A$ in $\frV$, then it acts trivially on the whole of $A$.
\end{Remark}

\subsection{Minimal representation of subdirect products.}\label{esMinRep}

By Birkhoff's Theorem \ref{etBirkhoof_Estimate}, a free algebra $F$ in $\var A$ is a subalgebra of a Cartesian power of  the algebra $A$. Therefore $F$ is a subdirect product of algebras isomorphic to subalgebras of $A$. Subdirect products of a family of algebras is not defined up to isomorphism. So in this subsection we define \emph{minimal} representations and discuss their propertiess.

Recall that a {\it section} of an algebra $A$ is a factor-algebra of a subalgebra of $A$. All sections of an algebra $A$, except for $A/\{0\}$ are called {\it proper} sections of $A$.
Here we call a set $\cal S$ of finite non-zero algebras over a field $\bf F$  {\it sectionally closed} if for every $S\in \cal S$ and every non-zero section $T$ of $S$,  the set $\cal S$ contains an  isomorphic copy of the algebra $T$. If a finite algebra $A$ is isomorphic to a subdirect product, where the factors isomorphic
to some algebras from $\cal S$ then we will write $A\in \texttt{sd}(\cal S)$.

If $A$ is a subdirect product of $S_1,\dots, S_m$, then we may assume that $|S_1|\ge|S_2|\ge\cdots\ge|S_m|$.
Such a representation of $A$ as a subdirect product of algebras $S_i$ from $\cal S$ is called {\it minimal}
if the sequence $(|S_1|,|S_2|,\ldots,|S_m|)$ is minimal with respect to the lexicographical ordering
in comparison to other such representations of $A$ as an element of the class $\texttt{sd}(\cal S)$.

To formulate the next well-known Proposition, we recall that an algebra $S$ is called {\it monolitic} if the intersection $M(S)$ of all non-zero ideals of $S$ is non-zero. The unique minimal ideal $M(S)$ is called the {\it monolith} of $S$. Every algebra $S$ is a sudirect product of  the
monolithic factor-algebras of $S$. Therefore if some $S_i$ in a representation of $A\in \texttt{sd}(\cal S)$ 
  is not monolithic, i.e. $S_i$ is a subdirect product of its proper factor-algebras, then
  one can choose a lexicographically lesser representation for $A$. As the consequence we have
  
  \begin{Prop}\label{minsd} Let $\cal S$ be a sectionally closed set of finite algebras and suppose that
  a finite algebra $A$ belongs to the class \texttt{sd}(\cal S). Then  all the factors $S_i$ in the
  minimal representation of $A$ with respect to $\cal S$, are monolithic algebras.
  
  \end{Prop}

Combining Proposition \ref{minsd} with Theorem \ref{etBirkhoof_Estimate} we obtain

\begin{Cor}\label{etBirkhoof_General} Let $\frV$ be a  variety of algebras over a finite field $\F$. Assume that $\frV$ is generated by a finite sectionally closed set of finite algebras $\cal S$. Then any free algebra of finite rank in $\frV$ is a subdirect product of finitely many monolithic algebras in $\cal S$.
\end{Cor}

We will also refer to the following properies of minimal representations.

\begin{Lemma}\label{minrep} (\emph{Minimal representations}) Let $S$ be a sectionally
closed set of finite algebras and a finite algebra A has a minimal representation with respect 
to $\cal S$ 
\begin{equation} \label{rep}
A\subset P = \prod_{i=1}^m S_i, \mbox{ such that }|S_1|\ge|S_2|\ge\cdots\ge|S_m|,
\end{equation}
Then for all $i=1,\ld,m$, the following are true:
\begin{enumerate}
\item[$\mathrm{(i)}$] $S_i$ is a monolithic algebra,
\item[$\mathrm{(ii)}$] the projection of $A$ into $S_i$ equals $S_i$,
\item [$\mathrm{(iii)}$] the intersection $A\cap S_i$ is an ideal of $S_i$,
\item[$\mathrm{(iv)}$] the monoliths $M(S_i)$ are minimal ideals in $A$.
\end{enumerate}
\end{Lemma}

\proof The first property is given by Lemma \ref{minsd}.  The second one follows from
the minimality of the representation (\ref{rep}) since the factor $S_i$ can be replaced
with the projection of $A$ to $S_i$. 

Since $A$ projects onto the whole $S_i$, any product of an element $b\in A\cap S_i$
by an element from $S_i$ can be replaced with a product by an element from $A$,
and property (iii) follows.

For the same reason, if $A$ contains a non-zero element $b\in S_i$, it has to contain the whole 
monolith $M_i$, and we are done.  Otherwise $A\cap S_i=\{ 0\}$, and it follows that $A$ is isomorphic to the projection of $A$ to the product of all factors of $P$ except for $S_i$, which
contradicts  the minimality of the representation (\ref{rep}). So the proof of (iv) is
complete.

\endproof

\subsection{Variety generated by a finite simple algebra.}\label{sSKA} 

A non-zero finite algebra $A$ is called {\it critical} if it does not belong to
the variety generated by all proper sections of $A$. If $A$ is not critical then the variety
$\var A$ is generated by the proper sections of $A$. Therefore any locally  finite variety $\frV$
can be generated by a set of critical algebras from $\frV$. However a variety generated by a finite algebra over a field can contain infinitely many subvarieties \cite{OVL}. So we raise the following

\begin{problem}\label{pbmInfNumberSubvar} Can a variety generated by a finite algebra over a field contain \emph{uncountably} many subvarieties?
\end{problem}

A critical algebra is monolithic since otherwise there are two non-zero ideals with zero intersection, and the algebra is a subdirect product of two proper factor-algebras.

The following theorem clarifies the structure of finite algebras in a variety
generated by a finite simple algebra.  It follows from Corollary \ref{cc} that a \emph{finite simple algebra is critical}. (Moreover, a non-zero finite algebra $A$ is critical if the product of any two non-zero ideals of $A$ is non-zero \cite{BDDF}.)

\begin{Thm}\label{tSDI} Let $C$ be a finite algebra in the variety $\frV=\var A$, where $A$ is a finite simple algebra. Then the following are true
\begin{enumerate}
\item[$\mathrm{(i)}$] $C$  is the direct sum of two ideals:  $C=P\oplus Q$. The ideal $P$ is a direct power of the algebra $A$ (possibly zero), while $Q$ belongs to the variety  $\frU$ generated by all proper subalgebras of  $A$. 
\item[$\mathrm{(ii)}$]  If $C=F$ is a free algebra of rank $n$ in $\frV$ then $F=P\oplus Q$, where $P$ is a direct power of $A$ and $Q$ is a $\frU$-free algebra of rank $n$. 
\end{enumerate} 
\end{Thm}

\proof One may assume that $A^2\ne \{ 0\}$ since otherwise the statement of the theorem
is obviously true. We first prove Claim (ii).

By Birkhoff's argument, $F\in \texttt{sd}(A)$, and one can choose a minimal representation of $F$
according to Proposition \ref{minsd}. Hence $F$  is the subdirect product of $P=A_1\times\cdots\times A_s$, where  $A_i\cong A$, for all $i=1,\ld,s$, and several proper sections  of $A$, that is,  $F \subset  K = P \oplus H$,  where $H \in \frU$. Since
the monolith of a simple algebra is the whole algebra, by Lemma \ref{minrep} (iv), we have $P\subset F$. It follows that $F= P\oplus Q$, where $Q= F\cap H\in \frU $.

The orders of the composition sections of finite algebras in $\frU$ are strictly less
than $|A|$ by Corollary \ref{cc}. Therefore $\frU$  is a proper subvariety of $\frV$
and the algebra $F/\frU(F)$ is a $\frU$-free algebra of rank $n$.
On the one hand, since $Q\in \frU$, we have $\frU(F)\subset  P$  and so $\frU(Q)=\{ 0\}$.
At the same time, since $A\notin\frU$, we have that $\frU(A_i)$ is a non-zero ideal of $A_i$.
This ideal has to be equal to the whole of simple algebra $A_i$, for every $i$, whence
$\frU(F)\ge \frU(P) = P$. Thus, $\frU(F) =P$ and  $Q\cong F/P = F/\frU(F)$ is a
$\frU$-free algebra of rank $n$, as required. 

Now to prove part (i), we should consider the factor-algebras $F/I = (P\oplus Q)/I$.
Then by Lemma \ref{lDirectSimple} (ii), it suffices to prove that any ideal $I$ is the direct sum of an ideal of $P$ and an ideal of $Q$. 

If for some $j$, we have $A_jI\ne \{ 0\}$ or $IA_j\ne \{ 0\}$,
then the intersection $I\cap A_j$ is a non-zero ideal of $A_j$, that must be equal $A_j$
and so $A_j\subset I$. Otherwise $A_jI=IA_j= \{ 0\}$, and so the projection of $I$ to
$A_j$ has to be zero, since $A_j$ is a simple algebra with non-zero multiplication.
It follows that the projection of $I$ to $P$ is equal to the intersection $I\cap P$.
In this case we automatically obtain that the projection of $I$ to $Q$ equals $I\cap Q$,
and so $I= (I\cap P)\oplus (I\cap Q)$, as desired.

\endproof

\begin{Remark}\label{noapprox} If an algebra $A$ has a proper non-zero subalgebra, then the ideal $Q$ in Theorem \ref{tSDI} (ii) is not zero, but its image is zero under any homomorphism $F\to A$ since $A$ is simple and $A\notin\frU$.
Hence ($\var A$)-free algebras of arbitrary finite rank are not residually-$A$, in contrast to the statement
of Theorem \ref{cApprox}.
\end{Remark}

\begin{Remark}\label{rAnnihilators}
As a side remark, the statement of Theorem \ref{tSDI}  does not hold in the case of  algebras of infinite dimension from $\var A$. Generally speaking, any finite-dimensional ideal $I$ in a relatively free algebra $F = F(x_1,\ld)$ of infinite rank (over any field) belongs to the annihilator of this algebra, in particular, its product is zero.  
  
  Indeed, one may assume, that $I \subset F(x_1,\ld,x_n)$. Since the right annihilator of $I$ in the linear span $\la x_{n+1}, x_{n+2},\ld\ra$ has finite codimension, there exists a nontrivial linear combination of the elements $x_{n+1},\ld$, annihilating $I$. This element can be included as $y_{n+1}$ in some free generating set of the form $\{x_1,\ld,x_n, y_{n+1},\ld\}$. Hence for any $v(x_1,\ld,x_n)\in I$  we obtain a relation $v(x_1,\ld,x_n) y_{n+1} = 0$.  By Lemma \ref{lRelX}, this relation must be an identical relation in $F$, hence $IF=\{ 0\}$. Similarly, $FI = \{ 0\}$.
  \end{Remark}  
  
  We call a non-zero algebra {\it minimal} if it has no proper nonzero subalgebras. Clearly, every minimal algebra is simple, and we have the following
  
  \begin{Cor} \label{minalg} A finite algebra $A$ over a field $\bf F$ is minimal if and only
  if every finite algebra $B$ of $\var A$ is isomorphic to a direct power of $A$.
  \end{Cor}
  
  \proof A proper subalgebra of a finite algebra $A$ is not isomorphic to a
  direct power of $A$, which proves the part 'if'. Now if we assume that $A$ is
  minimal, then the variety $\frU$ generated by the proper subalgebras
  of $A$ is trivial. So the part `only if' follows from Theorem \ref{tSDI} (i).
  \endproof
  
 \begin{Remark}\label{infnot}
 If $A$ is a minimal finite algebra with nonzero product then the free
algebra $F$ of infinite rank in $\var A$ is not the direct sum of copies of $A $. Indeed, by Remark \ref{rAnnihilators}  $F$ has no nonzero finite ideals since by Corollary \ref{minalg} the annihilators of non-zero finite subalgebras of $F$ are trivial. 
\end{Remark}

  A variety $\frV$ of algebras over a field $\F$ is called {\it minimal} if it is nontrivial and contains no proper  nontrivial subvarieties. 
  
  \begin{Cor}\label{minvar} A locally finite variety $\frV$ is minimal if and only if
  it is generated by one finite minimal algebra.
  
  \end{Cor}
  
  \proof Let $\frV$ be generated by a finite minimal algebra $A$. Since $\frV$
  is locally finite, a non-zero subvariety $\frU$ of $\frV$ contains a non-zero finite
  algebra $B$, which by Corollary \ref{minalg} is a direct power of $A$, and therefore
  $\frU\supset \var B = \var A =\frV$. So $\frV$  has no  proper nontrivial subvarieties.
  
  Conversely, let $A$ be a non-zero algebra of minimal dimension in a locally
  finite variety $\frV$. Then $A$ is a minimal algebra and $\var A=\frV$ if $\frV$
  is minimal. This completes the proof of the corollary.
  
  \endproof
  
  The argument of the previous paragraph shows that every non-zero locally
  finite\footnote{Using Zorn's lemma, one can drop the condition `locally finite'} variety contains a minimal variety. However, we have
  
  \begin{Prop}\label{finset} 
The set of non-isomorphic minimal algebras in a locally finite variety $\frV$ is finite.  
 \end{Prop}
\begin{proof}
  A minimal algebra is generated by any non-zero element. Hence it is a homomorphic image 
  of the one-generated $\frV$-free algebra $F_1$. This observation completes the proof
   since the one-generated algebra $F_1$ must be finite in a locally finite variety. 
\end{proof}

 \begin{Example}\label{eMinVar} The first examples of minimal algebras have been given in
 Subsection \ref{ssFermat}. They are one-dimensional.
 
 For a two-dimensional example, let $\F=\ZZ_2$ and  $A$ be an algebra over $\F$ with basis $(e,f)$ and the multiplication table $ee=f$, $ef=0$, $fe = e$, $ff=e$.  Since $(e+f)(e+f)= f$, no square of a nonzero element $a$ is a multiple of $a$. Thus, $A$ has no proper subalgebras. 
 
 One can find infinite series of finite minimal algebras in \cite{A}.
 
 \end{Example} 
 
\begin{problem}\label{pbm_one} Does there exist a minimal infinite-dimensional algebra? 
\end{problem}

\begin{problem}\label{pbm_two} Can a locally finite variety contain infinitely many pairwise non-isomorphic finite simple algebras?
  \end{problem}
\begin{problem}\label{pbm_three} Does an infinite locally finite, simple algebra have
   infinitely many pairwise non-isomorphic  finite simple sections ?
\end{problem}
   
   An answer in the affirmative to the group-theoretical version of  Problem \ref{pbm_three} was
   given in [OK]. In the case of Lie algebras see \cite{BHS}.
   
    \subsection{Computation of some free products.}\label{sFPFGV} 
    
  The {\it free product} of two algebras $B$ and $C$ in a variety $\frV$ is an algebra $D\in \frV$ such
  that there exist homomorphisms $\beta:B\to D$ and $\gamma: C\to D$ satisfying the following conditions:
  \begin{enumerate}
    \item[$\mathrm{(i)}$] $D$ is generated by the subalgebras $\beta(B)$ and $\gamma(C)$,
    \item[$\mathrm{(ii)}$]  For any homomorphisms $\beta':B\to E$ and $\gamma':C\to E$, where $E\in \frV$, there exists a homomorphism $\delta: D\to E$ such that $\beta' =\delta\beta$ and $\gamma' =\delta\gamma$.
  \end{enumerate}

   The free product $D=B*C$ exists and unique up to isomorphism (\cite{ASM}, Section 12). It is finite if
   $B$ and $C$ are finite algebras in a locally finite variety $\frV$. Therefore one can view $B*C$
   as an algebra of maximal order in $\frV$ generated by two isomorphic copies of  $B$ and $C$. In this section, we consider an example of computation of the free product in a locally
   finite variety $\frV$. We may identify $B$ and $C$ with their isomorphic images in $D$.
   
   \begin{Lemma} \label{frprpowers} Let $\frV$ be a variety generated by a simple algebra $A$, $B$ and $C$  two finite direct powers of the algebra $A$. Then the free product $D=B*C$ in $\frV$ is also a
   direct power of $A$.
   \end{Lemma}
   
   \proof By Theorem \ref{tSDI} (i), $D$ is a direct sum of two ideals $P$ and $Q$, where $P$ is a
   direct power of $A$ and $Q$ belongs to a proper subvariety $\frU$ of $\var A$. Consider
   the projection of $B$ to $Q$. If the image of a direct factor $A'\cong A$ of $B$ is not zero,
   then this image is isomorphic to $A$, because $A$ is a simple algebra, but $A\notin \frU$, a contradiction.
   So $B$ and also $C$ are subalgebras of $P$, and $Q=\{ 0\}$ since $D$ is generated by $B$ and $C$.
   \endproof
   
   In the example below, we assume that the finite simple algebra $A$ has no nontrivial automorphisms.
   Obviously, this condition is  limiting. However, we will show in Section \ref{esGPNA} that almost every $n$-dimensional non-associative algebra over a finite field is simple and has trivial
   automorphism group.

\begin{Prop}\label{pFreeProduct}
Let $A$ be a finite simple   algebra without nontrivial automorphisms and $A^2\nez$. Assume
that an algebra $B$ (an algebra $C$) is isomorphic to the product of $k$  (resp., of $\ell$) copies of $A$.
Then the free  product $D=B*C$ in $\frV = \var A$ is isomorphic to the direct product of  $(k+1)(\ell +1) -1$ copies of $A$.
\end{Prop}
\begin{proof} We can write $B$ as $A(1)\oplus\dots\oplus A(k)$ with $A(i)\cong A$ for $i=1,\dots,k$ and
by Lemma \ref{frprpowers}, we have similar decomposition $D= A_1\oplus\dots\oplus A_{t}$. By Lemma \ref{Diagonal}, there are disjoint subsets $J_1,\dots,J_k$ of the set $J=\{1,\dots,t\}$ such that
$A(i)$ is the diagonal $\Delta_i$ of the sum $\sum_{j\in J_i} A_j$. It follows that the set $J$ can be partitioned into $k+1$ subsets $J_0, J_1,\ld,J_k$, where $J_0$ complements the union of all other subsets $J_i$ in $J$. Hence any element from $B$  has zero projection onto $A_m$, if $m\in J_0$.

Similar partition $J=J'_0\cup\ld\cup J'_{\ell}$ is defined by  the subalgebra $C$ of $D$. Let $J_{ii'}=J_i\cap J'_{i'}$. In this case, the projection of any element from $B$ or $C$ to the  sum $\sum_{s\in J_{ii'}} A_j$ belongs to the diagonal $\Delta_{ii'}$ of this sum. Here  we use a simple fact that the direct power of an algebra without nontrivial automorphisms has a unique diagonal.

 Hence the sum of the diagonals $\Delta_{ii'}$ for all disjoint subsets $J_{ii'}$ contains both $B$ and $C$, hence equals to the whole $D$. Moreover, the subset $J_{00}$ is empty since $B$ and $C$ generate
 $D$. Some of other subsets  $J_{ii'}$ could be empty too, but in any case the number of the subalgebras  $\Delta_{ii'}\cong A$ is at most the number of the subsets $J_{ii'}$ with $(i,i')\ne (0,0)$, and so $t\le (k+1)(l+1)-1$.

To obtain the converse inequality, it is sufficient to construct an algebra in $\var $A generated by the images of the algebras $B$ and $C$ and having dimension $(k\ell+k+\ell)\dim A$. Our preceding  argument suggests the following approach. Let us enumerate the copies $A_{ii'}$  of the algebra $A$ by two indexes, $i=0,\ld,k$, $i'=0,\ld,\ell$, $(i,i')\ne (0,0)$, and consider their direct sum $D$.
  
 We isomorphically map each the summand $A(i)$ of $B$, $(i=1,\ld,k)$ onto the diagonal of the sum $A_{i0}\oplus\dots\oplus A_{il}$. It is obvious, that all these maps extend to a monomorphism $B\to D$.  By mapping the summands of $C$ into the diagonals of the sum  $A_{0j}\oplus\cdots\oplus A_{kj}$, we obtain an embedding $C\to D$.
 
 The product of such two ``orthogonal'' diagonals for $i\ne 0$,  $j\ne 0$ is  $A_{ij}^2 = A_{ij}$ since $A_{ij}$ is a simple
 algebra with non-zero multiplication. Hence we have $A_{ij}\subset  D$ for $i,j\ne 0$. The diagonal of the sum $A_{i0}+\cdots+A_{il}$ is also in $B\subset  D$ for $i=1,\ld,k$. Similarly we get $\ell$ diagonals $A_{0j}+\cdots+A_{kj}$ in $D$. Consequently, $D$ contains at least $kl +k +l$ direct summands isomorphic
 to $A$, and so $t\ge (k+1)(l+1)-1$, as desired.
 
\end{proof}

The second part of the proof provides us with the construction of the free product $B*C$.
 By induction, we have
 \begin{Cor}\label{c1} If $B_1,\ld, B_s$ are direct sums of $k_1,\ld,k_s$, resp.,  copies of the
 algebra $A$ from Proposition \ref{pFreeProduct}, then the free product $B_1*\dots*B_s$ is a direct sum of  $t=(k_1+1)\dots(k_s+1) - 1$ copies of $A$.
 
 \end{Cor}
 
 In other words, we have obtained the precise upper bound for the dimensions of any algebra of the variety, which is generated by subalgebras $B_1,\ld, B_s$.
 
 \medskip
\begin{Example}\label{eFReeProd}
 For example, two copies of  the algebra $A$ from Proposition \ref{pFreeProduct} in any algebra $B\in \var A$ cannot generate a subalgebra greater than  $A\times A\times A$. 
 
 Another example is a two-dimensional algebra from Example \ref{eMinVar}. One only need to check that algebra has no notrivial automorphisms.

  Also, let $B$ be an algebra satisfying the identities
  \[
 xy=yx,\;(xy)z = x(yz),\mbox{ and }x^q=x\mbox{ where }q=|{\bf F}|.
  \]
If $B$ is generated by $s$ subalgebras, each isomorphic to a finite field $\bf F$,  then $B\cong \F^t$ where $t\le 2^s-1$, and this estimate is precise.
\end{Example} 

\subsection{Characteristically simple algebra.}\label{essCSA} The properties of characteristically simple and verbally simple finite groups are well-known. But a subalgebra  invariant with respect to all automorphisms of an algebra is not necessarily an ideal, as opposed to normal subgroups in Group Theory. So the case of algebras requires a special treatment. 
  
  Recall that a subalgebra $B$ of an algebra $A$ is called \emph{characteristic} if $B$ is invariant under all automorphisms of $A$. We say that an  algebra $A$ is \emph{strongly characteristically simple}, for short a CS-algebra, respectively, \emph{weakly characteristically simple}, for short CW-algebra, if  $A$ has no nonzero proper subalgebras, respectively, ideals, which are invariant under all of its automorphisms. We also say that $A$ is \emph{verbally simple} if it has no nonzero proper verbal ideals.
  
  A finite direct power $P=\underbrace{A\oplus\cdots\oplus A}_k$ of a simple algebra $A$ is both a CW-algebra and verbally
  simple by Lemma \ref{minalg} (i), since there are automorphisms permuting the direct factors
  of $P$. Conversely, let $B$ be a finite $CW$-algebra and $M$ be a minimal
  ideal of $B$. Then for any automorphism $\alpha$ of $B$ the image $\alpha(M)$ is a minimal ideal as well, and so the sum of such images over all $\alpha$ must be equal to $B$
  since the sum is a characteristic ideal of $B$. This sum is a direct some of
  some ideals isomorphic to $M$. Then any ideal of $M$  is an ideal of $B$,
  and the minimality of $M$ implies that $M$ is simple. So a finite algebra is CW (or verbally simple) if and only if it is a direct power of a finite simple algebra.
  
  Speaking about $CS$-algebras, first of all, any algebra with zero multiplication is CS. Also, a minimal algebra is a simple CS-algebra. As for nonsimple CS-algebras, we  prove the following.
  
\begin{Thm}\label{pcscw}
A finite nonsimple algebra $P$ with nonzero multiplication is a CS-algebra if and only if
 $P=\underbrace{A\oplus\cdots\oplus A}_k$, $k>1$, where $A$ is a simple $CS$-algebra which is not a minimal algebra without nontrivial automorphisms. 
\end{Thm}
  
 \begin{proof} Let $P$ be a finite CS-algebra. Then it is also a CW-algebra,
 and so  $P=A_1\oplus\cdots\oplus A_k$, where $A_i$ are isomorphic copies of
 a simple algebra $A$. Since $P^2\ne \{ 0\}$, we have that $A^2\ne \{ 0\}$. If $A$ has a proper characteristic subalgebra $B\ne \{ 0\}$ then the sum of the $k$ copies $B_i\subset  A_i$ is a proper characteristic subalgebra of $P$ which is true because any automorphism of $P$ permutes the factors $A_i$, as it follows from Lemma \ref{lDirectSimple} (i). Thus, $A$ must be a $CS$-algebra. 
 
 It is not possible that $A$ is minimal and has no non-trivial automorphisms. Indeed, then
 $P$ has a unique diagonal $D$, which is proper characteristic subalgebra of $P$,
 a contradiction.
 
 Conversely, assume that $A$ is a finite simple $CS$-algebra and
 $P=\underbrace{A\oplus\cdots\oplus A}_k$, where $k\ge 2$ and $A_i$ are isomorphic copies of
 $A$. If $k>1$ we also assume that $A$ is not a minimal algebra with trivial automorphism group. 
 
 Consider a non-zero characteristic subalgebra $B\subset  P$. It has to have isomorphic
 projections $B_i$ to each of the factors $A_i$. If $B_1$ is a proper subalgebra
 of $A_1$ then there is an automorphism $\alpha$ of $A_1$ such that $\alpha(B_1)\ne B_1$,
 because $A_1$ is a CS-algebra. It follows that $B\ne \beta(B)$, where $\beta$ acts
 on $A_1$ as $\alpha$ and acts identically on $A_2,\ld, A_k$. So $B$ is not
 a characteristic subgroup of $P$, a contradiction.
 
 Thus, $B$ projects surjectively onto every $A_i$.
 By Lemma \ref{minalg}, $B$ is itself isomorphic
 to a direct power of $A$. Obviously $B$ cannot be a characteristic subalgebra of $A$ if
 $B$ is just the  product of some subalgebras $A_i$ and $B\ne P$. So by Lemma \ref{Diagonal}, one of the
 direct factors $A'\cong A$ of $B$ is a diagonal of a the direct sum $A_{i_1}\oplus\dots\oplus A_{i_s}$, where $s\ge 2$, and other direct factors of $B$ have zero projections to each of
 $A_{i_1},\dots, A_{i_s}$. Since $B$ is characteristic in $P$ and the automorphisms
 of $P$ can arbitrarily permute the factors $A_i$, it follows that $B$ is just a diagonal of the whole $P$. However a diagonal is invariant under all automorphisms of $P$ iff it is unique, that
 is iff (since $k\ge 2$) the algebra $A$ has no non-trivial automorphisms. Since $A$ is CS, in this case, $A$ also has no non-zero proper subalgebras. For this exception,
 we obtain that $B=P$, and the theorem is proved.
 
 \end{proof}
 
 Note that a finite field $\bf F$, as an algebra over itself, is minimal and
 has no non-identical automorphisms. Therefore the product ${\bf F}\times {\bf F}$
 is not a CS-algebra; the diagonal is a characteristic subalgebra of it.
  In the associative case, $CS$ refers only to the algebras with zero multiplication and the field of coefficients (as an algebra over itself), that is, the condition looks too strong.
  The two-dimensional minimal algebra from Example \ref{eMinVar} is CS, but its direct square is
  not CS, because  this algebra has no non-trivial
  automorphisms. The following is an example of a two-dimensional simple CS-algebra admitting
  non-trivial automorphisms.
  
  \begin{example}\label{eVSA}
A very symmetric algebra is an algebra $A$ over $\ZZ_2$ with a basis $(e,f)$ and the multiplication table
\[
e^2=e,  f^2=f,  ef=fe =e+f.
\]
It satisfies the identities $xy=yx$,  $x^2=x$. It is not difficult to verify that $A$ is
simple and every permutation of non-zero elements of $A$, extended with $0\mapsto 0$, is
an automorphism of $A$. Therefore, by Theorem \ref{pcscw}, all finite direct powers of $A$ are CS-algebras. $A$ is not minimal, and so we raise the following
\end{example}

\begin{problem}\label{pbm_four} Does there exist a minimal algebra $A$ with nonzero multiplication of order at least three such that the automorphism group $\Aut A$ acts transitively on the set of
nonzero elements of $A$? 
\end{problem}

Such an algebra has to be simple and relatively free, answering  [14, Problem 2.89] by Ryabukhin and
Flori.

\subsection{Semidirect products and extremal algebras.}\label{essSDP}

An algebra $C$ is a {\it semidirect sum} of an ideal $A$ and a subalgebra $B$ if $C = A+B$
and $A\cap B = \{ 0\}$.
Semidirect sums are an important tool, which is used in the study of extensions of algebras. Since solvable algebras appear as extensions of smaller dimension, semidirect sums are essential for the study of solvable algebras, as well. They are also important in the study of projectivity and injectivity of algebras in locally finite varieties, see Sections \ref{inj}, \ref{ssProjective}.

Let $A$ be an ideal with zero multiplication (abelian ideal) of an algebra $C$ over a field $\F$  and $B=C/A$. We set $\overline{c}= c+A\in B$, for any $c\in C$. Then the vector space sum $S=B+ A$ becomes a semidirect sum $B\ltimes A$ if one sets
\begin{equation}\label{eqSemiDirect}
(\overline{c_1},a_1)(\overline{c_2},a_2)=(\overline{c_1c_2},a_1c_2+c_1a_2),\mbox{ where }a_1,a_2\in A,\: c_1,c_2\in C.
\end{equation}
 
One denotes $S$ by $B\ltimes A$. An interesting, yet not immediately obvious fact, is the following.

\begin{Prop}\label{epSemidirect} 
If $\frV=\var C$ then $S=B\ltimes A\in \frV$.
\end{Prop}

\begin{proof} It is enough to show that $S$ is isomorphic to a section of the Cartesian square of $C$. So we set $P=C\times C$ and $Q=A\times A\subset P$. Consider the diagonal subalgebras $\tC=\{(c,c)\,|\, c\in C\}$, and $\tA=\{(a,a)\,|\, a\in A\}$. Next, take $T=\tC+Q$. Then $\tA$ is an ideal in $T$ since $Q$ is an algebra with zero multiplication. So we set $R=T/\tA$. It  remains to prove that $S\cong R$.
We define $\vp:S\to R$ by setting
\[
\vp(\overline{c},a)= (c+a,c)+\tA, \mbox{ where }c\in C\mbox{ and }a\in A.
\]
This setting is well-defined because if we replace $c$ by $c+a'$ then on the right-hand side we will have 
\[
(c+a'+a,c+a')+\tA=(c+a,a)+(a',a')+\tA=(c+a,c)+\tA. 
\]
If $\vp(\bar{c},a)=0$ than $(c+a,c)\in\tA$ so  $a=0$ an $c\in A$, that is $\bar{c}=0$. So $\vp$ is injective, and since $\dim S = \dim C$, it is an isomorphism of vector spaces. Finally,
\begin{eqnarray*}
&&\vp((\ov{c_1},a_1)(\ov{c_2},a_2))=\vp((\ov{c_1c_2},a_1c_2+c_1a_2))=(c_1c_2+a_1c_2+c_1a_2,c_1c_2)+\tA\\
&&=((c_1+a_1,c_1)+\tA)((c_2+a_2,c_2)+\tA)=\vp(\ov{c_1},a_1)\vp(\ov{c_2},a_2).
\end{eqnarray*}
Thus, $\vp$ is an isomorphism of $S$ and $R$, and so $S\in\frV$.
\end{proof}

\begin{Remark}\label{AtA} Let $A'$ be the image of the ideal $A$ under the isomorphism $\vp$.
Then for any product $ac$ (or $ca$) the   image under $\vp$ is a product $\vp(a)b$ (resp.,
$b\vp(a)$ for some $b\in B$, as this is seen in the equation (\ref{eqSemiDirect}). Therefore, if
the ideal $A$ is minimal in $C$, the ideal $A'$ is minimal in $S$.

\end{Remark}

Generally, $\var S$ can be a proper subvariety in   $\var C$. For example, suppose an associative algebra $C$ is 1-generated by $a$ and given by one relation $a^3=0$. Then $\dim A=2$. Let $A$ be an ideal of $C$ generated by  $a^2$. In this case the semidirect sum $S=B+A$ in question is the sum of two one-dimensional ideals, with zero product. Thus $S$ satisfies $xy=0$ which is not an identity of $C$.

On the other  hand, it is natural to ask if an arbitrary algebra $C$ with ideal $A$ and factor-algebra $B$ can be embedded in a semidirect sum $S'=A'+ B$, where $A'\in \var A$  (in particular, has zero product if $A$ has zero product).  The embedding of this kind are widely used in the Group Theory and in Lie algebra and can be built using so called \emph{wreath products}. The analogues of wreath products in the varieties of linear algebras given by permutational identities had been suggested in \cite{Burg}. In particular, any   algebra $C$ of that kind with ideal $A$ and factor-algebra $B$ is isomorphic to a subalgebra $C'$ in the semidirect sum  $B\ltimes A'$ where $\var A'= \var A$,  $C'\cap A' = A$ и $A'+C' = A'+ B$.

In the remainder of this section we will study a class of algebras, which are close to critical algebras.

A finite algebra $A$ is called \emph{extremal} if the variety $\var A$ 
cannot be  generated by a set of algebras having smaller orders than $A$.

Obviously, extremal algebras are critical, so they are monolithic, and every locally finite variety is generated by its extremal algebras. 

So we assume that $C$ is an extremal algebra with an abelian monolith $A$ and consider the semidirect sum $S=B\ltimes A$ provided in Proposition \ref{epSemidirect}, where $B=C/A$.

\begin{Thm}\label{pExtremal} If $C$ is an extremal algebra with an abelian monolith $A$ and $\var S = \var C$, then  $A$ has a semidirect complement $D$ in $C$, that is, $C=D\ltimes A$.
\end{Thm} 

\begin{proof} Let $\frV=\var C$. Since $\var S = \var C = \frV $ and $|S|=|C|$, it follows that $S$ is also an extremal algebra. The ideal $A$, viewed also as an ideal of $S$, remains minimal by Remark \ref{AtA}.

Now let $C$ be $n$-generated for some $n$ and $F$ a $\frV$-free algebra of rank $n$. Since $F\in\var S$, by Theorem \ref{etBirkhoof_Estimate} and Lemma \ref{minrep}, one can represent $F$ as a subdirect product in the direct product $P$ of monolithic sections of $S$. Some of the factors of $P$ must be isomorphic to $S$ itself, for otherwise, $P$ hence $F$ and $S$, would belong to a variety generated by algebras of smaller order than $S$, contradicting the extremality of $S$. 

Let us choose a minimal representation for $F$. Mark the monoliths of these factors $A_1,\ld, A_t$. By Lemma \ref{minrep} (iv), $F$ contains each of the monoliths, hence their direct sum $N$.

Since each $S_i$ is a semidirect sum with kernel $A_i$, also $P$ is a semidirect sum with kernel $N$, $P = N + T$, for some subalgebra $T\subset P$. Note that $T$ belongs to a proper subvariety $\frU\subset\frV$, which is generated by other sections of $P$ and the algebra  $S/A=B$.

Now since $F$ contains $N$, its natural projection to $T$ coincides with $Q=F\cap T$. It follows that $F$ is a semidirect sum $F = Q \ltimes N$, where $Q\in \frU$. Also, all $A_i$ remain to be minimal ideals in $F$, since $F$ project surjectively onto each $S_i$, $i=1,\ld,t$.

Let us remind that there is an epimorphism $\vp: F\twoheadrightarrow C$.  Since $F/N \in \frU$ and $C\notin \frU$, the ideal $N$ cannot belong to the kernel of $\vp$. Hence, $\vp(N)$ must contain the monolith $A$ of $C$. Since $N$ is the sum of minimal ideals in $F$, it must be mapped onto the sum $\vp(N)$ of minimal ideals in $C$. There is only one such ideal in $C$, and so  $\vp(N) =A$.

It follows that $C = \vp(F) = \vp(N) +\vp(Q)= A +Q'$, где $Q' =\vp(Q)\in \frU$. Thus, $Q'\ne C$, and then  $Q'\cap A \ne A$. Since $Q'\cap A$ is an ideal in $Q'$, and $A$ is abelian, it follows that $Q'\cap A$ is an ideal in $C$. But $A$ is a monolith of $C$ and so $Q'\cap A=\{ 0\}$. As a result, we split $C$ over $A$ by $Q'$,  $C = Q'\ltimes A$, as claimed.
\end{proof}

Theorem \ref{pExtremal} works for all critical algebras $C$ we are familiar with since all of them
are extremal. Thus, we have

\begin{problem}\label{pbCritExtrem} Can a variety $\var C$, where $C$ is a finite critical   algebra, be generated by a set of algebras whose orders are less than the order of $A$?  In other words, is any critical algebra extremal?
\end{problem}

 \section{Miscellaneous properties of locally finite varieties}\label{sSPIinLFV}
 
 \subsection{Socle height of algebras in finitely generated varieties.}\label{sSHA}
The socle $S_1=\soc(A)$ of a finite algebra (or ring) $A$ is the sum of all its minimal ideals (hence, the direct sum of some of them).  We define the second socle $S_2$,  by saying that $S_2/S_1=\soc(A/S_1)$, and so on. As a result, we obtain the ascending series of $A$, ending in $A$. We call this series a socle series of $A$. The length of the socle series of $A$ is called the \emph{socle height} of $A$. An equivalent definition of the socle height is the length of the shortest  ascending series of ideals: 
\[
\{ 0\}=J_0\subset J_1\subset J_2\subset\cdots,\mbox{where for all }k=1,2,\ld,\;\;  J_k/J_{k-1}\subset\soc(A/J_{k-1}).
\]

\begin{Thm}\label{tSHA} The socle heights of finite algebras in a finitely generated variety  are uniformly bounded.
\end{Thm}

\begin{proof} Let $\V=\var A$, $A$ a finite algebra. Since under an epimorphism of algebras $P\to Q$, the socle $\soc(P)$ is mapped inside the socle $\soc(Q)$, it is enough to prove that the socle heights of free algebras $F_n(\V)$, $n=1, 2,\ld$, are uniformly bounded.

  Let us remind that a locally finite variety is generated by its monolithic algebras, in particular, $\var A$  is generated by the set $\cS$ of all monolithic sections of $A$. We consider the partially ordered set $\X$ of subsets $\cT\subset\cS$ which contains all monolithic sections of its members. For each such $\cT$, we consider the variety $\var\cT$. There are only finitely many varieties $\var\cT$, and to prove our theorem, we can apply induction by $|\cT|$. The basis of induction is obvious, when $\cT$ is empty. Let $\cT$ be nonempty. We may assume that $\var\cT$ is not generated by proper subsets $\cT'\subset \cT$. 

Using the minimal representation, see Lemma \ref{minrep}, we represent the free algebra $G_n$ of $\var\cT$ as the subdirect product of algebras in $\cT$. Note that in $\cT$, there is an algebra $B$, which is not a proper section of other algebras in $\cT$. Moreover, one may assume, that $B$ cannot be removed from $\cT$, without decreasing the variety generated by $\cT$. Therefore the set $\cT'=\cT\backslash\{B\}\in\X$ and the socle lengths of finite algebras in $\frU=\var\cT'$ are uniformly bounded.

As a result, if $n$  is not less than the number of generators in $B$, then $B$ must be a factor of the direct product $D$ containing the minimal representation of the $\frV$-free algebra $G_n$. This follows because $B$ is a homomorphic image of $G_n$ and $B\notin \frU$. By Lemma \ref{minrep} (iv), the monolith $М$ of $B$ is a minimal ideal of $G_n$. 

Thus, we can see that for all direct factors $B_i$ of $D$, isomorphic to $B$, their monoliths $M_i$ are contained in the socle of $G_n$. Let $N$ be the sum of all $M_i$. If we form $D/N$, then by the choice of $\cT'$ and $B$, it follows that $D/N$ belongs to the variety $\frU$. Hence also $G_n/N\in\frU$. By induction, all finite algebras in $\frU$ have socle height at most $h$, for some integer $h$. This  also applies to  $G_n/N$. Since $N$ belongs to the socle of an algebra  $G_n$, the socle height of $G_n$ is at most $h+1$ for all $n$, completing the proof of the theorem.
\end{proof}

Using this theorem, we can define the \emph{socle height of a finitely generated variety} $\frV$ as maximum of socle heights of finite algebras in $\frV$.

\medskip

\begin{Example}\label{eSocle} Let be $A$ a nilpotent algebra. On the one hand, since every one-dimensional subspace of the annihilator $I(A)$ is a minimal ideal in $A$, we have $I(A)\subset\soc(A)$. On the other hand, by Lemma \ref{lNilpSimpleProp} (i), every minimal ideal of $A$ is contained in  $\soc(A)$. Hence $\soc(A)=I(A)$. Using induction, we conclude that the socle series of $A$ is the annihilator series of $A$, and so \emph{the socle height of $A$ is equal to the nilpotent depth of $A$} (see Subsection \ref{esssNS}).
\end{Example}

\begin{Remark}\label{nilvar} It follows from Theorem \ref{tSHA} and Example \ref{eSocle} that finite
nilpotent algebras in a finitely generated variety $\frV$ have uniformly bounded nilpotent
depths (and by Proposition \ref{pNcNd}, bounded nilpotent classes).  Since  $\frV$ is a locally finite variety, the same upper bound holds for all nilpotent algebras from $\frV$. Being a subclass of $\V$ closed under all three Birkhoff operations, \emph{the class of all nilpotent algebras in a finitely generated variety $\frV$ is a subvariety of $\frV$}.

\end{Remark}

\subsection{Socle height 1 and inherently semisimple algebras.}\label{essMASH}


We call a finite-dimensional algebra $A$ \emph{inherently semisimple} if any subalgebra of $A$ (including $A$ itself) is the direct product of simple algebras. Any minimal algebra is inherently semisimple but there are many more, even simple algebras, with this property. For example, any finite field extension $\bf E$ of the base field $\bf F$, is a simple inherently semisimple algebra. Another example is the simple algebra having three one-dimensional subalgebras from Example \ref{eVSA}. Moreover, almost all $n$-dimensional algebras have this property because, as we will show in Section \ref{esGPNA}, being simple and having no subalgebras of dimension greater than 1 is a generic property.

Obviously, any subalgebra of an inherently semisimple algebra is inherently semisimple.

\begin{Remark} Inherently semisimple algebras need not be semisimple in the traditional sense of associative or Lie algebras because they include algebras with zero multiplication.
\end{Remark}

\begin{Thm}\label{epSSSA}
The following properties of a non-zero locally finite variety $\frV$ are equivalent.
\begin{enumerate}
  \item [$\mathrm{(i)}$] In any finite algebra $A\in\frV$, given an ideal $I$, there is an ideal $J$ such that $A=I\oplus J$;
  \item [$\mathrm{(ii)}$] Any finite $A\in\frV$ is the direct product of simple algebras;
  \item [$\mathrm{(iii)}$] The socle height of $\frV$ equals 1;
  \item [$\mathrm{(iv)}$] $\frV$ is generated by a family of inherently semisimple algebras.
  \item[$\mathrm{(v)}$] All finite algebras in $\frV$ are inherently semisimple.
  \item[$\mathrm{(vi)}$] In any algebra $A\in\frV$, given a finite ideal $I$, there is an ideal $J$ such that $A=I\oplus J$;.
  \end{enumerate}
\end{Thm}
\begin{proof}
$\mathrm{(vi)}\to\mathrm{(i)}\to\mathrm{(ii)}\to \mathrm{(iii)}$ are obvious.

$\mathrm{(ii)}\to \mathrm{(i)}$ follows from Lemma \ref{lDirectSimple} (i).

$\mathrm{(iii)}\to \mathrm{(ii)}$  Any finite algebra $A\in\frV$ is the direct sum of its minimal ideals. These ideals must be simple themselves because their ideals are the ideals of the whole algebra.

$\mathrm{(ii)}\to \mathrm{(v)}\to \mathrm{(iv)}$ are obvious.


$\mathrm{(iv)}\to \mathrm{(ii)}$ By Lemma \ref{lDirectSimple} (ii), it is enough to prove that any  $\frV$-free algebra $F_n$ is the direct product of simple algebras. By Lemma \ref{minalg}, $F_n$ is contained as a subalgebra in the finite direct product of inherently semisimple algebras. Because subalgebras of inherently semisimple algebras are inherently semisimple themselves, this embedding of $F_n$ can be subdirect in the direct product of simple algebras in $\frV$. But by Lemma \ref{lDirectSimple},  a subdirect product of finitely many simple algebras is isomorphic to the direct product of some of them.

$\mathrm{(ii)}, \mathrm{(i)}\to \mathrm{(vi)}$ If $I$ is a simple direct factor of a finite algebra $A$, then the annihilator  $J$  of $I$ it is either $A$ or the direct complement of $I$. In both cases it is an ideal of $A$. Hence the same is true if an ideal $I$ is the direct product of several simple ideals of $A$. Then the condition (ii) and Lemma \ref{lDirectSimple} (i) imply that the annihilator of any ideal $I$ of a finite algebra $A\in\frV$ is an ideal of $A$. Since the variety $\frV$ is locally
finite, it follows that an annihilator of a finite ideal in arbitrary algebra $A\in\frV$ is an ideal.

By the same Lemma, if $I$ is a simple ideal of a finite algebra $A\in\frV$ and $I^2\ne \{ 0\}$, then
$A$ is the direct sum of $I$ and the annihilator of $I$ in $A$, and if $I^2=\{ 0\}$, then $I\cap A^2=\{ 0\}$. Since $\frV$ is locally finite, the same alternative holds for a finite simple ideal in any algebra $A\in\frV$. In both cases, we have a complementary ideal to $I$ in $A$. Indeed, if $I\cap A^2=\{ 0\}$, then the vector apace $I+A^2/A^2$ has a direct complement $J/A^2$ in the factor-algebra $A/A^2$, and the ideal $J$ is the direct complement of $I$ in $A$.

Let now a finite ideal $I$ of an algebra $A\in\frV$ be the direct sum of simple ideals $I_1,\dots, I_k$. Then by property (i), $I_1$ is an ideal in every finite subalgebra of $A$. Since $A$ is a locally finite algebra $I_1$ is an ideal of the whole algebra $A$. 
 
For a simple ideal $I_1$, we have already obtained a complementary ideal: $A= I_1\oplus J$. Now
the intersection $I\cap J$ is an ideal of $J$ whose order is less than $|I|$. By induction on the order, we have $J=(I\cap J)\oplus E$ for some ideal $E$, whence $A = I\oplus E$, as required.

\end{proof}

\begin{Prop}\label{clMinSub}
Let $\frV$ be a locally finite variety of socle height 1. Then all subideals of any algebra $A\in\frV$ are ideals in $A$.
\end{Prop}
\begin{proof}
 
If $A$ is finite then by  Theorem \ref{epSSSA} (ii), the algebra $A$ is the direct product
 of simple algebras. It follows from the description of the ideals in $A$ provided by Lemma \ref{lDirectSimple} (i),
 that an ideal $J$ of an ideal $I$ of the algebra $A$ is itself an ideal of $A$. If this property fails
 for an ideal $I$ and a subideal $J\subset  I$ of an infinite algebra, then there is a finitely generated
 subalgebra $B\subset  A$ such that the subideal $B\cap J$ is not an ideal in $B$. Since  $\frV$ is locally finite variety, we have that $B$ is a finite
 algebra, a contradiction. The proposition is proved.
 
\end{proof}

\subsection{Injective algebras in locally finite varieties.}\label{inj}

\begin{Lemma} Let $\frV$ be a variety of algebras generated by a finite algebra $A$.
Then the order of the monolith $M$ of any monolithic algebra $B\in\frV$ does not exceed the
maximum $m$ of the orders of all chief sections of $A$.

\end{Lemma}\label{chieffact}

\proof The following is similar to the `chief centralizer argument' used for the groups in \cite{KN} (see also \cite[Section 5.2]{HN}). If $B$ is finite, this follows from Corollary \ref{cc}. So we need to deal with the case where $M$ is a monolith of an infinite-dimensional algebra $B$. Assume that $M$ contains $m$ different non-zero elements $b_1,\ld, b_m$. Then the differences $b_{ij} = b_i-b_j$ belong to $M$ and are also non-zero for all $i>j$. Since $M$ is minimal, it is generated, as an ideal, by each of these $ \frac{m(m+1)}{2}$ elements in $B$. Let $S$ be the set of these elements. Since one
needs finitely many element of $B$, to show that one of the element of $S$ belongs to an ideal generated
by another element of $S$, there is a finitely generated, and so finite, subalgebra $C\subset  B$
such that the 'ideal closure' of any element from $S$ in $C$ contains $S$. So there is
a chief section $L/K$ of $C$ such that $S\subset L$ and $S\cap K=\varnothing$. Since $b_i-b_j\notin K$
for $i\ne j$, the non-zero cosets $b_1+K,\dots, b_m+K$ are pairwise different, and so $|L/K|>m$, a contradiction. 

\endproof 

\begin{Prop} \label{resi}If a finitely generated variety $\frV= \var A$ has socle height $1$,
then every algebra $B\in \frV$ is residually finite.
\end{Prop}

\proof One may assume that the algebra $B$ is monolithic. By Lemma \ref{chieffact}, the
monolith $M$ of $B$ is finite. By Theorem \ref{epSSSA}, $B = M\oplus J$ for some ideal $J$
of $B$, whence $J=\{ 0\}$ since $B$ is monolithic. Thus $|B|=|M| < \infty$, which proves the proposition.

\endproof

Recall that a subalgebra $A$ is a \emph{retract} in an algebra $B$ if there is an ideal $I$ in $A$ such that $B=A+I$ and $A\cap I = \{ 0\}$, that is, this sum is semidirect.
An algebra $A$ is called \emph{injective} in $\frV$  if it is a retract in any  algebra of $\frV$ where it is contained as a subalgebra. 

\begin{Thm}\label{epRetracts}
The following properties of a locally finite variety $\frV$ are equivalent.
\begin{enumerate}
  \item [$\mathrm{(i)}$]  Any finite subalgebra $A$ of a finite  algebra $B\in\frV$ is a retract.
  \item [$\mathrm{(ii)}$] $\frV$ is generated by a finite set $\cA$ of finite minimal algebras.
  \item [$\mathrm{(iii)}$] Every finite algebra is injective in $\frV$.
  \end{enumerate}
\end{Thm} 
\proof $\mathrm{(i)}\to \mathrm{(ii)}$. The property (i) implies, that every ideal $A$ of a
finite algebra $B\in\frV$ can be directly complemented in $B$ with an ideal $C$. It follows that
$B$ is a direct product of finite simple algebras, and so the variety $\frV$ is generated by
a set of finite simple algebras. A proper non-zero subalgebra of a simple algebra cannot be a
retract. Therefore, all finite simple algebras in $\frV$ are minimal. We obtain property(ii) since
by Proposition \ref{finset}, the set of pairwise non-isomorphic minimal algebras of a locally
finite variety is finite.

 $\mathrm{(ii)}\to \mathrm{(iii)}$. Let $A$ be a finite subalgebra of an algebra  $B\in\frV$.
 Since $\frV$ has socle height $1$ by Theorem \ref{epSSSA}, and it is finitely generated by property (ii),  the algebra
 $B$ is residually finite by Proposition \ref{resi}. It follows that $B$ has an ideal $I$ of
 finite index such that $A\cap I=\{ 0\}$, because $A$ is a finite subalgebra.
 
 It suffices to prove that the factor-algebra $A+I/I$ is a retract in the finite algebra $B/I$,
 because if the ideal $J/I$ complements $A+I/I$ in $B/I$, then the ideal $J$ complements
 $A$ in $B$. So one may assume that $B$ is a finite algebra.  By Theorem \ref{epSSSA}, $B$ is the direct sum of  simple minimal algebras.
 
 So let $B=S_1\oplus\dots\oplus S_k$, where the simple ideals $S_i$ are minimal algebras.
 If $A\cap S_1=\{ 0\}$, then by induction the factor-algebra $A+S_1/S_1$ can be complemented
 with an ideal $J/S_1$ in $B/S_1$, whence $J$ is a complement of $A$ in $B$.
 Otherwise $A\cap S_1=S_1$ since $S_1$ is a minimal algebra. Again, by induction, there
 is a complement $J$ of $A\cap (S_2\oplus\dots\oplus S_k)$ in $S_2\oplus\dots\oplus S_k$,
 which is also a complement of $A$ in $B$. We have proved that the subalgebra $A$ is a retract of $B$.
 
 $\mathrm{(iii)}\to \mathrm{(i)}$ follows from the definitions.
 
\endproof

\subsection{Projective algebras in finitely generated varieties.}\label{ssProjective}

An algebra $B$ is called \emph{projective} in a variety  $\frV$ if it is contained as a retract in a $\frV$-free algebra. In particular, $\frV$-free algebras are projective. The equivalent definition
says that $B$ is projective, if for any epimorphism $\alpha: A\to B$, where $A\in\frV$, the is a
homomorphism $\beta: B\to A$ such that $\alpha\beta = \id_B$.
 
Given an algebra $B$, we say that an ideal $I$ of $B$ is \emph{semidirectly complementable} in $B$ if there is a subalgebra $C$ in $B$ such that $B=C\oplus I$.  We call an algebra $A$ \emph{nice} if  any ideal $I$ of any subalgebra $B\subset A$ is directly complementable in $B$. It follows
from this definition, that all sections of a nice algebra are nice. For example, the associative algebra of $2\times 2$ matrices, any $2$-dimensional Lie algebra and any $3$-dimensional simple Lie algebra are nice, while nilpotent algebras of class $c\ge 2$ are not.

\begin{Thm}\label{tBenign} Let a variety $\frV$ be generated by a finite family ${\cal A} = \{A_i\,|\,i=1,\ld,m\}$ of finite algebras $A_i$. Then the following are equivalent.
\begin{enumerate}
  \item[$\mathrm{(i)}$] All finite algebras of $\frV$ are projective in $\frV$;
 \item[$\mathrm{(ii)}$] All finite algebras of $\frV$ are nice;
  \item[$\mathrm{(iii)}$] All algebras of $\cA$ are nice.
\end{enumerate}
\end{Thm}

\begin{proof}
$\mathrm{(i)\to(ii)}$ If a finite algebra $A\in\frV$ is not nice then there is a subalgebra $B$ in $A$ in which there is an ideal $I$ without a complement in $B$. If $C=B/I$ then $C$ is not projective.

$\mathrm{(ii)\to(iii)}$ This is a trivial implication.

$\mathrm{(iii)\to(i)}$
Suppose all algebras $A_i\in\cA$ are nice.
To proceed, we need an auxiliary result.
\begin{Lemma}\label{lBenProj}
Let $B$ be a finite algebra, $B=N\oplus D$, where $N$ is the direct sum of minimal ideals of $B$ and $D$ a subalgebra. Suppose that every ideal of $D$ is semidirectly complementable in $D$.  Then each ideal $L$ of $B$ is semidirectly complementable in $B$.
\end{Lemma}
\begin{proof}
First, any ideal  $J\subset  N$ of $B$ is directly complementable by an ideal of $B$ inside $N$.
Indeed, if $J\ne N$, then there is a minimal ideal $M$ of $B$ from the direct decomposition of $N$ such
that $J\cap M=\{ 0\}$. Using  induction on the codimension of $J$, one may assume that the ideal
$J\oplus M$ of $B$ is complementable in $N$, hence so is $J$.

Returning to $L$, by the hypotheses, the ideal $(L+N)/N$ in $B/N\cong D$ is semidirectly complementable by a subalgebra $C/N$, where $C = N \oplus (D\cap C)$ is a semidirect sum. Set $E=D\cap  C$. Then,
\begin{equation}\label{eCAP}
 E\cap(L+N) = E\cap (C\cap (L+N))= E\cap N \ez.
\end{equation} 

So we have $B = (L+N) \oplus E$. Next we can write $N$ as the direct sum of ideals $L\cap N$ and $K$ of $B$, as noted above. Hence, $B = L +(K+E)$. If  $x\in L\cap(K+E)$ then $x = y+z$ where $y\in K$, $z\in E$, so that $z \in L+N$, and then by (\ref{eCAP}) we should have $z=0$. In this case, using $L\cap K=\{ 0\}$, we derive also $x=y=0$. As a result, the sum $L +(K+E)$ is semidirect, completing the proof of the lemma.
\end{proof}
Now we return to the proof of Theorem \ref{tBenign}. Let $A$ be a finite algebra in $\frV$. We have $A\cong F/J$ where $F=F_n(\frV)$, for some natural $n$. Since $\frV$ is locally finite,  $F$ is finite. So it is sufficient to prove that each ideal in each finite free, hence also every finite algebra in $\frV$, has semidirect complement.

As noted in Section \ref{esGF}, every finite algebra belongs to the variety generated by its monolithic sections.  We also have noted above that all sections of nice algebras are nice. So we may assume that all algebras $A_i$ are monolithic. Even more: we may assume that the set $\cA$ is closed under taking monolithic sections of its members.

Clearly there are only finitely many subsets with these properties in $\cA$. Accordingly, we have only finitely many subvarieties generated by these subsets. Let us consider only minimal closed subsets $\cal B$, i.e. such that the varieties generated by such $\cB$ are not generated by proper closed subsets of $\cB$.  This allows us to proceed by induction and assume that the desired property of ideals in free algebras of a variety $\frV$ already holds for the ideals of free algebras of any proper subvariety $\frV'$ of $\frV$ generated
by a closed system of monolithic sections of algebras in $\cA$ (in other words these ideals are semidirectly complementable in the free algebras $F_n(\frV')$). The basis of induction is an evident  case where  $\cal A$ is empty and the variety contains only zero algebra.

So, let $\frU$ be a nonzero variety generated by one of minimal subset $\cal A'$ of  $\cal A$. By Birkhoff's construction, a finitely generated free algebra  $F$ of $\frU$ is a subdirect product of a finite direct product $P$ of algebras $A'_i\in\cal A'$.

As usual, we may assume this emebedding minimal (see Section \ref{esMinRep}) and so by Lemma \ref{minrep} (iv), the monoliths $N_i$ of the factors $A'_i$ are minimal ideals in $F$. Since the factors $A'_i$ are nice, we have semidirect factorizations $A'_i = N_i+S_i$ and $P=N+S$, where $N$ is the product of monoliths $N_i$, while $S$ is the product of   $S_i$. Since $F$ contains $N$, we can say that the whole algebra $F$ is  a semidirect product of an ideal $N$  and $D\subset  S$.

Notice that $S$ and $D$ belong to the variety $\frW$ generated by proper factor-algebras of the algebras $A'_i$. So, $\frW$ is generated by a proper closed subset  ${\cal A}''$ of ${\cal A}'$.  By induction, the ideals of $D$ are semidirectly complementable in $D$.  Now we can apply Lemma \ref{lBenProj} to $F$, proving that each ideal of $F$ is semidirectly complementable. Applying induction, we complete the proof.
\end{proof}

\begin{problem}\label{eFinProj} Let $\frV$ be a locally finite variety of algebras  over a field.
Is it true, that all finite projective  algebras in $\frV$ are $\frV$-free if and only if  $\frV$ is locally nilpotent?
 \end{problem}

\subsection{Anti-Schreier varieties.}\label{ssAntiSchreier}

All finite algebras in the varieties considered in Subsection \ref{ssProjective} are ebmeddable 
in free algebras of those varieties. So one can modify Condition (i) in Theorem \ref{tBenign} by requiring that \emph{all} algebras in a variety $\frV$ can be embedded in free algebras. We can call such varieties \emph{anti-Schreier}. This contrasts with once popular \emph{Schreier} varieties where all subalgebras of free algebras were free. Similar to Schreier varieties, Anti-Schreier varieties are scarce. 

\begin{Thm}\label{tEmbedToFree}
Let $\frV$ be a variety of algebras over any field. If every algebra $A\in\frV$ can be embedded in a $\frV$-free algebra then $\frV$ is the trivial variety.
\end{Thm}

\begin{proof} Let us choose a cardinality  $\aleph$ which is greater than the cardinality  $\aleph'$ of a free algebra of countable rank in $\frV$. Let $F_{\alpha}$ be a $\frV$-free algebra of rank  $\aleph$ with free generators $x_{\alpha,\beta}$, $\alpha \in I$, $\beta \in J$, $|I|=|J|= \aleph $.  Let us denote by  $H$ the direct sum 
\[
H=\bigoplus_{\alpha \in I} F_{\alpha}.
\]
Assume that $H$ can be isomorphically  embedded in a $\frV$-algebra $F$ with a free generating set $Y$: $F = F(Y)$.

We denote by $G$ the subgroup of $\Aut F$, permuting the free generators. It is obvious that every element of $F$ can be mapped by an automorphism from $G$ into an element of a fixed subalgebra $F(Z)$, with a countable set of free generators $Z\subset Y$. As a results, the cardinality of the set of orbits of the action of the group $G$ on $F$ does not exceed $\aleph'$.

We fix one of these $G$-orbits $C$ and for $\alpha\in I$ define a subset  $J'_{\alpha}\subset J$ as the set of
all indices $\beta$ such $x_{\alpha,\beta}\in C$. Then, by definition, the subset $I'= I'(C)\subset I$  consists of all indices  $\alpha\in I$ such that  the cardinality of $ J'_{\alpha}$ is $\aleph$.

Assume that the cardinality of $I'(C)$ is less than $\aleph$ for every orbit $C$. Since the number of $G$-orbits is at most $\aleph'$, the union $L$ of all subsets $I'(C)$ has cardinality less than $\aleph$, while for $\alpha\not\in L$ all subsets of  $J'_{\alpha}(C)$ do not cover the set $\{x_{\alpha,\beta}\}_{\beta\in J}$. In this case, some element $x_{\alpha,\beta}$ does not belong to any orbit, a contradiction.
It then follows that there exists an orbit $C$ for which the set $I'= I'(C)$ has cardinality $\aleph$. 

There is an element $v = v(z_1,\dots, z_m)\in F(Z)$ such that for each element  $x_{\alpha,\beta}$ with index  $(\alpha,\beta)\in \bigcup_{\alpha\in I'}(\alpha, J'_{\alpha})=K$ there is an automorphism in $G$ mapping this element to $v$. We will choose some subsets of $K$ to obtain additional properties of elements $x_{\alpha,\beta}$ with indices $(\alpha,\beta)$ from these subsets.

For $(\alpha,\beta)\in K$, the element $v(z_1,\dots,z_m)$ is the image of $x_{\alpha,\beta}$ under an automorphism
permuting the generators of $Y$. So we have $x_{\alpha,\beta} = v(y_{i_1},\dots,
y_{i_m})$, where the variables $y_{i_1},\dots,
y_{i_m}$ are different elements of $Y$, depending on $(\alpha,\beta)$. If we fix $\alpha\in I_{\alpha},$
then in the infinite set $J_{\alpha}$, one can choose an infinite subset $J_{\alpha}'$ such
that either all variables $y_{i_1}=y_{i_1}(\beta)$ are different for different $\beta\in J_{\alpha}'$,
or all of them are equal. Let us denote this set of variables (with one or infinitely many elements) by $Y_{\alpha,1}$. For the same reason, one may assume that the same property holds for
the variables in $v$ with numbers $2,\ld,m$, which give sets $Y_{\alpha,i}$ for $i=2,\dots,m$. Keeping the same notation $J_{\alpha}'$, we obtain the subset $K'= \bigcup_{\alpha\in I'}(\alpha,J'_{\alpha})\subset K$.

Choosing an infinite subset  $I'' \subset I'$, one may assume, without loss of generality, that the subset $Y_{\alpha,i}$ is infinite if $i=1,\dots,k$ for $\alpha\in I''$ and have only one variable $y_i = y_i(\alpha)$ if $\alpha\in I''$ and $i>k$. Then one can choose two different $\alpha $ and $\alpha'$ in $I''$ such that $y_i(\alpha)=y_i(\alpha')$ for some values of $i>k$, say for $i=k+1,\dots, \ell$, and $y_i(\alpha)\ne y_j(\alpha')$ for $i,j>\ell$.

Now the infinite sets $J'_{\alpha}$ and $J'_{\alpha'}$ can be replaced with their infinite
subsets $J''_{\alpha}$ and $J''_{\alpha'}$  such that the corresponding  $Y''_{\alpha,i}$ and $Y''_{\alpha',i}$ have empty intersection for $i\le k$ and both do not contain the one-element sets $Y_{\alpha,i}$ $Y_{\alpha',i}$ for $i>k$.

Thus, one can choose two different $\beta$ and $\beta'$ from $J''_{\alpha}$ so that
\[
x_{\alpha,\beta}= v(y_{11},\dots,y_{1k}, y'_{k+1},\dots,y'_m),\; x_{\alpha,\beta'}= v(y_{21},\dots,y_{2k}, y'_{k+1},\dots,y'_m),
\]
where all $2k$ variables $y_{ij}$ are pairwise different and also different from the variables $y_k'$.
Similarly for  $\alpha'$ and some $\beta\in J''_{\alpha'}$, we choose 
\[
x_{\alpha',\beta''}= v(y_{31},\dots,y_{3k}, y''_{k+1},\dots,y''_m),
\]
where each $y_{3j}$ is different from all other letters in the expression for $x_{\alpha,\beta}$, $x_{\alpha,\beta'}$, and $x_{\alpha',\beta''}$, while  $y''_j=y'_j$, if $k+1\le j\le l$ and  $y''_j$ is different from all other letters in the expression for $x_{\alpha,\beta}$, $x_{\alpha,\beta'}$, and $x_{\alpha',\beta''}$, if $j>\ell$.

 This choice of $y_{ij}$ and $y_k$ makes the following map well defined:
\begin{eqnarray*}
  &&y_{11}\mapsto y_{11},\dots, y_{1k}\mapsto y_{1k},\, y'_{k+1}\mapsto y'_{k+1},\dots, y'_{m}\mapsto y'_{m}\\
   &&y_{31}\mapsto y_{21},\ld, y_{3k}\mapsto y_{2k},\, y''_{k+1}\mapsto y'_{k+1},\dots, y''_{m}\mapsto y'_{m}.
   \end{eqnarray*}  
Since all $y_{3i}$ are unique, while  $y''_{i}$ are either unique or coincide with $y'_{i}$, it follows that the above map is well defined.

The above map of free generators extends to an endomorphism of  $F$, mapping the pair $(x_{\alpha,\beta}, x_{\alpha',\beta''})$ into the pair $(x_{\alpha,\beta}, x_{\alpha,\beta'})$. Here the product of the elements of the first pair is zero because they belong to different direct summands $F_{\alpha}$ and $F_{\alpha'}$ of the algebra $H$. But then zero is also the product of two different free generators of $F_{\alpha}$, proving that the product in all algebras in $\frV$ is zero.
\end{proof}

\begin{Remark}\label{rCAEFA}
There exist locally finite varieties with nonzero multiplication in which any {\it countable}
 algebra can be embedded in a free algebra. An example will be given later in Section \ref{sssFAMV}.
\end{Remark}

\subsection{Solvable varieties.}\label{ssSolvable}

It is customary to define solvable algebras in two ways: first, by identical relations and, second, by certain series of ideals.

\medskip

Our calculations below will be preformed in an absolutely free algebra $\cF_\infty$ of countable rank with the free set of generators $x_1,x_2,\ld$. The identities $s_k(x_1,\ld,x_{2^k})=0$ are defined by induction. For $k=1$, we set $s_1(x_1,x_2)=x_1x_2$. For greater values of $k$, we set
\[
s_k(x_1,\ld,x_{2^k})=s_k(x_1,\ld,x_{2^{k-1}})s_k(x_{2^{k-1}+1},\ld,x_{2^{k}}).
\]

Also by induction, one defines the \emph{commutator series} of an algebra $A$, namely, $A^{(1)}=A$, and
$A^{(k)}= (A^{(k-1)})^2$ for $k>1$.

\begin{Prop}\label{solv} Let $A$ be an algebra. Then the following are equivalent.
\begin{enumerate}
  \item[$\mathrm{(i)}$] $A^{(\ell+1)}=\{ 0\}$ for some $\ell\ge 0$;
  \item[$\mathrm{(ii)}$] there exists a \emph{solvable} series of subspaces
\begin{equation}\label{defsol}
A=A_1\supset A_2\supset\ld\supset A_{\ell-1}\supset A_{\ell+1}=\{0 \},
\end{equation}
such that 
$A_{j}^2\subset A_{j+1}$, for all $j=1,\ld,\ell$,
  \item[$\mathrm{(iii)}$] $A$ satisfies the identity $s_{\ell}(x_1,\ld,x_{2^{\ell}})=0$.
\end{enumerate}
\end{Prop}

\proof (i)$\to$ (ii) is true since the commutator series is solvable.

(ii)$\to$ (iii) Clearly, $s_1(x_1,x_2)=0$ if $\ell =1$ since $A^2=\{ 0\}$ in this case.
By induction, the ideal $A_2$ from (\ref{defsol}) satisfies the identity $s_{\ell-1}(x_1,\ld,x_{2^{\ell-1}})=0$ if $\ell>1$. Since the products $a_1a_2,\ld, a_{2^{\ell}-1}a_{2^{\ell}}$ are in $A^2= A^{(2)}$ for any $a_i\in A$, one can replace the variables of $s_{\ell-1}$ with these $2^{\ell-1}$ products and obtain the identity $s_{\ell}(x_1,\ld,x_{2^{\ell}})=0$ in $A$.

(iii)$\to$(i). This is obvious for $\ell=1$. Let $\ell>1$ and $A$ satisfies $s_\ell(x_1,\ld,x_{2^\ell})=0$. Note that $A^{(2)}=A^2$ is simply a linear span of all products
$uv$, $u,v\in\cF_\infty$. Therefore, to check $s_{\ell-1}(x_1,\ld,x_{2^{\ell-1}})=0$ in
$A^2$, it is sufficient to replace the arguments only by such products. After this, the validity of $s_{\ell-1}(x_1,\ld,x_{2^{\ell-1}})=0$
in $A^2$ immediately follows from the validity of  $s_\ell(x_1,\ld,x_{2^\ell})=0$ in $A$. So by induction,
$A^{(\ell+1)}= (A^2)^{(\ell)}=\{ 0\}$.
 
\endproof

By definition, an algebra $A$ is called \emph{solvable} if it satisfies one of the equivalent
conditions of Proposition \ref{defsol}. The minimal $\ell$ is called the  \emph{solvable length}
of $A$. In particular,  $\ell =1$ for algebras with zero multiplication. The containment
$A_{j}^2\subset A_{j+1}$ implies that $A_{j+1}$  is and ideal of $A_j$ for $j=1,\dots,\ell$.

An algebra is called \emph{locally solvable} if any of its finitely generated subalgebras is solvable. 

A variety $\frV$ is called  solvable ( locally solvable), if any algebra in $\frV$ is solvable (resp., locally solvable). The solvable length of algebras in solvable variety has to be uniformly
bounded. (Use Cartesian products for a proof!)

\bigskip

In contrast with the case of nilpotent algebras, in the case of a solvable algebra $A$, it is not necessary that $A$ has non-zero abelian ideals. As an example, let $A$ be a 3-dimensional commutative algebra with a basis $a, b, c$ and the multiplication table 
\[
a^2 = b,\;  ab = b^2 = c,\;  ac = b,\;  bc = c,\;  c^2 = 0.
\]
Here $A^2 = \la b,c\ra$, $(A^2)^2 = \la c\ra$, $((A^2)^2)^2\ez$. One can easily see that the last nonzero term of the commutator series $(A^2)^2=\la c\ra$ is not an ideal, but only a subideal. Also, there $A$ has no abelian ideals.

But we still can obtain a property similar to that in Remark \ref{nilvar}.

\begin{Prop}\label{clSolvSubVar}
Let $\frV=\var A$, where $|A|\le\infty$. Then the set of all solvable algebras in $\V$ is a subvariety of $\frV$.
\end{Prop}
\begin{proof}
From Lemma \ref{minrep}, we know that the orders of all chief sections of finite algebras in $\frV$ are bounded by $|A|$. It follows that the lengths of the solvable series of solvable chief sections $M$ are bounded because the length of  $M \supset M^2 \supset (M^2)^2\supset\cdots$ cannot be greater than  $\dim M$. Therefore the solvability lengths of the solvable socles, which are the sums of minimal ideals, has bounded solvable length. Since the socle height of all finite algebras in $\frV$ is bounded by Theorem \ref{tSHA} and the solvable length of an extension of algebras does not exceed the sum of solvable lengths of the participating algebras, it follows that the solvable length of any finite solvable algebra in $\frV$ is uniformly bounded from above. Since $\frV$ is locally finite, the same is true for all solvable algebras in $\frV$. As a result, the set of all solvable algebras in $\frV$, which is obviously  closed under subalgebras and factor-algebras, is also closed under Cartesian product. So, the class of solvable algebras in a finitely generated variety $\V$ is a subvariety of $\V$, as claimed.
\end{proof}

\begin{Thm}\label{tA2leA}
 A variety $\frV$ of   algebras is solvable if and only if $A^2$ is a proper subalgebra in $A$ for any nonzero algebra $A\in \frV$.
\end{Thm}
\begin{proof}
If $A^{(2)}= A^2=A$, then by the inductive definition, every subalgebra $A^{(k)}$ is equal to $A$ and $A$ is not solvable unless $A=\{ 0\}$. 

Conversely, let us assume that $\frV$ is not solvable. 
Then for no $\ell$ an identity of the form $s_\ell(x_1,\ld,x_{2^\ell})=0$ is satisfied in the $\frV$-free algebra $F$ of countable rank. To prove the theorem,
we need to produce a nonzero algebra $A \in \frV$, such that $A^2 = A$.

First, we take the copies  $F(i) = F(x_{i1}, x_{i2}, \ld)$ of the algebra $F$ and construct
their Cartensian product  $C = \prod_{\ell=1}^\infty F(i)$ whose elements are infinite
sequences, where the $i$-th component belongs to $F(i)$. Then we define a series of
sequences in $C$, as follows.

The first one  is the sequence 
\[
a_{11} = (x_{11}, s_1(x_{21},x_{22}),  s_2(x_{31},x_{32},x_{33}, x_{34}),\ld).
\]
Two more sequences follow: 
\[
a_{21} = (0, x_{21},  s_1(x_{31},x_{32}), s_2(x_{41}, x_{42}, x_{43}, x_{44}),\ld)
\]
 and
\[
a_{22} = (0, x_{22},  s_1(x_{33},x_{34}), s_2(x_{45}, x_{46}, x_{47}, x_{48}),\ld).
\]

Using the identities
 
\[
s_j(z_1,\dots, z_{2^j})= s_{j-1}(z_1,\dots, z_{2^{j-1}})s_{j-1}(z_{2^{j-1}+1},\dots, z_{2^j}),
\]
 we can see that the components of product $a_{21}a_{22}$, except for the first one,
are the components of $a_{11}$. Then we define 
\[
a_{31} = (0,0,x_{31}, s_1(x_{41},x_{42}),\dots)\mbox{ 
and }a_{32}=(0,0,x_{32}, s_1(x_{43},x_{44}),\dots)
\]
so that $a_{31}a_{32}$ and $a_{21}$ have
the same components except for the second one. Similarly, we construct $a_{33}$ and $a_{34}$
such that $a_{33}a_{34}$ is equal to $a_{22}$ up to the second component. For each of
the four sequences $a_{3j}$, $j=1,..,4$, we similarly define eight elements  $a_{4j}$, $j=1,\ld8$ of $C$, and so on.

Let $B$ be the subalgebra of $C$ generated by all $a_{ij}$, $i=1,2,\dots$. Denote by $D$ the
 ideal of sequences with finitely many non-zero components. Then the algebra $A=C/(C\cap D)$
is non-zero, and the image $\bar a_{ij}$ of every generator $a_{ij}$ in $A$ is equal to a
product of some $\bar a_{i+1,j'}$ and $\bar  a_{i+1,j''}$. So all the generators of $A$ belongs
to $A^2$, and so $A^2=A$, as desired.
\end{proof}

\begin{example}\label{exxtop} An associative-and-commutative variety $\frV$ defined by the additional identity  $x^p=0$ over a field of characteristic $p>0$ is locally nilpotent but not solvable. The annihilator of its free algebra of infinite rank is trivial. The construction of Theorem \ref{tA2leA} gives us a non-zero algebra $A\in\frV$ such that $A=A^2$.
\end{example}

We conclude this section by considering the enveloping associative algebras (see Section \ref{sssEAV}) for one natural class of solvable algebras.  A finite-dimensional algebra $A$ is called {\it supersolvable} if
there is a series of ideals
\begin{equation}\label{super}
A=A_0\supset A_1\supset A_2\supset\dots \supset A_k=\{ 0\}
\end{equation}
where
the factors $A_{i-1}/A_i$ are one-dimensional algebras with zero multiplications.
(The definition resembles the definition of supersolvable groups.)

For example, the Lie algebra of triangular $m\times m$-matrices is supersolvable. If $A$ is a finite supersolvable algebra, then by Proposition
\ref{solv} and Corollary \ref{cc}, every finite algebra in the variety
$\var A$ is supersolvable.

\begin{Prop}\label{compso} Let $A$ be a finite solvable algebra. Then the
following are equivalent.
\begin{enumerate}
  \item[$\mathrm{(i)}$] $A$ is supersolvable;
  \item[$\mathrm{(ii)}$] there is a series of ideals $\{ 0\}=U_0\subset U_1\subset \dots \subset U_k=U$ in the associative enveloping algebra $U=U(A)$, such that $\dim U_i/U_{i-1}=1$ for $i=1,\dots,k$;
  \item[$\mathrm{(iii)}$] simple algebras in the  variety of associative algebras $\frU=\var U(A)$ are one-dimensional.
\end{enumerate}
\end{Prop}

\proof (i)$\to$(ii) Since the image of a term $A_i$ of the series (\ref{super})
under any left or right multiplication by an element $a\in A$ belongs to $A_i$,
one can choose a basis in $A$, such that any element of the enveloping algebra $U$
has a triangular matrix with respect to this basis. So we are done, since the
associative algebra of triangular matrices possesses the required series of ideals.

(ii)$\to$(iii)  By Corollary \ref{cc}, the dimensions of  chief sections of finite algebras in $\frU$ cannot be greater than the dimensions of chief sections of $U$. So the simple algebras in $\frU$ must be one-dimensional.

(iii)$\to$ (i) If $A$ has a chief section $V$ of dimension $m>1$, then $V$ is a simple module
of dimension $m$ over the enveloping algebra $U$. In this case, $A$ maps onto a primitive algebra $A'$ which has a faithful simple module $V$ of dimension $m$. By \cite[Chapter 3]{Her}, $A'$ is a simple algebra of dimension at least $2$, a contradiction. So there is a chief series (\ref{super}) in $A$, and $A$ is supersovable.

\endproof

\section{Dimension functions of locally finite varieties}\label{esFS}

One of the most natural  numerical invariants associated with a locally finite variety of arbitrary
universal algebras is the function $f(n)=f_{\frV}(n)$, whose value at an integer $n$ is the order of the  $\frV$-free algebra $F_n$ of rank $n$.  In \cite{HRM}, this function is called the \emph{free spectrum} of  $\frV$. If $\frV$ is a locally finite variety  of \emph{linear} algebras over a field $\F$ then it is more convenient to replace $f(n)$ with the {\it dimension function}
$d(n)=d_\frV(n) = \dim_{\bf F} F_n(\V)$; in this case $f(n) = q^{d(n)}$, where $q=|\bf F|$. If $\V=\var A$ then we write $d_\V(n)=d_A(n)$ and call $d_A(n)$ the \emph{dimension function of the algebra} $A$. Thus, $d(n)$ is the supremum of the dimensions of all $n$-generated algebras in the variety $\frV$. 

\subsection{First estimates of dimension functions.}

We  start with some estimates of the growth of the function $d(n)$.  Let us define the \emph{contents} $\cont(v)$ of a monomial $v\in \cF(X)$  as the (finite) subset $Y\subset X$ such that $v\in\cF(Y)$ but $v\not\in\cF(Z)$ where $Z$ is a proper subset of $Y$.   For example, $\cont((x_3x_2)x_3)=\{x_2,x_3\}$.

Given an arbitrary relatively free algebra $F=F(X)$ and a finite subset $J\subset X$, we denote by $F[J]$ the linear span of the images of all monomials $v$ with $\cont(v)=J$, under the epimorphism ${\cal F}(X)\to F(X)$ induced by the identity map of $X$.  To see that $F[J]$ is a  subalgebra of $F(X)$, assume $x\in J$, $u,v\in F[J]$ but $uv\not\in F[J]$, that is, $uv$ does not depend on $x$. If we map $x$ to $0$ and all other variables into themselves then each $u,v$ will be mapped into zero but $uv$ will be mapped into itself, a contradiction. 

\begin{Lemma}\label{grad} Any relatively free algebra $F=F(X)$ is a vector space direct sum of the subalgebras $F[J]$ over all the finite $J\subset X$. If $I$, $J$ are arbitrary subsets of $X$ such that  $|I|=|J|$ and $\dim F(J)<\infty$, then $\dim F[I]=\dim F[J]$.
\end{Lemma}

\begin{proof} Since $F(X)$ is generated by $X$, it follows that $F(X)$ is the sum of subalgebras $F[J]$. Now let us assume that there is a linear dependence 
\begin{equation}\label{elindep}
a_1+\dots+a_k=0,
\end{equation}
where all summands are non-zero and belong to different summands. We may assume that
$a_1\in F[J]$, such that the cardinality of $J$ is minimal for all
the summands. Consider the mapping $X\to F(X)$ which is identical on $J$, and maps
$x_i\mapsto 0$ if $x_i\notin J$. Let $\vp$ be the extension of this map to an endomorphism of $F(X)$. Since every $a_s$ with $s>1$, is a linear combination of the products of generators,
where at least one of the factors does not belong to $J$, we have $\vp(a_s)=0$. Also, $\vp(a_1)=a_1$.  Applying $\vp$ to (\ref{elindep}), we obtain $a_1=0$, a contradiction. This completes the proof of the first statement of the lemma.

The second statement is also true because there is an automorphism $\alpha$ of $F(X)$ such that $\alpha(J)=I$.
\end{proof}

Since $F[I]F[J] \subset F[I\cup J]$, Lemma \ref{grad} provides us with the grading
\begin{equation}\label{gra} 
F(X) = \bigoplus_{J\subset X, |J|<\infty} F[J]
\end{equation}
of $F(X)$ by the semigroup of finite subsets of $X$ with the join operation. It is a commutative
 idempotent semigroup.

\begin{Prop} \label{formdn} Let $F_n=F(X)$ be a finite-dimensional relatively free
algebra over any field with the free basis $X=\{ x_1,\dots, x_n\}$ and let $d_k = \dim F[J_k]$
for $J_k =\{x_1,\dots,x_k\}$, $k=1,\dots,n$. Then
\begin{equation}\label{dn}
\dim F_n = \sum_{J\subset X} \dim F[J] = \sum_{k=1}^n  d_k {n\choose k} 
\end{equation}
\end{Prop}

\proof The first equality in (\ref{dn}) follows from (\ref{gra}). The second is true since there are ${n\choose k}$ subsets of cardinality
$k$ in the set $\{x_1,\dots,x_n\}$ and  by Lemma \ref{grad} all of them have
the same dimension $d_k$.
\endproof

It is well known that the dimension function of nilpotent varieties is a polynomial functions  up to the big O. It fact, it is equal to a polynomial. Let us define polynomials $p_k(n)$ by 
\[
 p_k(n)={n\choose k} = \frac{1}{k!}\prod_{i=0}^{k-1} (n-i) 
\]

 \begin{Cor}\label{polynom}
If $\frV$ is a nilpotent variety of class $c$, then the dimension
function $d_\frV(n)$ is a polynomial function of degree $c$.
\end{Cor} 
\proof Any product of $c+1$ factors equals $0$ in an algebra from $\frV$.
So we have $d_{c+1}=d_{c+2}=\dots =0$ in (\ref{dn}). The coefficient $d_c\ne 0$
since otherwise all products of $c$ factors have to be zero in the algebras of $\frV$,
and the nilpotency class of $\frV$ would be less than $c$. Since $p_c(n)={n\choose c}$
is a polynomial of degree $c$ in $n$ and $p_k(n)={n\choose k}$ has a smaller degree
for $k<c$, the Proposition is proved. (The argument also works for $n<c$ since
$p_k(n)=0$ if $n<k$.)
\endproof

In the case of varieties of linear algebras, the following Corollary and the next Example
revise the estimate of \cite[Theorem 12.3]{HRM}, making it sharp.

 \begin{Cor}\label{2n1}
Let  $\frV$ be a locally finite variety of   algebras with the
dimension function $d(n)$. If $\frV$ is not nilpotent  then $d(n)\ge 2^n -1$.
\end{Cor}

\proof If in the formula (\ref{dn}), $d_k=0$ for some $k$ then the variety $\frV$
is nilpotent of class at most $k+1$. Therefore, by the assumption, we have
$d_k\ge 1$ for every $k\ge 1$. Then we see from (\ref{dn}) that
\[
d(n)=\dim F_n \ge \sum_{k=1}^n {n\choose k} = 2^n-1.
\]
\endproof

\begin{Example}\label{ex2nm1}
The associative and commutative variety $\frV$ over a field of characteristic 2, with an additional identity $x^2=0$, is locally nilpotent  but not nilpotent. A basis of $F_n(\frV)$ is formed by the product $x_{i_1}\cdots x_{i_s}$, where $s\ge 1$ and $1\le i_1<\ld <i_s\le n$. Hence,  $d(n)  = 2^n - 1 $. This shows that the estimate of the Corollary \ref{2n1} cannot be improved, in the general case.
\end{Example}

\subsection{Dimension functions for simple algebras.}\label{ssDFSA}

To estimate the dimension function of the variety generated by a finite simple algebra $A$, we need the following.

\begin{Lemma} \label{dnsimple} Let $A$ be a finite simple algebra with non-zero multiplication over a field $\F$ of order $q$, and $G=\Aut(A)$ the  automorphism group of  $A$. Denote by $t(n)$ the maximal integer such that the direct power $A^{t(n)}$ can be generated by $n$ elements. Then
\begin{enumerate}
  \item[$\mathrm{(i)}$] for all sufficiently large $n$, we have 
  \[
  t(n)\ge\frac{|A|^n}{|G|}\left(1-O(q^{-n}\right);
  \]
  \item[$\mathrm{(ii)}$] if $A$ is a minimal algebra then for all $n\ge 1$, we have
  \[
  t(n) = \frac{|A|^n-1}{|G|}.
  \]
\end{enumerate}
\end{Lemma}

\begin{proof} $\mathrm{(i)}$ We may assume that $n$ is at least the number of generators in $A$. For each maximal proper subspace $A'\subset A$, we have $|A'|=\frac{|A|}{q}$ and the number of $n$-tuples of elements in $A'$ is equal to $|A'|^n$. If $s$ is the number of proper subspaces in $A$, then the number of $n$-tuples generating $A$ is at least $|A|^n-s|A'|^n=|A|^n(1-O(q^{-n}))$.

Let us call two $n$-tuples generating $A$ {\it equivalent} if one of them can be mapped into another by an automorphism of $A$. Then one can find $N=\frac{|A|^n}{|G|}\left(1-O(q^{-n})\right)$ pairwise inequivalent $n$-tuples in $A$. Let us choose a maximal set of pairwise inequivalent $n$-tuples and denote it by $Y$. One can enumerate the $n$-tuples in $Y$ by the numbers from $1$ to
$N$.

Now, let $A_j$ be an isomorphic copy of  $A$, $j=1,\dots,N$. In the direct product $P$ of all $A_j$, we select the elements 
\begin{eqnarray*}
  x_1&=& (x_{11},\ld,x_{1N}) \\
  x_2 &=& (x_{21},\ld,x_{2N}) \\
  \ld &=& \;\,\ld\ld\ld\ld\\
  x_n &=&  (x_{n1},\ld,x_{nN})
\end{eqnarray*}
where the coordinates of the $j$th column is the $j$th generating tuple of $A$. 

Let $B$ be the subalgebra of $P$ generated by the elements $x_1,\dots, x_n$ and $C(k)$ be the projection of $B$ into the subalgebra $C(k)=A_1\times\dots\times  A_k\subset P$. We will prove by induction on $k$ that $B$ projects {\it onto} $C(k )$. Then, for $k=N$, we will obtain $B = C(N)=P$, and so $P$ is generated by $n$ elements, proving Claim (i) of the Lemma. For the base of induction, where $k=1$, our claim holds since  $x_{11},\ld,x_{n1}$ are the generators of $A_1=C(1)$.

Assume now that $k>1$ and that $B$ projects onto $C(k-1)$. In this case, $C(k-1)$ is generated by the projections $z_i$ of $x_i$ to $C(k-1)$. By Lemma \ref{lDirectSimple}, the kernel of any epimorphism $\vp: C(k-1)\to A$ is a product of $k-2$ direct factors $A_i$ of $C(k-1)$. So $\vp$ is a composition of the projection $\pi_j$ of $C(k-1)$ onto a factor $A_j$, where $j\le k-1$ and an isomorphism $A_j\to A$. The images of $z_1,\dots, z_n$ under $\pi_j$ are $x_{1j},\ld,x_{nj}$. So their images under $\vp$ form a $n$-tuple equivalent to the tuple $(x_{1j},\ld,x_{nj})$, where $j<k$.
 
 Let $D$ be the projection of $B$ to $C(k)$ generated by the images $y_i$
 of $x_i$. If the projection $\psi$ of $D$ onto $C(k-1)$ given by $y_i\mapsto z_i$ is an isomorphism of $D$ onto $C(k-1)$ then the image of the $n$-tuple $(y_1,\ld,y_n)$
 under an arbitrary epimorphism $D\to A$ is also equivalent to one of
 the tuples  $(x_{1j},\ld,x_{nj})$, where $j<k$. But this is not the case
 since $D$ has a projection onto $A_k$ under which $y_1,\ld,y_n$ map to
 $x_{1k},\ld,x_{nk}$, and the tuple  $(x_{1k},\ld,x_{nk})$ is not
 equivalent to any of the tuples mentioned above.
 
 Thus, $\psi$ has a non-trivial kernel $K$ that has to be contained in $A_j$. Since $A_k$ is a simple algebra we have $K = A_k$. Therefore $|D| = |A_k||C(k-1|= |C_k|$, and so $D=C(k)$, because $D$ is a subalgebra of $C(k)$. 
 
 \medskip
  
   $\mathrm{(ii)}$ Since $A$ has no non-zero proper subalgebras, it is generated by an
   arbitrary set of non-zero elements. Therefore one gets $|A|^n-1$
   generating $n$-tuples. If $\alpha$ is a nontrivial automorphism of $A$ then the set of fixed points of $\alpha$ is a subalgebra of $A$, that is, $A$ or $\{ 0\}$. Hence, if $a\in A$ and $a\ne 0$, then $\alpha(a)\ne a$. As a result, the $G$-orbit of every
   non-zero element of $A$ has cardinality equal to $G$. So we have
   exactly $N=\frac{|A|^n-1}{|G|}$ equivalence classes  of generating
   $n$-tuples for every $n\ge 1$. This proves that $t(n)\ge \frac{|A|^n-1}{|G|}$.
    To obtain the opposite inequality, it suffices to observe
   that  a set $x_1,\ld, x_n$ cannot generate the direct product
   $P=A_1 \times\dots\times A_t$ of copies of $A$, if the projections
   on the diferent factors, say $(x_{11},\dots,x_{n1})$ and
   $(x_{12},\dots,x_{n2})$ are equivalent. Indeed, in this case
   there is an isomorphism $\alpha: A_1\to A_2$ such that
   $\alpha(x_{k1}) = x_{k2}$ for $k=1,\dots, n$. It follows
   that the projection of the subalgebra $B$, generated by  $x_1,\ld, x_n$,
   into $A_1\times A_2$ is the diagonal of this product defined
   by the isomorphism $\alpha$. Hence $B\ne P$.
   
   The lemma is proved.
   
 \end{proof}
 
 \begin{Thm} \label{limA} If $A$ is a finite simple algebra over a field $\bf F$, $A^2\ne 0$, and 
${\frV}= \var A$, then there exists a positive number $c$ such that for all natural $n$ one has
\[
 c|A|^n \dim_{\bf F}A < d_\frV(n) < |A|^n \dim_{\bf F}A.
 \]
 In particular,
 $\lim_{n\to \infty}\sqrt[n]{d(n)} = |A|$.
\end{Thm}

\proof By Lemma \ref{dnsimple}, there is a positive number $c$ such that if $A^{t(A)}$ is $n$-generated then  $t(A)>  c|A|^n $. This proves the left inequality. The right inequality
follows from the standard Birkhoff inequality (\ref{eBirkhoff}). 

\endproof

Birkhoff's estimate of the dimension function $d_\frV(n)$ of a finitely generated variety $\frV=\var A$ implies that there is $\limsup_{n\to\infty}\sqrt[n]{d(n)}<\infty$. We call this limit the {\it exponent} $\mathrm{Exp}(\V)$ of $\frV$ or $\mathrm{Exp}(A)$ if $\V=\var A$. By Corollaries \ref{polynom}, \ref{2n1}, the exponent is $1$ iff the finite algebra $A$ is  nilpotent, otherwise it is at least $2$. By Theorem \ref{limA}, the exponent is an integer if $A$ is a simple algebra.

\begin{problem}\label{pbmIntegralExp} If $A$ is a finite algebra, does there exist a limit of $\sqrt[n]{d_A(n)}$ as  $n\to\infty$ ? Is it always true that $\mathrm{Exp}(A)$ is an integer? This limit has to coincide with the limit of $\dfrac{d_A(n)}{d_A(n-1)}$, if the latter exists.
\end{problem}
\endproof

Recall the Birkhoff's function $b(n) = (|A|^n -1)\dim A$ (See subsection \ref{ssBCN}.)
  
\begin{Thm}\label{etFreeSpectrum1} Let $A$ be a finite algebra  over a field
$\bf F$, $A^\ne 0$, and $d(n)$ the dimension function of the variety $\frV= \var A$. Then $d(n) \le b(n)$ and the functions $d(n)$ and $b(n)$ coincide if and only if  $A$ is a minimal algebra having no nontrivial automorphisms.

If  $A$ is a minimal algebra, then the algebra $F_n = F_n(\frV)$ is the direct sum of 
$\frac{|A|^n-1}{|\Aut A|}$ ideals isomorphic to $A$.
\end{Thm} 

\proof By Theorem \ref{etBirkhoof_Estimate}, the inequality $d(n) \le b(n)$ always holds and it is strict if $A$ has a non-zero proper subalgebra by Remark \ref{subBirk}. If $A$ is minimal algebra and the automorphism group $G$ of $A$ is non-trivial, this inequality is also strict by Corollary \ref{minalg} and Lemma \ref{dnsimple}.
 
If the algebra $A$ is minimal and has $G=\{1\}$, then by Lemma \ref{dnsimple} (ii), the variety $\frV$ contains an $n$-generated subalgebra of dimension  $(|A|^n -1)\dim_{\bf F} A$ for every $n\ge 1$,
 and so we obtain the equality of the functions $d(n)$ and $b(n)$.

For the second statement, we recall, that by Corollary \ref{minalg}, $F_n$ is the direct sum
of ideals isomorphic to $A$. So this is an $n$-generated sum with the maximal number of summands  $t(n)$. It remains to apply Lemma \ref{dnsimple} (ii).

\endproof

\subsection{Dimension functions of arbitrary finite algebras.}
For non-simple algebras we have less spectacular exponential estimates.

\begin{Thm}\label{epNongn}
Let $A$ be a non-nilpotent finite algebra over a field $\F$, $|\F|=q$, $\frV$ the variety generated by $A$, and $d(n)$ the dimension function of  $\frV$. Then 
$d(n)\ge c q^n$ for some $c>0$ and all positive integers $n$, and so the exponent of $\frV$ is at least $q$.
\end{Thm}

\begin{proof}
  By choosing $A$ minimal non-nilpotent, we may assume that $A$  has a minimal ideal $I$ and $A/I$ is nilpotent. Then $I$ does not annihilate the whole of $A$ by Lemma \ref{lNilpSimpleProp}. Without loss of generality, we may assume that $AI\nez$. Also we can  assume that $n\ge n_0 = \dim A$.
  
  Recall (see Section \ref{sssEAV}) the left enveloping algebra $U_L(A)$, an associative subalgebra in $\End_{\bf F} I$ generated by all operators of the left multiplications $L(a)$, where $a\in A$. It is generated by the image $W\ne \{ 0\}$ of $A$ under the linear map $L:A\to \End_{\bf F} I$. 
  
  Note that given any set of $n$ elements $L(b_1),\ld,L(b_n)$  spanning $W$, one can select $n$ elements $a_1,\ld,a_n\in A$, spanning $A$, so that $L(a_1)=L(b_1),\ld,L(a_n)=L(b_n)$. Indeed, if, say, $\{L(b_1),\ld,L(b_k)\}$, $k\le n$, is a maximal linearly independent subset of that set, then any $a_1,\ld,a_k$ with $L(a_1)=L(b_1),\ld,L(a_k)=L(b_k)$ are linearly independent in $A$. If $X=\la a_1,\ld,a_k\ra$ and $K=\Ker L$ then $A=X\oplus K$, $\dim K=m-k$. Choose $\{ c_{k+1},\ld,c_n\}$ in $X$ so that $L(c_{k+1})=L(b_{k+1}),\ld,L(c_n)=L(b_n)$. Choose a basis $\{ d_{k+1},\ld,d_m\}$ in $Y$. Now the desired set $a_1,\ld,a_n$ is the following:
  \[
 a_1,\ld,a_k,a_{k+1}=c_{k+1}+d_{k+1},\ld,a_m=c_m+d_m, a_{m+1}=c_{k+1},\ld,a_n=c_n.
  \]
  
  The number of  $n$-tuples of elements in $W$ equals $|W|^n$. An $n$-tuple $w_1,\ld,w_n$ spans $W$ if and only if there is no maximal proper subspace $W'$ such that $w_1,\ld,w_n\in W'$. For each maximal proper subspace  $W'\subset  W$, the number of $n$-tuples of elements in $W'$ equals $|W'|^n$, where  $|W'|=\frac{|W|}{q}$. If  $s$ is the number of subspaces like $W'$, then the number of $n$-tuples generating $W$ will be at least $|W|^n-s|W'|^n$. Even for not so large values of $n$, we have that the number of generating $n$ tuples is greater than $\frac{1}{2}|W|^n$.

Let us call two $n$-tuples spanning $W$ equivalent if one can be mapped into another by an automorphism of $U$. Obviously, one can find more than  $\frac{|W|^n}{2m}$ pairwise inequivalent generating $n$-tuples in $W$, where $m$ is the order of the automorphism group $\Aut U$. Let us denote by $Y$ any of such sets of pairwise inequivalent $n$-tuples. It was noted earlier that there is a preimage in $A$ spanning $A$ for every $n$-tuple from $Y$. Choose a set $Z$  consisting of one preimage in $A$ for each $n$-tuple in $Y$, as above. So we have that  $|Z|>c|W|^n$ where $c = \frac{1}{2m}$, and the images of different  $n$-tuples from  $Z$ in $U$ are pairwise inequivalent. We enumerate the $n$-tuples of $Z$ by the numbers from $1$ to $N > c |W|^n$.

Now, for any $j=1,\ld,N$, we take an isomorphic copy $A_j$ of $A$ spanned by
the $n$-tuple $(x_{1j},\dots,x_{nj})$ having number $j\le N$, a copy $I_j$ of $I$, and a copy $U_j$ of $U$. In the direct product $P$ of all $A_j$ we select $N$-tuples  
\[
x_i = (x_{i1},\ld,x_{iN}),\,i=1,\ld,n,
\]
Let $y_{1j}=L(x_{1j}),\ld, y_{nj}=L(y_{nj})$ be the images of $x_{1j},\ld,x_{nj}$ in $U_j$.

Notice that for any $j'\ne  j$ the sets of relations of $U_{j'}$ and $U_j$ with respect to the selected $n$-tuples of generators are different. Indeed, if not, by Dyck's Lemma there is an isomorphism of these (associative) algebras transforming one of these tuples into another, which contradicts the choice of $n$-tuples in $Y$ and $Z$. Moreover, it follows from  $\dim U_{j'}= \dim U_j$ 
that there is an associative polynomial $p_{jj'}$ in $n$ variables, whose values are zero in the selected generators of  $U_{j'}$, but not $U_{j}$.

Let us consider the subalgebra $B\subset P$ generated by the $n$-tuple $x_1,\ld,x_n$. Obviously, $B\in\var A$. We will show that $B$ contains all the ideals  $I_j, j=1,\ld,N$. This will restrict the dimension of $B$ from below by an exponential function $c|W|^n\ge cq^n$. For the proof, we assume that $j=1$.
  
Since $A_j$ is not nilpotent and $A_j/I_j$ is nilpotent, some powers of $P$ and of $B$ are in  the product of $I_j$'s,
  and there is an element $a\in B$ belonging to this product with non-zero
  first component  $a_1\in I_1$. We chose such $a$ with minimal set of
  non-zero components $a_j$. Since $I_1$ is a minimal ideal in $A_1$ and $B$
  projects onto $A_1$, one can change $a$ and $a_1$ for $a'$ and $a'_1$ with
  the same properties, where $a_1'$ is an arbitrary non-zero element in $I_1$.
  
  If also some $a_2\ne 0$, we recall the associative polynomials $p_{ij}$ and use $p_{12}$. The value of $p_{12}$ in $y$'s annihilates all chosen generators of $U_2$ but not the generators of $U_1$. Hence there is $a'_1\in I_1$ which is not annihilated by $p_{12}(y_{11},\ld,y_{n1})$ but $p_{12}(y_{12},\ld,y_{n2})$ annihilates
  any element from $I_2$. Notice that the linear operators $y_{i1}$ and $y_{i2}$
  acts on $I_1$ and $I_2$, resp., as left multiplication by $x_i$. Hence
  the operator $p_{12}(L(x_1),\ld,L(x_n))$ applied to $a'$ leaves the
  first component of $a'$ non-zero, but decreases the number of other
  non-zero components. Clearly, the element $p_{12}(L(x_1),\ld,L(x_n))(a')$ also belongs to both $B$ and the product of the ideals $I_j$, contrary to the choice of $a$.
  Thus, $B$ contains a non-zero element of $I_1$, and so $I_1\subset  B$,
  which completes the proof.
  
\end{proof}

The estimate of the exponent of $\var A$ in Theorem \ref{epNongn} looks weak when compared
with the  value obtained in Theorem \ref{limA}. It seemed plausible that the number $q$ in the conclusion of Theorem \ref{epNongn} could be replaced by the maximum of orders of chief factors of  $A$. In that case, Theorem \ref{limA} would be an easy consequence. Unfortunately, that expectation turned out to be false. Below, we provide a counterexample.

\begin{Example}\label{Liemetab} Here we show that
  there exists a finite algebra $A$ with monolith $M$ of arbitrary high order $q^d$ such that $\mathrm{Exp}(A)=q$.
Let $A$ be a Lie algebra over a field $\F$ of order $q$ which is a semidirect product of an abelian ideal $M$ and a one-dimensional subalgebra $\langle a\rangle$. It follows from the existence of irreducible polynomials of arbitrarily high degree over a finite field $\F$ that one can choose $M$ irreducible of arbitrarily high dimension $d$. Thus, $M$ is a monolith of $A$ of dimension $d$.  

Let $F_n$ be a free algebra of rank $n$ in $\var A$. It follows from Corollary \ref{etBirkhoof_General} and Lemma \ref{minrep}  that $F_n^2$ is the direct sum of minimal ideals of dimension $d$. On each of these ideals, $F_n$ acts by adjoint action as the projection of $F_n$ on the respective semidirect factor of the minimal representation of $F_n$, more precisely, as a one-dimensional factor-algebra of $F_n$. Since we have less than $q^n$ of epimorphisms of $F_n$ to a one-dimensional algebra, in the decomposition of $F_n^2$ as the direct sum of minimal ideal there will be less than $q^n$ pairwise non-isomorphic $F_n$-modules, in other words, $F_n^2$ has less than $q^n$ homogeneous components which are the sums of isomorphic $F^n$-modules.
To proceed we need

\begin{Lemma}\label{lGenRem} Let $P$ be a finite left module over an associative or Lie ring $R$ such that $P=\underbrace{Q\times\cdots\times Q}_t$ where $Q$ is an $R$ module of order $m>0$. If $P$ can be generated as an $R$-module by $r$ elements then $t<mr$.
\end{Lemma}
\begin{proof}
 If we fix $b\in P$, and view it as a function $\{ 1,2,\ld,t\}\to Q$, then $b$ has  $<m$ different nonzero values. In this case, $b$ is contained in the sum of $<m$ diagonal submodules isomorphic to $Q$. Hence, it generates a submodule of order $<m^{m}$. But then $r$ elements of $P$ can only generate a submodule of order $<m^{rm}$, so that $m^t<m^{rm}$ and then $t<mr$.
\end{proof} 

Now we come back to $F_n^2$. As a module over $F_n$, it is generated by $\frac{n(n-1)}{2}$ products (=commutators) of free generators. Each homogenious component, being a direct summand of $F_n^2$, can be generated by the same number of elements. By Lemma \ref{lGenRem}, the number of direct summands of dimension $d$ in each component is less than $q^d\frac{n(n-1)}{2}$.

Since the number of homogenious components is less than $q^n$, the whole of ideal $F_n^2$ contains  $<q^nq^dn(n-1)/2$ of direct summands of dimension $d$. So $\dim F_n^2<q^nq^dd\frac{n(n-1)}{2}$. Taking the $n$th root of the latter value and having in mind the lower estimate of Theorem \ref{epNongn}, we derive that the exponent of the variety $\var A$ equals $q$, much less than $q^d$.
\end{Example}

\subsection{Free algebras of minimal varieties and their subalgebras.}\label{sssFAMV}
A minimal locally finite variety $\frV$ is generated by a minimal finite algebra (Corollary \ref{minvar}).
For every $n$, we have the standard embedding of the $\frV$-free algebra $F_n=F(x_1,\dots,x_n)$ into
the $\frV$-free algebra $F_{n+1}=F(x_1,\dots,x_n, x_{n+1})$ identical on the free generators $x_1,\dots,x_n$. In this subsection,  using Lemma \ref{dnsimple}, we describe this embedding
in terms of an embedding of direct powers of the algebra $A$. We finish this section by proving Theorem \ref{tSFAMV} whose statement is in contrast to Theorem \ref{tEmbedToFree}.

Let $A$ be a finite minimal algebra with nonzero product, $\frV=\var A$. By Corollary \ref{minalg} the free algebra $F=F_n(\frV)$ is equal to $A_1\times\cdots\times A_s$ where $A_i\cong A$ for all $i=1,\ld,s$, $s=f(n)$. Similarly, $F_{n+1} = A'_1 \times\cdots\times A'_{s'}$, where   $s'= f(n+1)$   and all the factors are also isomorphic to $A$.

By Lemma \ref{Diagonal}, every $A_i$ is a diagonal  of the sum $\sum_{j\in J_i}A'_j$ with mutually disjoint sets of indices $J_i$. Now  every permutation of factors  $A_i$ defines an automorphism of $F_n$. Also every automorphism of $F_n$ can be extended to an automorphism of $F_{n+1}$ by mapping  $x_{n+1}\mapsto x_{n+1}$. It follows that all the diagonals $A_i$, for  $i=1,\ld,s$, are automorphically conjugate in $F_{n+1}$.

In particular, the ideals $K_i$ generated by these diagonals must have the same dimension, for all $i=1,\ld,s$. Since $K_i= \sum_{j\in J_i}A'_j$ , all the subsets $J_i$ have the same cardinality, which we denote by  $m=m(n)$.

It is obvious that the factor-algebra of $F_{n+1}$ by the ideal generated by the image of $F_n$ under the standard embedding, is isomorphic to $F_1$, which has in its direct factorization $f(1)$ factors isomorphic to $A$. On the other hand, it follows from the above, that the ideal generated by $F_n$ in $F_{n+1}$ has exactly $ms$ direct factors isomorphic $A$. Therefore, $f(n+1) = m(n)f(n)+f(1)$.

Now by Lemma \ref{dnsimple} (ii), for any $n\ge 0$, we have  $f(n) =\frac{|A|^n-1}{r},$  where $r = |\Aut A|$. As a result, 
\[
\frac{|A|^{n+1}-1}{r} = m(n)\,\frac{|A|^n-1}{r} + \frac{|A|-1}{r},
\]
so that $m=m(n)=|A|$, for any $n\ge 0$. In other words, for any $n\ge 1$, the standard embedding $F_n\to F_{n+1}$ identifies every direct factor of $F_n$ isomorphic to $A$ with a diagonal of the product of $|A|$ factors in the direct decomposition of $F_{n+1}$. Thus, we have explicitly described the embeddings of $F_n$ into $F_{n+1}$ in the language of the direct powers of $A$. We also have that the free algebra of countable rank in $\var A$ is the direct limit $F_1\to F_2\to\ld$ of such embeddings of the direct powers of $A$, where
every $F_n$ is diagonally embedded into its own ideal closure in $F_{n+1}$ which is the direct
product of $|A|$ ideals of $F_{n+1}$ isomorphic to $F_n$.
\medskip

\begin{example}\label{exLagrange} The simplest example of a minimal algebra is a finite field $\F$, as an algebra over itself. Not only it has no proper subalgebras but also no notrivial automorphisms. In this case,  $F_n$  is the factor-algebra of the algebra $\F[x_1,\ld,x_n]$ of polynomials with zero free term by the ideal generated by all polynomials  $x_i^q-x_i$, where $q=|\F|$. Its elements are uniquely represented by the polynomials of degree less than $q$ in each variable, with constant term zero. In order to write  $F_1$ as the direct product of $q-1$ ideals $A_j$ isomorphic to $\F$, we introduce for each nonzero  $\alpha\in \bf F$ the Lagrange's interpolation polynomial 
\[
f_\alpha(x)=\prod_{\beta\in \F\setminus{ \alpha}}(x-\beta)
\]
 and consider the set $A_\alpha$ of all multiples of $f_\alpha(x)$. It is easy to observe that $A_\alpha$ is an algebra isomorphic to $\F$ , it consists of all truncated polynomials taking zero value at every $\beta\ne\alpha$. The isomorphic mapping of $A_\alpha$ into $\F$ can be described as follows: given $g(x)\in A_\alpha$, map $g(x)$ to $g(\alpha)$.

One can see that decomposing a polynomial as the sum of components  with respect to the direct sum of algebras $A_\alpha$ amounts to considering the ``Lagrange interpolation'' 
\[
f=\sum_{\alpha\in\F\setminus \{0\}}\lambda_{\alpha}f_\alpha, \;\;\; \lambda_{\alpha}\in \bf F.
\] 

Under the standard embedding $F_1=F(x)$ into $F_2=F(x,y)$ the factor $A_{\alpha}$ experiences a diagonal embedding into its own ideal closure $ A_{\alpha}F(y)$ in  $F_2$. In its turn, this closure decomposes as the sum of ideals isomorphic to $\F$.  To obtain this decomposition, one should apply the same argument now to $F(y)$.
\end{example}

For any minimal algebra $A$, switching from the standard representation of the elements in $F_n$ as nonassociative polynomials in free generators to its decomposition with respect to the direct summands $A_j$ isomorphic to $A$, is a full analogue of Lagrange's interpolation. But the explicit computation of the generalized Lagrange's functions  $f_j=f_j(x_1,\dots,x_n)$, that is, such elements of $F_n$, which have only one nonzero projection to the simple factors, seems, in the general case, to be quite a challenging problem.

In particular, what can be said about the degrees of these polynomials ? So we formulate

\medskip

\begin{problem}\label{pbmFALinSpan} Given a finitely generated variety $\frV$, does there exist a constant $c=c(\frV)$ such that for any $n$, the free algebra $F_\frV(x_1,\ld,x_n)$ is the linear span of monomials of degree at most $cn$ in $x_1,\ld,x_n$?
\end{problem}

 We apply the description of the embeddings $F_n\to F_{n+1}$ in the following.

\begin{Thm}\label{tSFAMV} Let $A$ be a minimal finite algebra. Then every countable algebra of the
variety $\var A$ can be isomorphically embedded in the $\var A$-free algebra $F$ of countable
rank.
\end{Thm}

\begin{proof}  By $F_n$,
$n=1,2,\dots$, we denote a $\var A$-free algebras of rank $n$,  $n=1,2,\dots$. $F_n$ is embedded in $F_{n+1}$ in a standard way. Let $F$ denote  the free $\var A$-algebra of countable rank.

 One can assume that $A^2\nez$ since otherwise the statement is obvious. We start with two remarks.

If an algebra $A$ is embedded as a diagonal in the direct product $P=A_1\times\dots\times A_k$ of isomorphic copies $A_i$ of $A$, then for every $\ell=1,\dots, k$, there is a subalgebra $Q\subset P $
with the direct decomposition $Q=A'_1\times\dots\times A'_{\ell}$, $A_i'\cong A$, such that
$A$ is a diagonal subalgebra of $Q$. Indeed, if suffices to take $\ell=k-1\ge 1$. If $S$ be the projection of $A$  to the subalgebra $A_{k-1}\times A_k$ of $P$, then $A$ is a diagonal of the subalgebra $Q=A_1\times\dots\times A_{k-2}\times S$.

The second remark we need is the following. We have shown earlier in this section, that each direct factor of the free algebra $F_n$ which is isomorphic to $A$ can be embedded as a diagonal into a direct product of $|A|$ direct factors of $F_{n+1}$. In turn, each of these latter is a diagonal in the product of $|A|$ direct factors of $F_{n+2}$, and therefore a direct factor of $F_n$ becomes a diagonal of a product of $|A|^2$ direct factors of $F_{n+2}$, and so on.

Let now $B$ be a subalgebra of $F_n$ and also $B$ a subalgebra of a finite algebra $C \in \var A$. Then by Corollary \ref{minalg},  $B$ is the  direct product $A_1\times\dots\times A_k$ , $C$ is the  direct product
$A'_1\times\dots\times A'_l$ and $F_n = A_1''\times\dots\times A''_m$, where all factors in the products are isomorphic to $A$. By Lemma \ref{Diagonal}, every $A_i$ is a diagonal of
some $\sum_{j\in J_i} A'_j$ with disjoint sets $J'_i$ and similarly  $A_i'$ is the diagonal of $\sum_{j\in J'_i} A''_j$.

By the second remark,  $A_i$ becomes a diagonal of some $S_i=\sum_{j\in J''_i} A'''_j$ in a direct decomposition of  $F_m$, where  one can choose $m>n$ so that $|J''_{i'}|\ge |J_{i}|$ for any $i$ and $i'$. By the first remark, $A_i$ is
a diagonal of a subalgebra $B_i$ in $S_i$ with $|J_i|$ direct factors, and so $B_i\cong \sum_{j\in J_i} A'_j$
and therefore the  sum of $B_i$'s  contains $B$ and is isomorphic to $C$. Thus, for every two finite algebras $B\subset C\in \var A$
and  an embedding of $B$ into $F_n$, there is $m\ge n$ such that the embedding of $B$ into $F_n$ extends to the embedding of $C$ to $F_m$.

Let now $B$ be a countable algebra in $\var A$. Since it is locally finite, we have an acsending series
of finite algebras $0=B_0\subset B_1\subset B_2\dots$ such that $\cup_{i=1}^\infty B_i = B$
and $F_1\subset F_2\subset \dots$
such that  $\bigcup_{i=1}^\infty F_i = F$. We have proved that if we have an embedding $\iota_i:B_i\to F_{n_i}$,  this can be extended to an embedding $\iota_{i+1}:B_{i+1}\to F_{n_{i+1}}$,
where $n_{i+1}>n_i$. Since every $\iota_{i+1}$ is equal to $\iota_i$ when restricted
to $B_i$, one gets an embedding of the whole of $B$ into $F$, as required.
\end{proof}

\subsection{Dimension functions in varieties of associative algebras.}\label{cross}

Here we obtain an explicit formula for the dimension function of
the product of two finitely generated varieties of associative algebras.

Given two varieties of associative algebras $\frU$ and $\frV$ over a field $\F$, one defines their \emph{product} $\frV\frU$ as the class of all associative algebras satisfying all identities of the form $u(x_1,\ld, x_m)v(x_{m+1},\ld,x_{m+n})=0$, where $u(x_1,\ld,x_m)=0$ is an identical relation in $\frU$ and $v(x_{m+1},\ld,x_{m+n})=0$ is an identical relation in $\frV$. This multiplication of  
varieties is associative, and we denote by $\texttt{Ass}(\F)$ the semigroup of all associative varieties, except for the variety of all associative algebras, under this multiplication. For example, the variety  of all associative nilpotent algebras of class $ c$ is the $(c+1)$th power of the trivial variety given by the identity $x=0$. 

Since we consider
associative algebras only in this subsection,  it suffices here to use associative polynomial without constant terms. They form the free associative algebra $F=F[x_1,x_2,\dots]$.

\begin{Remark} \label{RemBL} Let the algebra $F^1$ be obtained from $F$ by adjoining of $1$. A theorem of Bergman - Lewin \cite{BeLe} states that the set of all proper nonzero ideals of $F^1$ is a free semigroup under
the ideal multiplication. Then in Theorem 7, they prove that $T$-ideals form a free subsemigroup
of this semigroup. Observe that in the proof of Theorem 7, the authors uses only those automorphisms
and endomorphisms of $F^1$ which leave the subalgebra $F$ invariant. Therefore, although the set of 
$T$-ideals is smaller than the set of verbal ideals of $F$, exactly
the same proof allows us to assert that the non-zero verbal ideals of $F$ form a free 
semigroup under multiplication. Thus, the {\it semigroup $\texttt{Ass}(\F)$ is free}.
\end{Remark}

Let us start with the following observation.

\begin{Prop} \label{subsemi} Let $\F$ be a finite field. Then the set of finitely generated  varieties of associative algebras over $\F$ is a free subsemigroup in  $\texttt{Ass}(\F)$ 
\end{Prop} 

\begin{proof} We have $\frV=\var A$ and $\frV=\var B$ for some finite algebras $A$ and $B$. Then the union $\frU\cup\frV$ in the lattice of varieties is generated by $A\times B$. Hence $\frU\cup\frV$ is a  finitely generated variety. Now, let $U$, $V$ be the verbal ideals of the free associative algebra of countable rank $F$, corresponding to $\frU$, respectively, $\frV$. Then $\frU\cup\frV$ is defined by  $U\cap V$ and $\frU\frV$ is defined by $UV$. Obviously, $U\cap V /UV$ is an ideal of $F/UV$ with zero product. Thus each algebra in $\frU\frV$ is an extension of an algebra in $\frU\cup\frV$ by an abelian ideal. In associative algebras, any extension  of a locally finite algebra by a locally finite ideal is locally finite \cite[10.2.1]{Jac}. So $\frU\frV$ is a locally finite variety. By (\cite{Kruse,Lvov}),  or by Remark \ref{nilvar}, the nilpotency classes of nilpotent algebras are bounded by a constant $c$ in the variety $\frU\cup\frV$. So $2c+1$ bounds the nilpotent
classes of nilpotent algebras in $\frU\frV$. Now if the nilpotent classes of nilpotent  algebras in a locally finite variety of associative algebras are bounded then this variety is finitely generated (Kruse - L'vov Theorem (see \cite{Kruse,Lvov}) So is $\frU\frV$. Thus, the set of finitely generated associative  varieties is a subsemigroup of the free semigroup $\texttt{Ass}(\F)$.

If a product of two varieties  $\frU\frV$ is finitely generated, then both  $\frU$ and $\frV$ are
finitely generated, because a subvariety of a finitely generated variety of associative algebras  is finitely generated \cite{Kruse,Lvov}. Since
the semigroup $\texttt{Ass}(\F)$ is freely generated by all indecomposable varieties,
it follows that the subsemigroup of finitely generated varieties is generated by the finitely generated
indecomposable varieties. Hence this subsemigroup is free itself.

\end{proof}

\begin{Thm} \label{dnUV} Let $\frU$ and $\frV$ be a locally finite varieties of associative algebras
over a field $\bf F$ with dimension functions $d_1(n)$ and $d_2(n)$, respectively. Then the
dimension function $d(n)$ of the product $\frU\frV$ is $(n-1)d_1(n)d_2(n) + n(d_1(n)+d_2(n)+1)$.
\end{Thm}

\proof It is well-known, see \cite{FIR}, that the free associative algebra is a free ideal ring. In particular, every submodule of a free module over $F(X)$ is free.  The algebra
$F(X)$ is a free module over itself freely generated by $X$. If we attach 1 to $F(X)$ in a standard way, then the algebra $F'(X)$ thus obtained is a free $F(X)$-module of rank 1. If $s$ is the codimension of $J$ in $F(X)$ then $s+1$ is the codimension of $J$ in $F'(X)$. Now let $L$ be the free $F'(X)$-module of rank $r$ and $J$ be a left ideal of $F(X)$ of codimension $s$. Then $J$ has codimension $s+1$ in $F'(X)$. We have $L\cong (F'(X))^r$ (direct power) and $L/JL\cong (F'(X)/JF'(X))^r=(F'(X)/J)^r$. As a result, if $J$ is a left ideal of codimension $s$ in $F(X)$ then the codimension of the left submodule $JL$ in $L$ is given by
\begin{equation}\label{sr}
\dim L/JL=(s+1)r .
\end{equation}
To proceed, we will use an analog of Schreier's formula for the ranks of the subgroups of free groups established by Lewin \cite{Lew}.  Lewin's formula says that if $F(X)$ is a the free associative algebra of rank $n$ over a field $\bf F$, $M$ a free (left) $F(X)$-module of finite rank, and $N$ a submodule of finite codimension in $M$, then $N$ is a free $F(X)$-module and
\begin{equation}\label{sl}
\rank N = (n-1)(\dim(M/N)+1) + \rank M
\end{equation}

Returning to the varieties $\frU$ and $\frV$, let $U$ and $V$ be the verbal ideals corresponding to $\frU$ and $\frV$ in the free associative algebra $F=F(X)$ of rank $n$. Then $d(n) =\dim F/UV= \dim F/V +\dim V/UV$. Here the first summand equals $d_2(n)$. Then $V$ is a free $F(X)$-modules, and by the formula (\ref{sl}), $\rank V = (n-1)(d_2(n)+1)+ n$. Since the codimension of $U$ in $F(X)$ is $d_1(n)$, the formula (\ref{sr}) gives $\dim V/UV = (d_1(n)+1)((n-1)(d_2(n)+1)+ n)$, whence
\begin{eqnarray*}
  d(n) &=& (d_1(n)+1)((n-1)(d_2(n)+1)+ n)+d_2(n)  \\
 &=& (n-1)d_1(n)d_2(n) + n(d_1(n)+d_2(n)+1).
\end{eqnarray*}
\endproof

\begin{Remark} \label{UVVU}
\begin{enumerate}
  \item[$\mathrm{(i)}$]  It follows from Proposition \ref{subsemi} and Theorem \ref{dnUV}, that if a finitely generated variety $\frV$ has exponent $e$, then the finitely generated variety ${\frV}^2$ has exponent $e^2$ .
  \item[$\mathrm{(ii)}$] By Bergman - Lewin's theorem, the equality $\frU\frV=\frV\frU$ holds iff both varieties $\frU$ and $\frV$ are some powers of the same associative variety $\frW$. Nevertheless, by Theorem \ref{dnUV},
the \emph{dimension functions of the products $\frU\frV$ and $\frV\frU$ are equal for any pair of
locally finite associative varieties $\frU$ and $\frV$}.
\item[$\mathrm{(iii)}$] It follows from Corollary \ref{polynom} and Theorem \ref{dnUV} that if two
nilpotent varieties of associative algebras have nilpotency classes $c_1$ and $c_2$, then
the nilpotency class of their product is exactly $c_1+c_2+1$.
\end{enumerate}

\end{Remark}

\subsection{Few more numerical estimates.}\label{sCCNALF} We conclude Section \ref{esFS} with few remarks concerning the growth in locally finite varieties.

\begin{Prop}\label{pNumber} If a locally finite variety $\frV$ of algebras over a field $\bf F$ has an algebra with nonzero product then the growth of the number of pairwise nonisomorphic $n$-generated algebras in $\frV$ is at least exponential as a function in $n$.
\end{Prop}
\begin{proof} By Theorem \ref{epNongn}, if $\frV$ contains a non-nilpotent finite algebra $A$, then the order of the $\frV$-free algebra $F_n$ of rank $n$ is at least $cq^n$ where $c>0$ and $q=|{\bf F}|$. However the sections of a chief series of $F_n$ have bounded dimensions  by Lemma \ref{esMinRep}. Therefore the number of the ideals in the chief series is at least $c'q^n$ for some $c'>0$, and we have at least $c'q^n$ $n$-generated factor-algebras of $F_n$ with different dimensions.

If  $\frV$ contains a finite nilpotent algebra $A$ with non-zero multiplication, then it contains
the free algebras $F_n$ of the nilpotent variety $\var A$ of class $c\ge 2$. In \cite[Equation 20]{BOGP}), the dimension of the power $F_n^c$ was bounded from below by $kn^c$ for some $k>0$, and by (\cite[Equation (50)]{BOGP} the number of pairwise nonisomorphic factor-algebras of $F_n$ was bounded from below by 
\[
q^{\frac{(\dim F^c_n)^2}{4} - n^2 -1}.
\]
 Since $c\ge 2$, this number
is bounded from below by $q^{\kappa n^4}$, where $\kappa>0$.

\end{proof}

\begin{Remark} In general case, the simple exponential estimate of the Proposition \ref{pNumber}
cannot be improved. By Corollary \ref{minalg}, all $n$-generated algebras in the variety
$\var A$, where $A$ is a finite minimal   algebra, are direct powers $A^{t(n)}$ of $A$ where by Birkhoff's theorem  $t(n)$ is bounded by an exponential function of $n$. So one gets the same upper bound  for the number of non-isomorphic $n$-generated algebras in $\var A$.
\end{Remark}

\begin{Remark} Recall that a variety of algebras $\frV$ is called Schreier if all subalgebras
of $\frV$-free algebras are $\frV$-free themselves. For a survey and the bibliography of results on Schreier varieties of algebras see \cite{MSY}. The simple estimates of Proposition \ref{pNumber} show that a locally finite variety of algebras with nontrivial multiplication cannnot be a Schreier variety.
Indeed, if it is locally nilpotent then the free algebras of rank $n\ge 2$ contains two-dimensional
subalgebras with zero multiplication. Otherwise the number of ideals of different dimensions
in $F_n(\frV )$ is at least $c'q^n$, while the number of non-isomorphic $\frV$-free subalgebras of ranks $k=1,\dots,n$ in it is just $n$.
\end{Remark}

We finish this section by a short discussion of the order of the automorphism groups of free algebras $F_n(\frV)$ of a locally finite variety $\frV$ of algebras over a finite field $\F$. These grow fast. To see how fast, we set $\alpha(n)=|\Aut F_n(\frV)|$. A rough estimate from below is given by the inequality   $\alpha(n)\ge q^{d(n-1)}$, where $q=|\F|$, because any map 
\[
x_1\mapsto x_1, \dots, x_{n-1} \mapsto x_{n-1}, x_n\mapsto x_n +a,\mbox{ where } a\in F_{n-1}=F(x_1,\dots,x_{n-1})
\]
extends to an automorphism of $F_n$.

Using Corollary \ref{2n1}, this argument implies the first claim in the following.
\begin{Remark}\label{clAutDouble}
If $\frV$ is a non-nilpotent locally finite variety then $\alpha(n)$ grows at least as fast as the double exponent. If $\frV=\var A$, for a finite algebra $A$, $\alpha(n)$ cannot grow faster than the double exponent.
\end{Remark}

The second claim follows from Birkhoff's estimate in  Theorem \ref{etBirkhoof_Estimate} because    the number of all the homomorphisms $F_n\to F_n$ is equal to $|F_n|^n = q^{nd(n)}$.
 
The claim of Remark \ref{clAutDouble} contrasts with the situation where  $\frV$ is nilpotent because in that case, the growth of  $\alpha(n)$ is bounded by a simple exponent. The precise formula is given in \cite[Section 8]{BOGP}: 
\[
\alpha(n) = |\GL_n({\bf F})|q^{n(d(n)-n)}
\] 
where, in this case, by Corollary \ref{polynom} the dimension $d(n)$ is a polynomial .

\section{Generic properties of finite algebras}\label{esGPNA}

In the next  subsections, we will single out several properties of finite algebras that hold in \emph{almost all} algebras. Such properties are called \emph{generic}. Besides, at the end of this section, we compare the rates of growth of the numbers of $m$-dimensional finite algebras, as $m\to\infty$, in three classes: all non-associative algebras, all nilpotent algebras, and all solvable algebras.

\subsection{Preliminary remarks.} 
 
Informally speaking, one says that a property of algebras of a class $\cal C$ is generic if
“almost every” object of $\cal C$ enjoys this property. The words “almost every” are understood differently, depending on the class $\cal C$. At the same time, it is always true that if several properties are generic then their finite intersection is also generic. Moreover, sometimes, it is easier to prove that some property is generic than to provide examples. For instance, we prove below that each of the following properties of finite non-associative algebras is generic: $A$ is simple, $A$ is a one-generated algebra, $A$ has no non-trivial automorphisms. It follows that a generic finite algebra enjoys all these characteristics.

Below we use the following formal definition of being general.
Let $\vp(m)$ be the number of non-isomorphic $m$-dimensional algebras
over a finite field $\bf F$. Given an abstract property $\cal P$,
we denote by $\vp_{\cal P}(m)$ the number of $m$-dimensional algebras over $\bf F$ with property $\cal P$. If 
\[
\lim_{m\to\infty} \frac{\vp_{\cal P}(m)}{\vp(m)} =1,
\]
 then we call $\cP$ {\it generic}.
Loosely speaking, one also can say that a generic algebra satisfies $\cal P$.

Recall that every $m$-dimensional algebra $A$ can be treated as
a tensor of type $(2,1)$ on an $m$-dimensional vector space $V$.
The corresponding trilinear function $T$ of type $(2,1)$ is
defined by $T(f,x,y)= f(xy)$, where $x,y \in V$,
$f\in V^*$, and $V^*$ is the dual space of $V$. With respect to
a basis $E=(e_1,\dots,e_m)$ of $V$, the tensor $T$
has $m^3$ coordinates $c_{ij}^l=e^l(e_ie_j)$, where $e^l$ belongs
to the dual basis $E^*$. The individual coordinates $c_{ij}^\ell$ are called \emph{structure constants}. The multiplication of the basis  vectors
in $A$ is given by the \emph{structural $m^3$-tuple} $C_E$:
\begin{equation}\label{eqCIJK}
C_E = (c_{ij}^l)_{i,j,l=1}^m\mbox{ so that } e_ie_j = \sum_{l=1}^m c_{ij}^l e_l, \;\; i,j =1,\dots,m
\end{equation}

Vice versa, these formulae uniquely define a bilinear multiplication
on $V$ for an arbitrary  $m^3$-tuple $C_E = (c_{ij}^l)_{i,j,l=1}^m$.

Let $\cal C$ be the set of all $m^3$-tuples $(c_{ij}^l)_{i,j,l=1}^m$.
We call two elements $(c_{ij}^l)_{i,j,l=1}^m$ and
$((c'_{ij})^l)_{i,j,l=1}^m$ of $\cal C$ {\it equivalent} if they 
are the structural $m^3$-tuples of two isomorphic
algebras with respect to some bases  $E$ and $E'$.
Since there is an obvious isomorphism between two algebras
with the same structural $m^3$-tuples, this relation is
transitive, and so it is an equivalence, indeed. By the same
reason, an equivalence class is equal to the set of all
structural $m^3$-tuples $C_E$ for a single
$m$-dimensional algebra $A$ with respect to the all bases $E$ of $A$.

Thus, the number of $m^3$-tuples in an equivalence class does not exceed the number of bases in an $m$-dimensional vector space, that is, bounded from above by $|\GL(V)|$, and there is a one-to-one
correspondence between the set of isomorphism classes of $m$-dimensional algebras and the set of equivalence classes in $\cC$.

We call a subset $\cal D$ of the set $\cal C$ {\it perfect} if it is the union of several equivalence classes.

\begin{Lemma}\label{perfect} 
Let  $\cal D$ be a subset of the set $\cal C$. Suppose every algebra of an abstract  class $\cal K$ has a basis $E$ such that the structural $m^3$-tuple $C_E$ belong to $\cal D$. Also suppose that every  $m^3$-tuple from $\cal D$ defines an algebra from $\cal K$. Then the following are true for the number $r$ of pairwise nonisomorphic algebras in $\cal K$
\begin{enumerate}
  \item[$\mathrm{(i)}$]  $\frac{|{\cal D}|}{|\GL_m({\bf F})|}\le r\le |\cal D|$;
  \item[$\mathrm{(ii)}$] if the subset $\cal D$ is perfect and every algebra in $\cal K$
has no non-trivial automorphisms, then $r = \frac{|{\cal D}|}{|\GL_m({\bf F})|}$.
\end{enumerate}
\end{Lemma}

\begin{proof} 
$\mathrm{(i)}$ The first statement follows from the correspondence
between the isomorphism classes of $m$-dimensional algebras over $\bf F$ and the equivalence classes of structural $m^3$-tuples, and the fact that
the size of every equivalence class is bounded by  $|\GL_m(\bf F)|$.

$\mathrm{(ii)}$ We note that every equivalence class contained in
$\cal D$ contains exactly  $|\GL_m(\bf F)|$ elements. Indeed,
otherwise an algebra $A\in \cal K$ has to have the same $m^3$-tuples
of structure constants with respect to different bases $E$ and $E'$
of $A$. Hence the mapping $E\to E'$ would extends to an automorphism
of $A$, a contradiction.
\end{proof}
For a field $\F$ of order $q$, the order of the group $\GL_m(\F)$ equals
\[
q^{m^2}\prod_{i=1}^m (1-q^{-i}).
\]
 Since the product $\prod_{i=1}^{\infty} (1-q^{-i})$ converges to a positive constant $d$ as fast as a geometric sequence, we have
\begin{equation}\label{eqGL}
|\GL_m(\F)| = q^{m^2-c+O(q^{-m})},
\end{equation}
where 
\begin{equation}\label{eqcd}
c = -\log_q d>0\mbox{ and }-\log_q \prod_{i=m+1}^{\infty} (1-q^{-i})= O(q^{-m})>0.
\end{equation}

To simplify the notation,
we write $\exp_q(N)$ in place of $q^{N}$, $N$ a real number.

\subsection{Generic algebras have no nontrivial automorphisms.}\label{ssSGP}

Our first theorem about the generic properties of finite algebras is the following.

\begin{Thm}\label{tGNonAut} Let $\vp(m)$ be the number of pairwise nonisomorphic algebras of dimension $m$ over a field $\F$ of order $q$ and $\psi(m)$ the number of those of them which admit a nontrivial automorphism. Using $c$ from (\ref{eqcd}), we have
\[
\vp(m) = \exp_q(m^3 - m^2 + c - O(q^{-m})).
\]
Also, as $m\to\infty$, $\frac{\psi(m)}{\vp(m)}$ converges to 0 exponentially fast.
\end{Thm}

\begin{proof}

Every automorphism $\beta$ of a finite-dimensional algebra $V$ over the field $\F$ is a linear map $V\to V$, which turns $V$ into a module over the polynomial algebra $\F[t]$. The least number of the generators of this module is called the \emph{rank} $r$ of the automorphism $\beta$. 

First, we estimate the number of structural $m^3$-tuples that define algebras in which $\beta$ is an automorphism.

\medskip

\emph{Step 1.} We fix a nonsingular linear operator $\beta$ of rank $r \le \frac{4m}{5}$. It follows from \cite[Lemma 10]{BOGP} that there are subspaces $U\subset W$ in $V$, such that  $s=\dim U \ge \frac{m}{15}$ and $V = W +\beta(U)$ is the direct sum. In this case, one can choose a basis $E=( e_1,\ld,e_m)$ in $V$, such that the first $s$ vectors $e_1,\ld,e_s$ form a basis in $U$, while the last $s$ vectors  $e_{m-s+1}=\beta(e_1),\ld,e_m=\beta(e_s)$ form a basis in $\beta(U)$.

If one knows the structure constants with respect to $E$ that define
all the products $e_ie_j$ except for $e_i,e_j\in \beta(U)$, then
the latter $ms^2$ constants are also  uniquely determined since
\[
 e_ie_j = \beta(e_k)\beta(e_l)=\beta(e_ke_l)
  \]
for $e_k,e_l\in U\subset W,$  and the right-hand side is known. Hence there
are at most $\exp_q(m^3-ms^2)\le \exp_q\left(m^3 -\frac{m^3}{225}\right)$ different $m^3$-tuples of structure constants defining in the basis $E$ those algebras where $\beta$ is an automorphism. The same upper bound
holds for the number of non-isomorphic algebras admitting this
automorphism $\beta$.
The number of linear operators $\beta$ is obviously bounded by
$\exp_q(m^2)$, and so Step 1 gives, by Lemma \ref{perfect} (i), less than $\exp_q \left(m^3 -\frac{m ^3}{225}+m^2\right)$ algebras admitting an automorphism
of rank $r \le \frac{4m}{5}$.

\medskip

\emph{Step 2}. We consider now a linear  operator $\beta$ on the space $V$ with the condition $r > \frac{4m}{5}$. By  \cite[Proposition 6]{BOGP} it follows, that $V$ contains an invariant subspace $U$ of dimension $s= 2r-m$, on which $\beta$ acts as a scalar operator, and $U$ has a $\beta$-invariant direct complement  in $V$. Hence the centralizer $C(\beta)$ in $\GL(V)$ contains a subgroup, isomorphic to $\GL(U)$. As a result, the index of $C(\beta)$ satisfies

\begin{equation}\label{eqGLV}
[\GL(V): C(\beta)]\le \exp_q (m^2 -(2r-m)^2)= \exp_q(4r(m-r)).
\end{equation}

Let us estimate the number $N(r)$ of different linear operators $\beta$  with given rank $r > \frac{4m}{5}$. In the canonical basis, $\beta$ is a block-diagonal square matrix with one square block of size $m-s$ and one more, the scalar block $\lambda I_s$. Such matrices belong to the subspace of dimension $(m-s)^2+1=4(m-r)^2+1$. Considering equation (\ref{eqGLV}), we obtain 

\begin{equation}\label{Nr}
N(r)\le \exp_q(4(m-r)^2+1)\exp_q(4r(m-r))=\exp_q(4m(m-r)+1). 
\end{equation}

Applying \cite[Lemma 10]{BOGP} again to  such $\beta$, we find subspaces $Z$ and $W$ in $V$, such that $Z\subset W$,  $\dim W = \left[\frac{2m+r}{3}\right]$,  $V = W \oplus\beta(Z)$, and

\[l= \dim Z = m-\left[\frac{2m+r}{3}\right] \ge \frac{m-r}{3}.\]

Again, we choose a basis $E=(e_1,\ld,e_m)$ in $V$, such that its first $\ell$ vectors form a basis in $Z$, while the last are the images of the first ones under $\beta$, hence forming the basis in $\beta(Z)$. By \cite[Lemma 11]{BOGP}, there is an $r$-dimensional subspace $U'\subset V$, on which $\beta$ acts as a scalar multiplication by $\lambda$. Its intersection with $Z$ is trivial. Set $Y=U'\cap W$. By Grassmann's formula, the dimension of $Y$  is no less than $\left[\frac{2m+r}{3}\right]+ r -m$. Since $r>\frac{4m}{5}$, we have $\dim Y \ge\frac{11m}{15}-\frac{2}{3}$. We may assume that the basis of $Y$ is included in the basis $E$.

If $e_i\in Y$ and $e_j=\beta(e_k)\in \beta(Z)$, then under the assumption that
$\beta$ is an automorphism of an algebra on the space $V$, we
have

\[e_ie_j =\lambda^{-1}\beta(e_i)\beta(e_k) = \lambda^{-1}\beta(e_ie_k),\]
and so all such products $e_ie_j$ (and also products  $e_je_i$,
where $e_i\in Y$ and $e_j\in \beta(Z)$),
are determined by $\beta$ and by the products of other pairs of the basis vectors. Therefore to define an algebra with the automorphism
$\beta$ one needs at most 
\[
m^3 - 2m\left(\frac{11m}{15}-\frac{2}{3}\right)\frac{m-r}{3}= \frac{23m^3 + 22m^2r +20m^2 -20mr}{45}
\]
scalars. Hence we obtain at most
\[N(r)\exp_q\frac{23m^3 + 22m^2r+20m^2 -20mr}{45}\le\] \[\exp_q\left(\frac{23m^3 + 22m^2r+20m^2 -20mr}{45}+ 4(m(m-r)+1)\right)
\]
of non-isomorphic  $m$-dimensional algebras admitting an automorphism of rank $r>\frac{4m}{5}$.

The function in the exponent at the right side is linear in $r$, with coefficient  of $r$ being $22m^2-200m$. So it increases with $r$, if $m\ge 10$. It follows that, with $m$ sufficiently large, the function has maximum value on the interval $[0,m-3]$  when $r=m-3$. In this case the number of non-isomorphic algebras is bounded by  
\[
\exp_q\frac{23m^3 + 22m^2(m-3)+60 +540m +45}{45} = \exp_q\left(m^3 -\frac{22m^2}{15} +12m+\frac 73\right).
\]

We need better estimates for $r=m-1$ and $r=m-2$. In both cases, by choosing $Z$ and $W$, as before, we have $\dim Z =1$, and in $W$ there is a $\beta$-invariant subspace  $Y$ of dimension $t\ge m-3$, on which $\beta$ acts as a scalar. Also we choose the basis
$E$ as above.
If a basis vector $e_i$ belongs to $Y$ and $e_n =\beta(e_1)\in\beta(Z)$, then $e_ie_n =\beta(e_i)\beta(e_1) =\beta(e_ie_1)$. Hence $ m $ sructure constants for
$e_ie_n$ are determined by $\beta$ and by the constants
$c_{ij}^k$ where $i,j<n$. The same argument works for $e_je_i$. 
Therefore for each of the values $r=m-1$ and $r=m-2$
we obtain at most  $\exp_q(m^3-2m(m-3))$ non-isomorphic
algebras admitting the automorphism $\beta$. 

The sum $N(m-1)+N(m-2)$ is less than $\exp_q(8m+2)$ by (\ref{Nr}).
Hence the number of non-isomorphic
algebras admitting  automorphism of ranks $m-1$ or $m-2$ is less
than $\exp_q(m^3-2m(m-3)+8m+2)$.

Note that on the algebras with nonzero product the map $\lambda\,\id$ is an automorphism only if $\lambda=1$. Thus to obtains our cumulative estimate in the case $r\ge \frac{4m}{5}$, we need to sum up our estimates over $r$ from 0 to $m-1$ only. We will obtain, that the number of algebras, which  admit an automorphism of rank $r>\frac{4m}{5}$, is not greater than 
$\exp_q(m^3 - \frac{7m^2}{5}-1)$,
for sufficiently large $m$.

Combining this with the estimate of Case 1, we obtain that  there are less than 
\begin{equation}\label{75}
\exp_q\frac{m^3 - 7m^2}{5}
\end{equation}
 algebras, admitting a nontrivial automorphism. 

\medskip

\emph{Step 3}. Now we can complete the proof.  
Let $\cX$ be the set of all structural $m^3$-tuples corresponding to the algebras 
 admitting a non-trivial automorphism. It follows from the last
  estimate and Lemma \ref{perfect} (i) that the perfect set $\cal X$ 
 has cardinality less than 
 $\exp_q(m^2) \exp_q\left(m^3 - \frac{7m^2}{5}\right)=\exp_q\left(m^3 - \frac{2m^2}{5}\right)$.
 So the perfect complement $\cal Y$ of $\cX$ in $\cal C$ is
 of cardinality at least $\exp_q(m^3)- \exp_q\left(m^3 - \frac{2m^2}{5}\right)$ which is $\exp_q(m^3-O(\exp_q(-m^2)))$.
 
 By this estimate for $\cal Y$, Lemma \ref{perfect} (ii), and formula (\ref{eqGL}), the number of non-isomorphic algebras
 of dimension $m$ having no non-trivial automorphism is 
\[
\frac{\Card({\cal Y})}{|\GL_m({\bf F})|}= \exp_q(m^3-m^2+c + O(q^{-m})).
\]
 Since the number of algebras with non-trivial automorphisms was bounded by $\exp_q\left(m^3 \frac{- 7m^2}{5}\right)$, the formula for $\vp(m)$ has been proven. 
 
 Comparing the estimates from the last two sentences, we obtain the proof of the second 
 claim of the theorem.
 
 \end{proof}

\subsection{Generic algebras are simple.}\label{essSimplicity}  A rather surprising generic property of finite linear algebras is the contents of the following.

\begin{Thm}\label{tGAS} 
A finite generic algebra is simple.
\end{Thm}

\begin{proof}
Let $V$ be vector space with basis $E = ( e_1,\ld,e_d,\ld,e_m)$ and $U=\la e_1,\ld,e_d\ra$ a $d$-dimensional subspace of $V$. We denote by $\cT(U)$ the subset of $m^3$-tuples $(c_{ij}^l)_{i,j,l=1}^m$ from $\cal C$ such that $U$ is an ideal in the algebra defined by $(c_{ij}^l)_{i,j,l=1}^m$ with respect to the basis $E$.  

For $1\le i\le d$, $1\le j\le m$ we have $e_ie_j\in U$. Invoking (\ref{eqCIJK}), we have $c_{ij}^k=0$, as soon as  $k>d$. The number of such equations is $dm(m-d)$. Similar condition for the structure constants is given by the products $e_je_i$, where $i\le d$. We must subtract the number the same equations in these two systems,
which are given by $1\le i\le j\le d$. So the total number of such restrictions equals  
$2dm(m-d)- d^2(m-d) = d(m-d)(2m-d)$. So
there are at most 
\[
\exp_q(m^3-d(m-d)(2m-d))=\exp_q(m^3-d^3+3md^2-2m^2d)
\]
$m$-dimensional algebras having a $d$-dimension ideal.

The cubic polynomial $p(x)=  - x^3+3mx^2- 2m^2x+m^3$ has the
derivative $p'(x) = -3x^2+6mx-2m^2$ with one zero $x_0 = \left(1-\frac{1}{\sqrt 3}\right)m$ on the segment $[1, m-1] $, and
$p(x_0)\le \frac{4m^3}{5}$. The values of $p(x)$ at the ends
of this segment are equal to $m^3 - 1 +3m-2m^2$ and
$m^3 -(m-1)^3+3m(m-1)2-2m^2(m-1)$. Both these values
are less than $m^3-2m^2 +3m$ for sufficiently large $m$,
which, consequently,  bounds from above the maximum of $p(x)$ on $[1,m-1]$, for all sufficiently large values of $m$..

It follows that the number of $m$-dimensional algebras
containing a $d$-dimensional ideal is less than 
$\exp_q(m^3-2m^2+3m)$ for all sufficiently large $m$.
Summing over $d=1,\ld,m-1$, we see that the number
of nonisomorphic $m$-dimensional algebras containing a proper
ideal is less than $m\exp_q(m^3-2m^2+3m)$. This
number is exponentially negligible in comparison with
the number of all $m$-dimensional algebras given by
Theorem \ref{tGNonAut}, as required.
\end{proof}

\subsection{Generic algebras are cyclic.}\label{ssyclic}

It is not clear, whether or not a generic finite-dimensional algebra  over a finite field has one-dimensional subalgebras.  The estimates given below are sufficient for the following weaker claim. We say that an algebra $A$ is \emph{cyclic}, if it can be generated by one element.

\begin{Thm}\label{tcyclic} 
Generic finite-dimensional algebras over a finite field are cyclic
and contain no proper subalgebras of dimension greater than 1. 
\end{Thm}
\begin{proof} The proof is split into several cases. 

In Case 1, we show that $\exp_q(\frac{m^3}{2} +m^2)$  bounds from above the number of algebras in which all  1-dimensional subspaces are subalgebras.

In Case 2, we show that $m\exp_q(m^3-m^2-2m)$ bounds from above the number of nonisomorphic algebras, having subalgebras of dimensions from 2 to $\frac{m}{2}$.

In Case 3, we show that  $m\exp_q(m^3 -\frac{3}{2}m^2)$ bounds from above the number of nonisomorphic algebras, having subalgebras of dimensions $\ell$, where $\frac{m}{2}<\ell\le m-1$.

Adding all above three number, we obtain, for large enough $m$, less than $q^{m^3-m^2 -\frac{3m}{2}}$ nonisomorphic $m$-dimensional algebras.  In the remaining cases, there is a one-dimensional subspace $\la a\ra$, which is not a subalgebra. In this case,  $a$  generates a subalgebra of  dimension $\ell\ge 2$. But such  subalgebras do not exist, unless $\ell=m$.
  
 Since by Theorem \ref{tGNonAut}, the number of all nonisomorphic algebras of dimension $m$ is greater than  $q^{m^3-m^2}$, we derive that the share of algebras, which cannot be generated by one element is less than $q^{-m}$, for all large enough $m$. Similarly,
 the estimates of Case 2 and Case 3 prove that a generic algebra
 has no proper subalgebras of dimension greater than $1$. So the statement of the theorem follows from
 the estimates of these three cases.
 
 Now we treat the three cases, one by one.
 
 \medskip
 
\emph{ Case} 1. We consider $m$-dimensional algebras, in which all  one-dimensional subspaces are subalgebras. Fix a basis $E=(e_1,\dots,e_m)$. 

For any $e_i\in E$, $e_i^2=c_{ii}^ie_i$, $i=1,\ld,m$. Now if $1\le i,j\le m$, then there is $\sigma_{ij}\in\F$ such that $(e_i+e_j)^2=\sigma_{ij}(e_i+e_j)$. So 
 \[
 \sigma_{ij}(e_i+e_j)=(e_i+e_j)^2=e_i^2+e_j^2+e_ie_j+e_je_i = c_{ii}^ie_i+c_{jj}^je_j +e_ie_j+e_je_i
 \]

 Therefore the product $e_je_i$ is determined by the product $e_ie_j$ and the
constants $c_{ii}^i, c_{jj}^j, \sigma_{ij}$. Hence the multiplication is uniquely
defined by the constants $c_{ij}^k$ for $i\le j$ and the constants $\sigma_{ij}$
with $i<j$. The total number of such constants is 
\[
\frac{m^2(m+1)}{2} +\frac{m(m-1)}{2} =
\frac{m^3}{2} + \frac{m^2 -m}{2}.
\]
 So the number of pairwise non-isomorphic algebras in Case 1
is less than\\ $\exp_q\left(\frac{m^3}{2} +m^2\right)$, as required.
 
\medskip  

\emph{ Case }2. Now we estimate, for each $2\le\ell\le\frac{m}{2}$, the number of nonisomorphic  algebras $A$ on a vector space $V$, having an $\ell$-dimensional subalgebra. For this goal, we fix a basis $E=(e_1,\dots,e_m)$ in $V$ and
estimate the cardinality of the set of structural $m^3$-tuples $\cal D$, such that
the algebra $A$ given in the basis $E$ by a structural $m^3$-tuple $(c_{ij}^{k})_{i,j,k=1}^m \in \cal D$, has an $\ell$-dimensional subalgebra.
Then $\cal D$ is a perfect set since it does not depend on the
choice of bases.

Let $U$ be an $\ell$-dimensional subspace of $V$. Choose an auxiliary basis  $F=(f_1,\ld,f_m)$ of $V$ such that $(f_1,\ld,f_{\ell})$ is a basis of $U$. Denote by ${\cal D}_F$
the set of $m^3$-tuples $C'=\{(c')_{ij}^{k}\}$ such that the algebra
defined by $C'$ with respect to $F$ contains the linear space $U$ as a subalgebra.

A structural $m^3$-tuple $\{(c')_{ij}^{k}\}$ belongs to ${\cal D}_F$ if and only if
it satisfies the conditions $(c')_{ij}^{k}=0$ for $i,j\le \ell$ and $k>\ell$. Altogether, we have $\ell^2(m-\ell)$ such equalities. It follows that every $m^3$-tuple from the set  ${\cal D}_F$ is defined by $m^3-\ell^2(m-\ell)$
constants, and so the cardinality of ${\cal D}_F$ is at most \[
\exp_q(m^3-\ell^2(m-\ell)).
\]
 Rewriting every $m^3$-tuple from  ${\cal D}_F$
in the basis $E$, we obtain a set ${\cal D}_U$ of $m^3$-tuples $C=\{c_{ij}^{k}\}_{i,j,k=1}^m $ such that the space $U$ becomes a subalgebra
in the algebra $A$ given in the basis $E$ by $C$. Obviously, $\cal D$
is the union of the sets ${\cal D}_U$ over all $\ell$-dimensional
subspaces $U$ of $V$.

The summation over all $\ell$-dimensional subspaces, whose number is less than   $q^{m\ell}$, provides us with less than $\exp_q(m^3-\ell^2(m-\ell)+m\ell)$ elements in
${\cal D}_U$. Let $\cal D'$ (resp. $\cal D''$) be the subsets of $\cal D$
defining algebras with (resp, without) non-trivial automorphisms.

The number of algebras defined by the $m^3$-tuples from $\cal D'$
is less than $\exp_q\left(m^3-\frac{7m^2}{5}\right)$ for sufficiently large $m$ by formula
(\ref{75}). By Lemma \ref{perfect}(ii) and formula (\ref{eqGL}), the number of algebras defined by $\cal D''$ is equal to \[
\frac{|{\cal D}''|}{|\GL(V)|}\le \frac{|{\cal D}|}{|\GL(V)|}\le \exp_q(m^3-\ell^2(m-\ell)+m\ell-m^2+m)
\]
for large enough $m$. Thus, for large enough $m$, there are less than \[\exp_q\left(m^3-\frac{7m^2}{5}\right)+\exp_q(m^3-\ell^2(m-\ell)+m\ell-m^2+m)< \exp_q(m^3-m^2 -2m)\] 
non-isomorphic algebras having an $\ell$-dimension subalgebra, because $\ell\in \left[2,\frac{m}{2}\right]$. The summation over $\ell$ from this segment gives us the upper bound
$m\exp_q(m^3-m^2 -2m)$, as desired.

\medskip

 \emph{Case} 3. We estimate the number of $m$-dimensional algebras with subalgebras of dimension $\ell$, where $m>\ell > \frac{m}{2}$. Such an $m$-dimensional algebra  $A$ can be obtained from an $\ell$-dimensional subalgebra $B$ by adding  $m-\ell$ basis vectors and taking all structure constants  $c_{ij}^k$, except those where $i,j\le\ell$. Altogether, it is sufficient to have $m(m^2-\ell^2)$  additional constants. 
 
 By Theorem \ref{tGNonAut}, there exists less than  $\exp_q(\ell^3-\ell^2 + const)$ nonisomorphic $\ell$-dimensional algebras. It follows, that there exists at most $\exp_q(m(m^2-\ell^2)+\ell^3-\ell^2+const)$ nonisomorphic algebras, having $\ell$-dimensional subalgebra. This  value is less than $\exp_q(m^3 -\frac{3}{2}m^2)$ for all sufficiently large $m$  and  $\frac{m}{2}<\ell\le m-1$. Taking summation over all such $\ell$, we obtain less than $m\exp_q\left(m^3 -\frac{3}{2}m^2\right)$ pairwise nonisomorphic algebras, as desired.

Now the proof is complete.
  
 \end{proof}
 The proof provides no information on one-dimensional subalgebras. So we raise the following
 \begin{problem}\label{pbmNumberWithoutProper} Let $\omega(m)$ be the number of $m$-dimensional
pairwise nonisomorphic algebras over a field of order $q$, without proper subalgebras. Is it true that \[
\limsup_{m\to\infty}\frac{\omega(m)}{q^{m^3-m^2}} >0?
\]
\end{problem}

 \subsection{Quasi-identities of generic algebras.}\label{essQ} 
 
 Given an absolute free  algebra ${\cal F}= {\cal F}(x_1,x_2,\dots)$  over a field $\bf F$ and nonassociative polynomials
 \[
 u_1(x_1,\ld,x_n),\ld, u_m(x_1,\ld,x_n),v(x_1,\ld,x_n)\in {\cal F},
 \]
 the formula
 \begin{eqnarray}\label{eQuasi-identity} 
 &&(\forall x_1,\ld,x_n) u_1(x_1,\ld,x_n)=0\&\ld\& u_m(x_1,\ld,x_n)=0\nonumber\\
 &&\longrightarrow v(x_1,\ld,x_n)=0
 \end{eqnarray}
 is called a \emph{quasi-identity}. We say that an algebra $A$ over $\F$ satisfies (\ref{eQuasi-identity}) if for any $a_1,\ld,a_n\in A$ the equalities $u_1(a_1,\ld,a_n)=0,\ld, u_m(a_1,\ld,a_n)=0$ in $A$ imply $v(a_1,\ld,a_n)=0$.

Note that in (\ref{eQuasi-identity}), it is allowed that $m=0$, that is the preamble of the statement can be empty. In that case, (\ref{eQuasi-identity}) is simply an identical relation $v(x_1,\ld,x_n)=0$. So if an algebra $B$ satisfies all the quasi-identities of an algebra $A$ then it also satisfies all the identities of $A$. One denotes by $\qvar A$ the \emph{quasi-variety generated by $A$}, that is, the class of all algebras over $\F$ satisfying all quasi-identities satisfied by $A$. One has $\qvar A\subset\var A$.
 
 A very simple example of a quasi-identity: $x^2=x\longrightarrow x=0$ is satisfied in any algebra without nonzero idempotents.
 
It follows from the definition, that the Cartesian products and subalgebras of algebras satisfying some set of quasi-identities also satisfy them. So by Birkhoff's Theorem \ref{eBirkhoff}, free algebras of $\var A$ are in $\qvar A$
for any algebra $A$. For a finite algebra $A$, the structure of the quasi-variety $\qvar A$ is very simple: any $C\in\qvar A$, not necessarily finite, is a subalgebra of a Cartesian power of $A$ (see  \cite[Corollary 9 in \S 11]{ASM}). So these two operations are sufficient to obtain any algebra of $\var A$, if
$\var A = \qvar A$ for a finite algebra $A$; no forming of the homomorphic images
is needed. Surprisingly, we have

 \begin{Thm}\label{epQuasi} The equality $\var A = \qvar A$ is a generic property of
 algebras $A$  of dimension $m$ over a finite field $\F$. 
\end{Thm}

\begin{proof} 

By Theorem \ref{tcyclic}, a finite generic algebra $A$ has no proper
subalgebras except for one-dimensional ones, and $A$ is simple by Theorem \ref{tGAS}.
Therefore $A$ is an inherently semisimple algebra. By Theorem \ref{epSSSA},
any finite algebra $B$ from the variety $\var A$ is a direct product of
algebras isomorphic to subalgebras of $A$, whence $B\in \qvar A$. Since
the variety $\var A$ is locally finite, it follows that every algebra
from $\var A$ also satisfies all the quasi-identities of  $A$.
Thus, we obtain the inclusion $\var A\subset \qvar A$. Since the opposite
inclusion is trivial, the theorem is proved. \end{proof}

\subsection{The numbers of nilpotent and solvable algebras.}\label{sssNA} 

Recall that $\vp(m)$ denotes the number of non-isomorphic $m$-dimensional algebras
over a field $\bf F$ of order $q$. Let $\nu(m)$ (let $\sigma(m)$) be the mumber
of nilpotent (resp., of solvable) algebras among them. It follows from the results
below that
\[ 
\lim_{m\to\infty} \frac{\log\nu(m)}{\log \vp(m)} =\frac 13\mbox{ \ and\  }
\lim_{m\to\infty} \frac{\log\sigma(m)}{\log\vp(m)} =\frac 23.
\]

\begin{Prop}\label{epNNA} 
The number $\nu(m)$ of nilpotent algebras of dimension $m$ equals $\exp_q\left(\frac{m^3}{3} + O(m^2)\right)$.
\end{Prop} 
\begin{proof}
 In a nilpotent algebra $A$ of dimesion $m$, there is a  series of ideals $\{ 0\}=I_0\subset I_1\subset \ld\subset I_m = A$ with factors of dimension 1 and conditions $A I_j \subset  I_{j-1}$ and $I_j A \subset  I_{j-1}$ for $j=1,\ld,m$. And, conversely, an algebra $A$ with such series is nilpotent (see Section \ref{ssEDN}).

As a result, $A$ is an algebra with a basis $E=\{e_1,\ld,e_m\}$ such that for any  $j=1,\ld,m$ the products $e_je_\ell$ and $e_le_j$ are linear combinations of basis vectors with indexes $<j$. We denote by $\mu(n)$ the number of collections of structure constants, under which these conditions hold with respect to the  basis $E$.

For the algebra $A/I_1$ we have a basis $\{e_2+I_1,\ld,e_m+I_1\}$ with analogous conditions, which yields $\mu(m-1)$ structural $m^3$-tuples with respect to this basis, that is, where the structure constants $c_{ij}^k$ without indexes equal 1. To obtain all the constants for $A$, we can complement them by any constants of the form $c_{ij}^1$, where $i,j>1$ ( $c_{1j}^\ell$ and $c_{i1}^\ell$ are zero). These additional constants can be chosen in $\exp_q((m-1)^2)$ ways. It follows that  $\mu(m) =\exp_q((m-1)^2)\mu(m-1)$, hence by induction  $\mu(m) = \exp_q(\lambda(m))$, where $\lambda(m) = \frac{(m-1)m(2m-1)}{6}$. This value is  $\frac{m^3}{3}+O(m^2)$. By Lemma \ref{perfect} (i) and the equality (\ref{eqGL}), $\nu(m)$ can also be written as $\exp_q\left(\frac{m^3}{3}+O(m^2)\right)$.

\end{proof}

\begin{Remark} In the classical varieties, the number of non-isomorphic nilpotent  algebras of dimension $m$ also grows fast,  whereas,  the simple algebras  of dimension $m$  are scarce. But in the case of arbitrary linear algebras,  the inspection of Theorem \ref{tGAS} and Proposition \ref{epNNA} reveals that, on the contrary, the number of simple algebras, grows approximately as the cube of the number of nilpotent algebras!
\end{Remark}

\begin{Prop}\label{epNSA} 
The number $\sigma(m)$ of solvable algebras of dimension $m$ is equals to $\exp_q\left(\frac{2m^3}{3} + O(m^2)\right)$.
\end{Prop}

\begin{proof}
  In a solvable algebra $A$ of dimension $m$, there is a  composition series of subideals $\{ 0\}=I_0\subset I_1\subset \ld\subset I_m = A$ with all factors of dimension 1 and conditions $I_j^2 \subset  I_{j-1}$  for $j=1,\ld,m$. And, conversely, an algebra $A$ with such series is solvable.

As a result, $A$ is an algebra, having a basis $E=\{e_1,\ld,e_m\}$ such that, for any  $j=1,\ld,m$, $\ell\le j$, the products $e_je_\ell$ and $e_le_j$ are linear combinations of basis vectors with indexes $<j$. We denote by $\rho(n)$ the number of structural $m^3$-tuples, satisfying these conditions, with respect to the basis $E$.

For the basis $E'=(e_1,\ld,e_{m-1})$ of $I_{m-1}$, there are $\rho(m-1)$  structural $m^3$-tuples, with the  given property. Additional constants $c_{mj}^\ell$ and $c_{im}^\ell$ must satisfy just the conditions $c_{mj}^m = 0$ and $c_{im}^m =0$. Hence we obtain $\exp_q(2m^2-3m +1)$ possibilities for the choices of additional constants, that is, $\rho(m) =\exp_q(2m^2-3m+1)\rho(m-1)$ In this case, by induction, $\rho(m)= \exp_q\left(\frac{2m^3}{3} +O(m^2)\right)$.

By Lemma \ref{perfect} (i) and the equality (\ref{eqGL}), $\sigma(m)$ can also be written as\\ $\exp_q\left(\frac{2m^3}{3}+O(m^2)\right)$.

\end{proof}

As a result, the number of nonisomorphic solvable algebras of dimension  $n$ is approximately equal to the square of the number of nonisomorphic nilpotent algebras of this dimension.

\end{document}